\documentclass[10pt]{amsart}
\usepackage{latexsym,amsmath,amssymb,graphicx,yhmath,enumitem}
\usepackage{float}
\usepackage[bf, small]{caption}
\usepackage[all,web]{xy}
\usepackage{color}
\usepackage{hyperref}

\usepackage{xcolor}

\def\xyellowspace{%
  \sbox0{\colorbox{yellow}{\strut\ }}
  \dimen0=\wd0\relax
  \hskip0pt\cleaders\box0\hskip\dimen0\hskip0pt}

\begingroup
\catcode`\ =\active%
\gdef\makeyellowspace{\let \xyellowspace\catcode`\ =\active}%
\endgroup

\def\?#1{\colorbox{yellow}{\strut#1}}

\DeclareFontFamily{OT1}{rsfs10}{}
\DeclareFontShape{OT1}{rsfs10}{m}{n}{ <-> rsfs10 }{}
\DeclareMathAlphabet{\mathscript}{OT1}{rsfs10}{m}{n}

\DeclareMathOperator{\im}{Im}       
\DeclareMathOperator{\Spec}{Spec}   
\DeclareMathOperator{\Hom}{Hom}     
\DeclareMathOperator{\Tors}{Tors}    
\DeclareMathOperator{\Pic}{Pic}     
\DeclareMathOperator{\Cl}{Cl}       
\DeclareMathOperator{\Cox}{Cox}     
\DeclareMathOperator{\Aut}{Aut}     
\DeclareMathOperator{\rk}{rk}       
\DeclareMathOperator{\coker}{coker} 
\DeclareMathOperator{\Vol}{Vol}     
\DeclareMathOperator{\Mov}{Mov}     
\DeclareMathOperator{\Nef}{Nef}     
\DeclareMathOperator{\Eff}{Eff}     
\DeclareMathOperator{\Relint}{Relint}  
\DeclareMathOperator{\Int}{Int}     
\DeclareMathOperator{\codim}{codim} 
\DeclareMathOperator{\diag}{diag}   
\DeclareMathOperator{\conv}{Conv}   
\DeclareMathOperator{\SNF}{SNF}     
\DeclareMathOperator{\mult}{mult}   
\DeclareMathOperator{\car}{char}    

\makeatletter
\def\widebreve{\mathpalette\wide@breve}
\def\wide@breve#1#2{\sbox\z@{$#1#2$}%
     \mathop{\vbox{\m@th\ialign{##\crcr
\kern0.08em\brevefill#1{0.8\wd\z@}\crcr\noalign{\nointerlineskip}%
                    $\hss#1#2\hss$\crcr}}}\limits}
\def\brevefill#1#2{$\m@th\sbox\tw@{$#1($}%
  \hss\resizebox{#2}{\wd\tw@}{\rotatebox[origin=c]{90}{\upshape(}}\hss$}
\makeatletter

\title[The multiplicity of a Mori Dream Space]
{The multiplicity of a Mori Dream Space:
applications to polar duality and the classification
of toric $\mathbb{Q}$-Fano varieties}

\author[M. Rossi]{Michele Rossi}

\date{\today}

\address{Dipartimento di Matematica, Universit\`a di Milano Bicocca,
Via Roberto Cozzi, 55, 20126 Milano} \email{michele.rossi@unimib.it}

\thanks{The author was partially supported by the I.N.D.A.M. as a member of the G.N.S.A.G.A. and by PRIN 2022L34E7W, Moduli spaces and birational geometry.\\
 Author's ORCID:0000-0001-6191-2087}

\def \a{\alpha }
\def \b{\beta }

\def\mm{\boldsymbol{\mu}}

\def \D{\Delta }
\def \Ga{\Gamma }
\def \Si{\Sigma }
\def \Th{\Theta }
\def \g{\gamma}
\def \vf{\varphi}

\def \ét{\'{e}tale}
\def \aa{\mathbf{a}}

\def \e{\mathbf{e}}
\def \q{\mathbf{q}}
\def \pp{\mathbf{p}}

\def \v{\mathbf{v}}
\def \n{\mathbf{n}}
\def \m{\mathbf{m}}
\def \w{\mathbf{w}}

\def \x{\mathbf{x}}

\def \1{\mathbf{1}}
\def \0{\mathbf{0}}

\def\P{{\mathbb{P}}}

\def\p2{\mathbb{P}^2}
\def\p3{\mathbb{P}^3}
\def\p4{\mathbb{P}^4}

\def\cO{\mathcal{O}}

\def\rk{\operatorname{rk}}

\def\GL{\operatorname{GL}}

\def\End{\operatorname{End}}

\def\Z{\mathbb{Z}}

\def\C{\mathbb{C}}
\def\K{\mathbb{K}}
\def\R{\mathbb{R}}

\def\Q{\mathbb{Q}}
\def\N{\mathbb{N}}

\def\T{\mathbb{T}}

\def\B{\mathcal{B}}
\def\L{\Lambda}

\def\X{\mathfrak{X}}

\def\irr{\mathcal{I}rr}

\def\CV{\mathcal{V}}

\def\SF{\mathcal{SF}}

\def\G{\mathcal{G}}

\def\I{\mathcal{I}}
\def\Ls{\mathcal{L}}
\def\Ga{\Gamma}
\def\De{\Delta}

\def\Weil{\mathcal{W}_{\T}}

\def\U1{\mathfrak{U}^{(1)}}

\theoremstyle{plain}
\newtheorem{theorem}{Theorem}[section]
\newtheorem{proposition}[theorem]{Proposition}
\newtheorem{prop-def}[theorem]{Proposition--Definition}

\newtheorem{thm-def}[theorem]{Theorem--Definition}
\newtheorem{corollary}[theorem]{Corollary}
\newtheorem{conjecture}[theorem]{Conjecture}
\newtheorem{lemma}[theorem]{Lemma}

\newtheorem*{a-proposition}{Proposition}

\theoremstyle{remark}
\newtheorem{remark}[theorem]{Remark}

\newtheorem{example}[theorem]{Example}

\theoremstyle{definition}
\newtheorem{definition}[theorem]{Definition}

\newtheorem*{step I}{Step I}
\newtheorem*{step II}{Step II}
\newtheorem*{step III}{Step III}
\newtheorem*{step IV}{Step IV}

\newtheorem*{acknowledgements}{Acknowledgements}

\newcommand{\oneline}{\vskip12pt}
\newcommand{\halfline}{\vskip6pt}
\newcommand{\etal}{\textit{et al.}}

\begin{document}


\begin{abstract}
We extend the multiplicity of fake weighted projective
spaces, as introduced by Averkov, Kasprzyk, Lehmann and Nill in 2021, 
to complete toric varieties and Mori Dream Spaces,
defining it as the order of the torsion subgroup of the
divisor class group. It equals the degree of the canonical
$1$-covering, a finite surjective morphism \'etale in
codimension one.
We classify toric $\Q$-Fano varieties with prescribed
weight matrix and Gorenstein index by admissible subgroups
of a finite abelian group, realized as the deck
transformation groups of their universal $1$-coverings.
The classification allows arbitrary class-group rank
without assuming $\Q$-factoriality, and yields an intrinsic
intermediate $1$-covering, the \emph{unitary} one, controlling the Gorenstein index.
We outline a computational classification strategy based
on weight matrices and finite subgroup enumeration.
In the Gorenstein Fano case, a perfect pairing of weight
groups expresses polar duality as an inclusion-reversing
subgroup correspondence.
For iso-weighted polar partners over $\C$, the mirror
correspondence admits an internal description within
the anticanonical family of their common universal
$1$-covering, through invariant subfamilies and their
quotient actions.
Finally, we derive divisibility relations and multiplicity
bounds, also applicable to Mori Dream Spaces whose
canonical toric ambient varieties admit suitable
$\Q$-Fano completions.
\end{abstract}

\maketitle

\tableofcontents

\section*{Introduction}

The multiplicity of a fake weighted projective space is the
degree of its universal $1$-covering by a weighted projective
space \cite{AKLN}. In this paper we extend this notion first to complete
toric varieties and then to Mori Dream Spaces (MDS).
The extension is defined by the torsion in the divisor class
group and has a geometric interpretation through canonical
$1$-coverings and canonical toric embeddings.
Throughout the paper, a $1$-covering means a finite
surjective morphism that is \'etale in codimension one.

Our main toric result classifies $\Q$-Fano varieties with a
fixed weight matrix and a prescribed ratio between their
Gorenstein index and that of their universal $1$-covering (see Theorem~\ref{s5:thm:classification}).
The classification has a topological interpretation:
the parametrizing subgroups act as the groups of deck
transformations of the associated universal $1$-coverings,
and therefore identify the \'etale fundamental groups in
codimension one of the quotients.
The same construction gives an intrinsic intermediate
$1$-covering controlling the change of Gorenstein index.
In the Gorenstein Fano case, polar duality reverses inclusions
between the corresponding subgroups (see Theorem~\ref{thm:classificazione}).
These descriptions connect multiplicity with weight matrices,
lattice volumes and anticanonical intersection numbers.

Throughout the paper, a \emph{Fano} variety has ample Cartier
anticanonical divisor, whereas a \emph{$\Q$-Fano} variety
has an ample Cartier multiple $-kK_X$.
The Gorenstein index is the least such positive integer $k$.
In the toric results, $\Q$-factoriality is assumed only
when explicitly stated, and \emph{rank} means
$r=\rk\Cl(X)$.

\subsection{Multiplicity beyond fake weighted projective spaces}

For a complete toric variety $X$ of dimension $n$ with fan
matrix $V$, we set
\[
\mult X:=|\Tors(\Cl(X))|
=[\Z^n:\Ls_c(V)],
\]
where $\Ls_c(V)$ is the lattice generated by the columns
of $V$. This is the degree of the universal $1$-covering
$\pi:Y\longrightarrow X$ and the order of its group of
deck transformations. The construction of $Y$ replaces the
ambient lattice by the sublattice generated by the primitive
ray generators; see \cite{RT-QUOT,R-CovMDS}.
Thus the definition recovers the usual multiplicity of
fake weighted projective spaces.

For a Mori Dream Space $X$ \cite{Hu-Keel}, we define
\[
\mult X:=|\Tors(\Cl(X))|,
\]
and interpret this integer as the degree of its canonical
$1$-covering \cite{R-CovMDS}; see
Definition~\ref{def:multMDS}.
Its canonical toric embedding $X\hookrightarrow W$
identifies the divisor class groups and hence their torsion.
If $W$ admits a sharp completion $Z$, which adds no rays,
then $\mult X=\mult Z$.
Consequently, the toric divisibility relations and bounds
apply when $Z$ satisfies the corresponding $\Q$-Fano
and singularity assumptions; see
Proposition~\ref{prop:MDSbound} and
Corollary~\ref{cor:MDSbounds}.
The dimension and index appearing in these bounds are
those of the ambient variety $Z$.
This applies, in particular, to $\Q$-factorial klt
$\Q$-Fano varieties whenever their toric ambient
completions satisfy these assumptions: such varieties
are Mori Dream Spaces by \cite[\S~1.3]{BCHMcK}.

\subsection{Classification with fixed weight data}

Fix a weight matrix $Q$ admitting a $\mathbb Q$-Fano toric
realization. Its integral Gale dual $W=\mathcal G(Q)$
determines the universal $1$-covering $Y$, whose fan consists
of the cones over the faces of $P=\operatorname{Conv}(W)$.
Let $\widehat k$ be the Gorenstein index of $Y$.
For a quotient $X$ of index $k$, the ratio
\[
h:=k/\widehat k
\]
is a positive integer, called its \emph{factor}.

We associate with $Q$ a finite abelian \emph{weight group}
$G_Q$, of order $g_Q$, and, for each $h\geq1$, a finite
group $\widehat G_Q^{\,h}$ of order $h^n g_Q$, called the \emph{$h$-extension} of $G_Q$, together
with an embedding
$\bar h:G_Q\hookrightarrow\widehat G_Q^{\,h}$.
Theorem~\ref{s5:thm:classification} describes the varieties
of weight matrix $Q$ and factor $h$ as quotients
\[
X\cong Y/H,\qquad
H\leq\widehat G_Q^{\,h},\qquad
\operatorname{mult}X=|H|.
\]
With the actions constructed in the proof, the quotient map
$\pi_H:Y\longrightarrow X$ is the universal $1$-covering
and
\[
H\cong\operatorname{Deck}(\pi_H)
\cong\pi_1^{\acute{e}t}(X)^{(1)}.
\]
Here $\operatorname{Deck}(\pi_H)=\Aut_X(Y)$ is the group
of automorphisms of $Y$ over $X$, and
$\pi_1^{\acute{e}t}(X)^{(1)}$ denotes the \'etale
fundamental group in codimension one.
Thus the subgroup parametrization records the groups of
deck transformations as well as the quotient varieties.
The admissible subgroups are characterized by two conditions:
their actions are free in codimension one and, setting
$H_1=H\cap\bar h(G_Q)$, they satisfy
\[
\exp(H/H_1)=h.
\]
The first condition preserves the primitive ray generators;
the second ensures that the factor is exactly $h$.
The subgroup parametrization retains the quotient map from
the fixed source $Y$. Toric isomorphism classes are obtained
by taking the orbits under the lattice automorphisms preserving
$P$.

This formulation works in arbitrary class group rank,
including the non-$\mathbb Q$-factorial case.
It organizes quotient actions, the exact index and
multiplicity within a single finite group for each pair
$(Q,h)$. In rank one, classification procedures for fake
weighted projective spaces at prescribed Gorenstein index
are already available \cite{HHHS,Bauerle25,Ghirlanda}.
The relation with these procedures and with the established
theory of toric $1$-coverings is discussed in
Remark~\ref{s5:rem:bibliographical-comparison}.

\subsection{The unitary $1$-covering}

The same subgroup data yield a distinguished factorization
\[
Y\longrightarrow X_1:=Y/H_1\longrightarrow X=Y/H.
\]
The variety $X_1$ has Gorenstein index $\widehat k$, and
the $1$-covering $X_1\longrightarrow X$ satisfies
\[
\deg(X_1\longrightarrow X)=|H/H_1|,
\qquad
h\mid\deg(X_1\longrightarrow X)\mid h^n.
\]
We call it the \emph{unitary $1$-covering}.
Its group of deck transformations is $H/H_1$; accordingly,
the factor $h$ is the exponent of this group, whereas
the degree of the $1$-covering is its order.
Its construction is intrinsic and universal: every intermediate
$1$-covering $X'\to X$ for which $X'$ has Gorenstein
index $\widehat k$ factors through $X_1\to X$. In particular, it is independent of the auxiliary
group identifications used to express the quotient actions.

Example~\ref{ex:QFanoCanonica2} gives a complete application
in rank two. For $Y=\mathbb P^2\times\mathbb P^1$ and
$h=2$, there are six admissible subgroups, forming two
orbits under the lattice automorphisms of the ray polytope.
The resulting varieties have multiplicities $2$ and $6$.
Their unitary $1$-coverings have sources $Y$ and its Gorenstein
quotient of multiplicity $3$, respectively.
The example thus determines all quotients for the specified
data and distinguishes the quotient maps from the resulting
toric isomorphism classes.

\subsection{The Gorenstein case and polar duality}

When $\widehat k=h=1$, every subgroup of $G_Q$ is
admissible. Let $Z^\circ$ be the universal $1$-covering
of the polar partner $Y^\circ$, and put $Z=(Z^\circ)^\circ$.
The variety $Z$ is the maximal Fano quotient of $Y$
having $Y$ as universal $1$-covering.
Theorem~\ref{thm:classificazione} describes all intermediate
$1$-coverings $Y\longrightarrow X\longrightarrow Z$ and
their polar partners through
\[
X\cong Y/H,\qquad
X^\circ\cong Z^\circ/H^\circ,\qquad
H^\circ\cong G_Q/H.
\]
The groups $H$ and $H^\circ$ act as the deck
transformation groups of the universal $1$-coverings
$Y\longrightarrow X$ and
$Z^\circ\longrightarrow X^\circ$, respectively.
They therefore identify
$\pi_1^{\acute{e}t}(X)^{(1)}$ and
$\pi_1^{\acute{e}t}(X^\circ)^{(1)}$.
In particular,
\[
\operatorname{mult}X\mid g_Q,
\qquad
\operatorname{mult}X\,
\operatorname{mult}(X^\circ)=g_Q.
\]
More precisely, the abstractly isomorphic weight groups
$G_Q\cong G_{Q^\circ}$ carry a perfect pairing.
Taking annihilators and using the compatible group
identifications in the theorem gives an inclusion-reversing
bijection
\[
\left\{H\leq G_Q\right\}
\longleftrightarrow
\left\{H^\circ\leq G_{Q^\circ}\right\},
\qquad H\longmapsto H^\circ,
\]
such that
\[
(Y/H)^\circ\cong Z^\circ/H^\circ,
\qquad
H^\circ\cong G_Q/H,
\]
Thus polar duality is expressed through a correspondence
between the groups of deck transformations.
The kernels in the exact sequences describing the $1$-coverings
are mutual annihilators under this pairing
(Proposition~\ref{prop:splittingGQ}).

This description is closely related to Batyrev's minimal
and maximal reflexive pairs \cite[\S~5.2--5.3]{Batyrev94}.
Our formulation expresses that lattice correspondence
through weight matrices, explicit cokernels and Cox quotient
data, and connects it with multiplicity and degree;
see Remark~\ref{rem:Batyrev}.

\subsection{The weight modulus and bounds for multiplicity}
\label{ssez:mod&deg}

The relation between multiplicity and anticanonical geometry
has been studied extensively for fake weighted projective
spaces. An initial motivation for this work was the following
sharp result of Averkov, Kasprzyk, Lehmann and Nill.

\begin{theorem}[{\cite[Thm.~1.1]{AKLN}}]\label{thm:AKLN}
Let $X$ be an $n$-dimensional fake weighted projective space
with at worst canonical singularities. Then:
\begin{enumerate}
\item[\textup{(i)}] If $n\leq3$, then
$\operatorname{mult}X\leq(n+1)^{n-1}$.
\item[\textup{(ii)}] If $n=4$, then
$\operatorname{mult}X\leq128$.
\item[\textup{(iii)}] If $n\geq5$, then
$\operatorname{mult}X\leq3(s_{n-1}-1)^2$, where
\[
s_1=2,\qquad
s_j=1+\prod_{i=1}^{j-1}s_i\quad(j\geq2)
\]
is the Sylvester sequence.
\end{enumerate}
For each dimension, the bound is attained by a unique
variety up to isomorphism.
\end{theorem}

In the Gorenstein case, Conrads' work gives the following
divisibility statement.

\begin{theorem}[{\cite[Prop.~5.5]{Conrads}}]\label{thm:Conrads}
Let $X$ be an $n$-dimensional Gorenstein fake weighted
projective space with universal $1$-covering
$Y=\mathbb P(Q)=\mathbb P(q_1,\ldots,q_{n+1})$,
where $Q$ is reduced and well-formed. Then
\[
\operatorname{mult}X\ \bigm|\
\frac{\left(\sum_{i=1}^{n+1}q_i\right)^{n-1}}
     {\prod_{i=1}^{n+1}q_i}
=\frac{(-K_Y)^n}{|Q|}\in\mathbb N,
\]
where $|Q|=\sum_{i=1}^{n+1}q_i$ and $-K_Y$ is Cartier.
\end{theorem}

For arbitrary rank, we define the \emph{weight modulus} by
\[
|Q|:=n!\operatorname{Vol}
\bigl(\operatorname{Conv}(\mathcal G(Q))\bigr),
\]
where a fundamental parallelepiped has volume one.
Proposition~\ref{prop:modulus} computes this invariant
using maximal minors and a subdivision compatible with
the facets of the ray polytope. In rank one, it recovers
the sum of the weights.
For a Fano toric variety, Theorem~\ref{thm:Fano-bounds}
gives
\[
\operatorname{mult}X=\frac{(-K_{X^\circ})^n}{|Q|},
\qquad
g_Q=\frac{(-K_Y)^n}{|Q^\circ|}.
\]
Together with $\operatorname{mult}X\mid g_Q$, these
identities extend Conrads' divisibility statement to
arbitrary rank.

For a $\mathbb Q$-Fano toric variety of index $k$ and
factor $h$, Theorem~\ref{thm:degree} and
Corollary~\ref{cor:divisibility} give
\[
\operatorname{mult}X\,
\frac{(-kK_X)^n}{|Q^\circ|}=h^n g_Q,
\qquad
\operatorname{mult}X\mid h^n g_Q.
\]
Here $Q^\circ$ is obtained from the rational polar polytope;
the construction does not require a polar toric variety
$X^\circ$ to exist. The intersection number in this formula
is taken for the Cartier divisor $-kK_X$.

The numerical bounds combine anticanonical estimates
\cite{Kasprzyk10,BKN} with the lattice volume inequality
\[
n!\operatorname{Vol}(P)
\geq\max\{f_0(P),f_{n-1}(P)\}
\]
of Lemma~\ref{lem:volumeverticesfacets}, for a lattice
polytope containing the origin in its interior.
Writing $r'=\max(r,r^\circ)$, where $n+r^\circ$ is the
number of vertices of the anticanonical polytope,
Theorem~\ref{thm:canonical-bound} gives, under the
canonical singularity assumption,
\[
\operatorname{mult}X\leq
\begin{cases}
\left[\dfrac{9}{2+r'}\right],&n=2,\\[6pt]
\left[\dfrac{72k^3}{3+r'}\right],&n=3,\\[6pt]
\left[\dfrac{2(s_n-1)^2k^n}{n+r'}\right],
&n\geq4.
\end{cases}
\]
Canonical toric surfaces have index one, and the Fano
bounds of Theorem~\ref{thm:bounds} correspond to $k=1$.
These estimates record the dependence on rank and polar
combinatorics; they are not generally sharp.
For fake weighted projective spaces,
Proposition~\ref{thm:fake-wps} applies B\"auerle's degree
and volume bounds \cite{Bauerle,Bauerle25}, retaining
$\widehat k$ and $|Q|$ without assuming canonical
singularities. Remark~\ref{rem:degree-conjecture}
compares these estimates with the direct multiplicity
bounds of \cite[Thm.~1.2]{Bauerle25}.

\subsection{Computational and geometric perspectives}

The subgroup parametrization suggests a computational
approach starting from weight matrices.
Section~\ref{ssez:WMA} outlines this \emph{Weight Matrices
Analysis} (WMA), which complements the polytope approach
underlying PALP \cite{KS,PALP}.
For fixed admissible weight data, the quotient construction
is finite. An effective procedure for generating the weight
matrices themselves remains to be developed.

Finally, Gorenstein fake weighted projective spaces and
their polar partners have the same weights
\cite[Lemma~5.3]{Conrads}.
Although this fails in general in higher rank,
\emph{iso-weighted} polar partners still occur.
When $Q=Q^\circ$, the varieties $X$ and $X^\circ$
share a universal $1$-covering $Y$, and the maximal
Fano quotient of $Y$ is its polar partner:
\[
Y^\circ\cong Y/G_Q.
\]
Over $\C$, this gives an interpretation of their mirror
correspondence as an \emph{internal duality within a
single anticanonical family}.
Indeed, pullback identifies the anticanonical families
of $X$ and $X^\circ$ with invariant subfamilies of
$|-K_Y|$, equipped with the actions of their respective
deck transformation groups $G$ and $G^\circ$.
The subgroup correspondence pairs these equivariant
subfamilies, whose quotients recover the two mirror
families. Thus the ambient variety and its anticanonical
family remain fixed throughout this description.
Section~\ref{ssez:isoweighted} expresses this aspect of
Batyrev's construction
\cite[Thm.~5.3.1 and Prop.~5.3.7]{Batyrev94}
through the common weight matrix and its weight group,
and raises the question of finding effective criteria
for iso-weighted polar partners.

The paper is organized as follows.
Section~\ref{sez:preliminari} fixes notation for toric
varieties, integral Gale duality and the GKZ decomposition.
Section~\ref{sez:bounds} introduces multiplicity, weight
groups and weight moduli, and proves the Fano degree
identities and bounds.
Section~\ref{sez:Fano} gives the subgroup classification
of Fano toric varieties and their polar partners.
Section~\ref{sez:QFano} develops the numerical results
for $\mathbb Q$-Fano varieties, and
Section~\ref{sez:Qclassification} proves the corresponding
classification theorem and constructs unitary $1$-coverings.
Section~\ref{sez:MDS} treats Mori Dream Spaces through
their canonical toric embeddings.
Section~\ref{sez:problemi} discusses WMA and iso-weighted
polar partners.

\begin{acknowledgements}
I thank Lea Terracini for her suggestions during the initial
stages of this work. Our attempts to prove a conjecture that
turned out to be false led to the main ideas developed here.
I also thank Cinzia Casagrande for drawing my attention to her
work \cite{Casagrande06} on the relation between dimension and
rank for $\Q$-factorial Gorenstein Fano toric varieties.

Many of the results of this paper were announced in a seminar
at the Stefan Banach International Mathematical Center in
Warsaw in May 2024. I thank J.~Buczy\'nski for the invitation
and the warm hospitality at IMPAN, and for pointing out two
errors in an earlier version of this paper.

Figures~\ref{fig:mov} and~\ref{fig:politopo} were produced
with \texttt{GeoGebra} \cite{GeoGebra}.
\end{acknowledgements}

\section{Preliminaries and notation}\label{sez:preliminari}

In the following we will work over an algebraically closed field $\K=\overline{\K}$,  with $\car \K =0$\,. Sometimes we will assume $\K$ to be the field $\C$ of complex numbers.

\subsection{Notation on toric varieties}
An \emph{$n$-dimensional toric variety} is a normal algebraic
variety $X$ containing a torus $\T\cong(\K^*)^n$ as a
Zariski open subset, such that the multiplication action of
$\T$ on itself extends to an action $\T\times X\longrightarrow X$.

We denote by $M$ the \emph{character lattice} of $\T$
and by $N$ its lattice of \emph{one-parameter subgroups}.
Thus characters are homomorphisms
$\chi:\T\longrightarrow\K^*$ and one-parameter subgroups
are homomorphisms $\lambda:\K^*\longrightarrow\T$ of
algebraic groups. We identify
\[
\T=\Hom(M,\K^*)
\]
on $\K$-points and use $\T$ for the acting torus
of the toric variety under consideration.
The lattices $M$ and $N$ are dual of rank $n$, via the pairing
\begin{equation*}
\begin{array}{ccc}
M\times N & \longrightarrow & \Hom_{\mathrm{alg.grp}}(\K^*,\K^*)\cong\Z\\
 \left( \chi,\lambda \right) & \longmapsto
& \chi\circ\lambda
\end{array}
\end{equation*}
For standard notation on toric varieties and their defining \emph{fans} we refer to the extensive treatment \cite{CLS}.

\subsubsection{List of notation}\label{sssez:lista}
\begin{eqnarray*}
  &M_{\R},N_{\R}& \text{denote the tensor products of lattices $M$ and $N$ with $\R$, respectively;} \\
  &\Si(i)& \text{is the \emph{$i$--skeleton of a fan $\Si$};}\\
  &\langle\v_1,\ldots,\v_s\rangle\subseteq L_{\R}& \text{is the cone generated by $\v_1,\ldots,\v_s\in L_{\R}$, where $L=N,M$;}\\
  && \text{if $s=1$ this cone is called the \emph{ray} generated by $\v_1$;} \\
  &\mathcal{L}(\v_1,\ldots,\v_s)\subseteq L& \text{is the sublattice spanned by $\v_1,\ldots,\v_s\in L$\,.}\\
\end{eqnarray*}

\subsection{Notation on matrices}\label{ssez:Not-matrici}
We will consider matrices with  either integer entries in $n$ rows and $m$ columns, that is $A\in\Z^{n,m}$, or rational entries, that is $A\in\Q^{n,m}$.

Given $A\in\Z^{n,m}$,
\begin{eqnarray*}
  &\mathcal{L}_r(A)\subseteq\Z^m& \text{is the sublattice spanned by the rows of $A$;} \\
  &\mathcal{L}_c(A)\subseteq\Z^n& \text{is the sublattice spanned by the columns of $A$.} \\
  \end{eqnarray*}

  Given $A\in\Q^{n,m}$,
\begin{eqnarray*}
&A_I\,,\,A^I& \text{$\forall\,I\subseteq\{1,\ldots,m\}$ the former is the submatrix of $A$ given by}\\
  && \text{the columns indexed by $I$ and the latter is the submatrix}\\
  && \text{of $A$ whose columns are indexed by the complementary }\\
  && \text{subset $\{1,\ldots,m\}\backslash I$;} \\
  &\text{\emph{positive}}& \text{means a matrix $A$ (vector $\v$) whose entries are non-negative,}\\
  && \text{we will also write either $A\ge\0$ or $\v\ge\0$;}\\
 & A^*:=(A\cdot A^T)^{-1}\cdot A& \text{is the \emph{transverse} matrix of $A$ with $\rk A=n$:}\\
  &&\text{if $m=n$ then}\ A^*=(A^T)^{-1}
\end{eqnarray*}

\subsubsection{Matrices, cones and polytopes}
Given an integer matrix
$$V=(\v_1,\ldots,\v_{m})\in\Z^{n,m}$$
\begin{eqnarray*}
  &\langle V\rangle=\langle\v_1,\ldots,\v_{m}\rangle\subseteq N_{\R}& \text{is the cone generated by the columns of $V$;} \\
  &\conv(V)& \text{is the convex hull of the columns of $V$,}\\
  && \text{treated as lattice points.}\\
  &\SF(V)=\SF(\v_1,\ldots,\v_{m})& \text{is the set of all simplicial fans $\Si$ such that}\\
  && \text{$\Sigma(1)=\{\langle\v_1\rangle,\ldots,\langle\v_{m}\rangle\}\subseteq N_{\R}$ and} \\ && \text{$|\Si|=\langle V\rangle$ \cite[Def.~1.3]{RT-LA&GD}.}\\
  & \I_\Si& :=\{I\subseteq\{1,\dots,m\}\,|\,\langle V^I\rangle\in\Si\}\\
  &\G(V)& \text{is a \emph{Gale dual} matrix of $V$ \cite[\S~3.1]{RT-LA&GD};} \\
  \end{eqnarray*}
\begin{definition}[Fan and weight matrices]\label{def:F,W-matrix}
Given a fan $\Si$  in $L_\R$, with $\rk L=n$, a \emph{fan matrix} of $\Si$ is a matrix $V\in\Z^{n,m}$ whose columns are given by all the primitive generators of rays in the 1-skeleton $\Si(1)$. Then $m=|\Si(1)|$ and $V$ is defined \emph{up to $\GL$-equivalence}, that is up to the left multiplication by a matrix in $\GL_n(\Z)$ and right multiplication by a permutation matrix in $\GL_m(\Z)$\,.

We will say that a fan $\Si$ is \emph{over} a matrix $V$ if $V$ is a fan matrix of $\Si$.

A Gale dual matrix $Q=\G(V)$ is called a \emph{weight matrix} of either $\Si$ or the toric variety $X(\Si)$ defined by $\Si$\,.
\end{definition}

\begin{definition}[Bunch of cones associated with a fan]
  Given a fan $\Si$  in $L_\R$, with $\rk L=n$, a fan matrix $V$ of $\Si$ and a weight matrix $Q=\G(V)$, the \emph{bunch of cones} associated with $\Si$ is the set of cones
  \begin{equation*}
    \B(\Si):=\{\langle Q_I\rangle\,|\,I\in\I_\Si\}=\{\langle Q_I\rangle\,|\langle V^I\rangle\in\Si\}
  \end{equation*}
  In a sense, the bunch of cones $\B(\Si)$ is \emph{a Gale dual entity} of the fan $\Si$.

  The converse also holds, that is, given a \emph{bunch of cones $\B$ over a matrix $Q$} (see \S~2.2.1 and in particular Def.~2.2.1.10 in \cite{ADHL}) the associated fan $\Si(\B)$ is defined as a fan over $W=\G(Q)$ as
  \begin{equation*}
    \Si(\B):=\{\langle W^I\rangle\,|\,\langle Q_I\rangle\in\B\}
  \end{equation*}
  Notice that, in general, if $Q=\G(V)$ then $W=\G(Q)$ is not $\GL$-equivalent to $V$ (see also the following Remark~\ref{rem:Univ-costruzione}).
\end{definition}

\subsubsection{$F$-matrices, $W$-matrices and $CF$-matrices}\label{sssez:matrici}
An \emph{$F$--matrix} \cite[Def.~3.10]{RT-LA&GD} is a matrix  $V\in \Z^{n,m}$ with integer entries, satisfying the following conditions:
\begin{itemize}
\item[(F.a)] $\rk(V)=n$;
\item[(F.b)] $V$ is \emph{$F$--complete} i.e. $\langle V\rangle=\R^n$ \cite[Def.~3.4]{RT-LA&GD};
\item[(F.c)] all the columns of $V$ are non zero;
\item[(F.d)] if ${\bf  v}$ is a column of $V$, then $V$ does not contain another column of the form $\lambda  {\bf  v}$ where $\lambda>0$ is a real number.
\end{itemize}
An $F$--matrix $V$ is called \emph{reduced} if every column of $V$ has coprime entries \cite[Def.~3.13]{RT-LA&GD}.
The most significant example of a reduced $F$-matrix is given by a fan matrix $V$ of a complete fan $\Si$.
\halfline

 A \emph{$W$--matrix} \cite[Def.~3.9]{RT-LA&GD} is a matrix $Q\in\Z^{r,m}$ satisfying the following conditions:
\begin{itemize}
\item[(W.a)] $\rk(Q)=r$;
\item[(W.b)] ${\mathcal L}_r(Q)$ does not have cotorsion in $\Z^{m}$;
\item[(W.c)] $Q$ is \emph{$W$--positive}, that is, $\mathcal{L}_r(Q)$ admits a basis consisting of positive vectors \cite[Def.~3.4]{RT-LA&GD}.
\item[(W.d)] Every column of $Q$ is non-zero.
\item[(W.e)] ${\mathcal L}_r(Q)$   does not contain vectors of the form $(0,\ldots,0,1,0,\ldots,0)$.
\item[(W.f)]  ${\mathcal L}_r(Q)$ does not contain vectors of the form $(0,a,0,\ldots,0,b,0,\ldots,0)$, with $ab<0$.
\end{itemize}
A $W$--matrix is called \emph{reduced} if $V=\G(Q)$ is a reduced $F$--matrix \cite[Def.~3.14, Thm.~3.15]{RT-LA&GD}

The most significant example of a reduced $W$-matrix $Q$ is given by a weight matrix $Q$ of a complete fan $\Si$. Notice that, by condition (W.c), a $W$-matrix is always $\GL$-equivalent to a positive matrix. In particular, this fact holds for the weight matrix of a complete fan.
\halfline

A \emph{$CF$-matrix} is an $F$-matrix satisfying the further condition:
\begin{itemize}
  \item[(F.e)] $\Ls_c(V)=\Z^n$
\end{itemize}
In particular, if $Q$ is a $W$-matrix, then $W=\G(Q)$ is a $CF$-matrix \cite[Prop.~3.12]{RT-LA&GD}.

\begin{remark}\label{rem:fan-over}
  Given a reduced $F$-matrix $V\in\Z^{n,n+r}$, with $r\ge 1$, there always exists a complete fan $\Si$ over $V$, as $\SF(V)\neq\emptyset$. Then, there always exists a $\Q$-factorial complete toric variety $X(\Si)$ admitting $V$ as a fan matrix.

  On the other hand, given a reduced $W$-matrix $Q\in\Z^{r,n+r}$, with $r,n\ge 1$, there always exists a $\Q$-factorial, complete toric variety admitting $Q$ as a weight matrix: in fact $W=\G(Q)\in\Z^{n,n+r}$ is a reduced $F$-matrix (actually also $CF$) with $r\ge 1$. Then, there exists a $\Q$-factorial complete toric variety $Y(\Theta)$, with $\Theta\in\SF(W)$, whose fan matrix is $W$ and whose weight matrix is $\G(W)=Q$: in fact, since $W=\G(Q)$ is a $CF$-matrix, then $\G(W)$ is $\GL$-equivalent to $Q$ by the extra-condition (F.e).
\end{remark}
\oneline

\subsection{The GKZ-decomposition: cells and chambers}\label{ssez:GKZ}
Given a toric variety $X(\Si)$, consider the pseudo-effective cone $\overline{\Eff}(X)$ and its subcone $\overline{\Mov}(X)$, determined by \emph{movable} divisors, that are divisors whose complete linear systems have no fixed divisorial components up to linear equivalence, which are cones inside
$$\Cl(X)_{\mathbb R}:=\Cl(X)\otimes_{\mathbb Z}\mathbb R
\cong\mathbb R^r$$
 where $r:=\rk\Cl(X)$ is the, so called, \emph{rank of $X$}. Both $\overline{\Eff}(X)$ and $\overline{\Mov}(X)$ support a fan structure, the so called \emph{Gelfand-Kapranov-Zelevinsky (GKZ)-decomposition} or \emph{secondary fan} (see e.g. \cite[\S\,14.4]{CLS} and \cite[\S\,2.2.2]{ADHL}). Here we will adopt notation introduced in \cite[\S~2.5]{R-wMDS}, to which the interested reader is referred for any further detail. A cone of the secondary fan is called a \emph{cell} of the GKZ-decomposition. A full dimensional cell is called a \emph{chamber} of the GKZ-decomposition.

If $X$ admits weight matrix $Q$ then
\begin{eqnarray*}
  \overline{\Eff}(X) = \langle Q\rangle &=:&\Eff(Q) \subset \R^r\\
  \overline{\Mov}(X) = \bigcap_{i=1}^{n+r}\left\langle Q^{\{i\}}\right\rangle&=:&\Mov(Q) \subset \R^r
\end{eqnarray*}
Moreover, the cone $\Nef(X)$ generated by classes of nef divisors is a sub-cone of $\overline{\Mov}(X)$ and it turns out that
\begin{equation*}
  \Nef(X) = \bigcap_{I\in \I_\Si} \langle Q_I\rangle \subseteq \Mov(Q)\subseteq\Eff(Q)
\end{equation*}
In other words:
\begin{itemize}
  \item[($*$)] \emph{$\Nef(X)$ is the intersection of all the cones in the bunch of cones $\B(\Si)$ associated with the fan $\Si$ of $X$.}
\end{itemize}
In particular, if $X$ is complete then $Q$ is a reduced $W$-matrix and it can be assumed to be a positive matrix, so that
$$\Nef(X)\subseteq\Mov(Q)\subseteq\Eff(Q)\subseteq\R^r_+$$
can be thought of cones inside the positive orthant $\R^r_+$ of $\R^r$.

For what follows, one can always assume that:
 \begin{itemize}
   \item[($**$)] \emph{$\Mov(Q)$ is of full dimension $r=\rk Q$ inside $\Eff(Q)$} \cite[Thm.~2.2.2.6~(i)]{ADHL}.
 \end{itemize}

\begin{definition}[Def.~15 in \cite{R-wMDS}]
  A cell $\g$ of the GKZ-decomposition of $\Mov(Q)$ is called a \emph{geometric cell} (or simply a \emph{g-cell}) if there exists a fan $\Si$ over $W=\G(Q)$ such that $\g=\Nef(X(\Si))$. A full dimensional geometric cell inside $\Mov(Q)$ is called a \emph{geometric chamber} (or simply a \emph{g-chamber}).
\end{definition}

Let us recall some key facts: for a proof consider \cite[Prop.~9 and 10]{R-wMDS}.
\begin{proposition}
  Let $Q$ be a reduced $W$-matrix. Then the following assertions hold:
  \begin{enumerate}
    \item the $\Nef$ cone of a $\Q$-factorial and complete toric variety, whose weight matrix is $Q$, is always a g-cell,
    \item every chamber $\g$ of the GKZ-decomposition of $\Mov(Q)$ is a g-chamber.
  \end{enumerate}
\end{proposition}

\subsection{Isomorphisms in codimension one and small $\Q$-factorial modification}

A birational map $f:X\dashrightarrow Y$ between algebraic varieties is called an \emph{isomorphism in codimension one} if it is biregular in codimension one, that is, there exist Zariski open subsets $U\subseteq X$ and $V\subseteq Y$ such that $f|_U:U\stackrel{\cong}{\longrightarrow}V$ is a biregular morphism and
$\codim (X\setminus U)\ge 2$ and $\codim(Y\setminus V)\ge 2$. We will write
\begin{equation*}
  f:X\cong_1 Y
\end{equation*}
Furthermore, suppose that $X$ and $Y$ are toric varieties and
that the small $\Q$-factorial modification $f:X\cong_1Y$ is
toric. After identifying their dense tori via $f$, their fans
have the same rays. Hence $X$ and $Y$ admit a common fan matrix
$V$ and a common weight matrix $Q=\G(V)$, and correspond to
two fans $\Si,\Theta\in\SF(V)$, respectively.
Their nef cones $\Nef(X)$ and $\Nef(Y)$ are therefore two
g-cells of the GKZ-decomposition of $\Mov(Q)$.
In particular (see \cite[\S~15.3]{CLS} and
\cite[Cor.~2.4]{Hu-Keel}):
\begin{enumerate}
  \item if $\Nef(X)$ and $\Nef(Y)$ are adjacent chambers, then $f$ is called a \emph{wall crossing},
  \item if $\Nef(X)$ and $\Nef(Y)$ are (not necessarily adjacent) chambers, then the small $\Q$ factorial modification $f$ is the composition of a finite number of wall crossings,
  \item if $\Nef(X)$ and $\Nef(Y)$ are just g-cells, then $f$ can be obtained by composing a finite number of small birational contractions, small birational resolutions and wall crossings.
\end{enumerate}

\section{Bounds on Fano Toric Varieties}\label{sez:bounds}
In the following:
\begin{itemize}
  \item a \emph{Gorenstein} algebraic variety is a normal complete algebraic variety $X$ whose canonical divisor $K_X$ is Cartier;
  \item a \emph{$\Q$-Gorenstein} algebraic variety is a normal complete algebraic variety $X$ whose canonical divisor $K_X$ admits a Cartier multiple $kK_X$, for some $k\in\N$;
  \item a \emph{Fano} algebraic variety is a Gorenstein projective variety $X$ whose anticanonical divisor $-K_X$ is ample;
  \item a \emph{$\Q$-Fano} algebraic variety is a $\Q$-Gorenstein projective variety $X$ whose anticanonical divisor $-K_X$ admits an ample multiple $-kK_X$, for some $k\in\N$;
  \item a \emph{weak $\Q$-Fano} algebraic variety is a $\Q$-Gorenstein variety whose anticanonical divisor $-K_X$ admits a semi-ample multiple $-kK_X$, for some $k\in\N$.
\end{itemize}
Let us underline that, under the adopted notation, ample and semi-ample divisors are Cartier divisors.

\begin{definition}[Gorenstein index]\label{def:index}
The least positive integer $k$ such that $kK_X$ is Cartier is
called the \emph{Gorenstein index} of $X$. When no confusion is
possible, we simply call it the \emph{index} of $X$.
\end{definition}

\subsection{Fano-Gorenstein condition on the weight matrix}
\begin{theorem}\label{thm:Gorenstein}
Let $X$ be a complete toric variety with
$$\dim X=n,\ \rk \Cl(X)=r,\ m:=n+r$$
Let $\Si$ be the fan of $X$ and $V$ be a fan matrix associated with $\Si$. Let $Q=\mathcal{G}(V)$ be a weight matrix of $X$. Define
\begin{equation*}
  \mathcal{I}_\Si(n):=\{I\subset\{1,\ldots,m\}\,|\,\langle V^I\rangle\in\Si(n)\}
\end{equation*}
Then:
\begin{enumerate}[label=(\alph*)]
  \item if $X$ is Gorenstein then
  \begin{equation*}
    \forall\,I\in\mathcal{I}_\Si(n)\quad\exists\,\mathbf{q}_I\in\Z^{|I|}:Q_I\cdot\q_I= Q\cdot \mathbf{1}_m\in \Z^r
  \end{equation*}
  \item $X$ is weak $\Q$-Fano if and only if
  \begin{equation*}
    \forall\,I\in\mathcal{I}_\Si(n)\quad\exists\,\mathbf{q}_I\in\Q^{|I|}:\q_I\ge \0_{|I|}\ \text{and}\  Q_I\cdot\q_I= Q\cdot \mathbf{1}_m\in \Z^r
  \end{equation*}
  where $\q_I\ge \0_{|I|}$ means that $\q_{I}$ admits non-negative entries only.
  \item $X$ is $\Q$-Fano if and only if $X$ is weak $\Q$-Fano and
  \begin{equation*}
    \forall\,I\in\mathcal{I}_\Si(n)\quad \q_I> \0_{|I|}
  \end{equation*}
  that is,  $\q_{|I|}$ admits strictly positive entries only.
\end{enumerate}
\end{theorem}

\begin{proof}
Write $V=(\v_1,\ldots,\v_m)$, and let $D_j$ be the
torus-invariant prime divisor corresponding to the ray
$\langle\v_j\rangle$. We use the canonical representative
\[
K_X=-\sum_{j=1}^m D_j.
\]
For every $I\in\I_\Si(n)$, set
$\sigma_I:=\langle V^I\rangle$.
Recall that $M=\Hom(N,\Z)$ and put
$M_\Q:=M\otimes_\Z\Q$.
Since $Q=\G(V)$, one has $QV^T=0$ and
\[
\ker\bigl(Q:\Q^m\longrightarrow\Q^r\bigr)
=\{V^T\m\mid\m\in M_\Q\},
\]
where we use the lattice coordinates determined by $V$
and the dual coordinates on $M$.

\smallskip
\noindent\textup{(a)}
Suppose that $X$ is Gorenstein.
By the toric Cartier criterion
\cite[Thm.~4.2.8]{CLS}, for every $I\in\I_\Si(n)$
there exists $\m_I\in M$ such that
\[
\langle\m_I,\v_j\rangle=-1
\qquad\text{for every }j\notin I.
\]
The integral vector
\[
\aa_I:=\1_m+V^T\m_I
\]
therefore vanishes outside $I$.
Let $\q_I\in\Z^{|I|}$ be its subvector with entries
indexed by $I$. Then
\[
Q_I\q_I
=Q\aa_I
=Q\1_m+QV^T\m_I
=Q\1_m,
\]
which proves \textup{(a)}.

\smallskip
\noindent\textup{(b)}
Suppose that, for every $I\in\I_\Si(n)$, there exists
$\q_I\in\Q^{|I|}$ such that
\[
Q_I\q_I=Q\1_m,
\qquad
\q_I\geq\0_{|I|}.
\]
Extend $\q_I$ to a vector $\aa_I\in\Q^m$ by inserting
zeros in the positions outside $I$.
Then $Q(\aa_I-\1_m)=0$, so the kernel identity above
gives $\m_I\in M_\Q$ such that
\[
\aa_I=\1_m+V^T\m_I.
\]
Consequently,
\[
\langle\m_I,\v_j\rangle=-1
\quad(j\notin I),
\qquad
\langle\m_I,\v_j\rangle\geq-1
\quad(j\in I).
\]
Choose a positive integer $\ell$ such that
$\ell\m_I\in M$ for every $I$.
The equalities show that $-\ell K_X$ is Cartier,
with Cartier data $\{\ell\m_I\}_{I\in\I_\Si(n)}$.
The inequalities give
\[
\langle\ell\m_I,\v_j\rangle\geq-\ell
\qquad\text{for all }I\in\I_\Si(n)
\text{ and }1\leq j\leq m.
\]
By the toric criterion for global generation
\cite[Thm.~6.1.7]{CLS}, $-\ell K_X$ is globally
generated. Hence $X$ is weak $\Q$-Fano.

Conversely, suppose that $X$ is weak $\Q$-Fano.
Choose a positive integer $\ell$ such that
$-\ell K_X$ is Cartier and globally generated.
Its Cartier data, divided by $\ell$, give vectors
$\m_I\in M_\Q$ satisfying the same equalities and
inequalities. Thus
\[
\aa_I:=\1_m+V^T\m_I
\]
is nonnegative and vanishes outside $I$.
Its subvector $\q_I$ indexed by $I$ satisfies
$\q_I\geq\0_{|I|}$ and
$Q_I\q_I=Q\1_m$, proving the converse.

\smallskip
\noindent\textup{(c)}
Under the correspondence used in \textup{(b)}, the
condition $\q_I>\0_{|I|}$ is equivalent to
\[
\langle\m_I,\v_j\rangle>-1
\qquad\text{for every }j\in I,
\]
while equality holds for every $j\notin I$.
After clearing denominators simultaneously, these
conditions become
\[
\langle\ell\m_I,\v_j\rangle=
-\ell\quad(j\notin I),
\qquad
\langle\ell\m_I,\v_j\rangle>
-\ell\quad(j\in I).
\]
For the complete fan $\Si$, these are precisely the
strict convexity conditions for the Cartier data of
$-\ell K_X$. By the toric ampleness criterion
\cite[Thm.~6.1.14]{CLS}, they are equivalent to
$-\ell K_X$ being ample.
In particular, they imply that $X$ is projective.
Thus the existence of a strictly positive solution
$\q_I$ for every $I\in\I_\Si(n)$ is equivalent to
$X$ being $\Q$-Fano, proving \textup{(c)}.
\end{proof}

\begin{definition}\label{def:Fano-Gorenstein_Q}
   Let $Q\in\Z^{r,n+r}$ be a reduced $W$-matrix with $r,n\ge1$. If there exists a fan  $\Th$ \emph{over} the Gale dual matrix $W=\G(Q)$ (recall Def.~\ref{def:F,W-matrix}), such that the toric variety $Y(\Th)$ satisfies:
    \begin{itemize}
      \item condition (a) in Theorem~\ref{thm:Gorenstein} then $Q$ is called a \emph{Gorenstein weight matrix},
      \item condition (b) in Theorem~\ref{thm:Gorenstein}  then $Q$ is called a \emph{weak $\Q$-Fano weight matrix},
      \item condition (c) in Theorem~\ref{thm:Gorenstein}  then $Q$ is called a \emph{$\Q$-Fano weight matrix},
      \item both conditions (a) and (b) then $Q$ is called a \emph{weak Fano weight matrix},
      \item both conditions (a) and (c) then $Q$ is called a \emph{Fano weight matrix}
    \end{itemize}
\end{definition}

\begin{remark}\label{rem:r=1}
  Notice that, for $r=1$, the reduced weight matrix $Q=(q_0,q_1,\ldots,q_n)$ is actually the reduced weight system of the wps $\P(q_0,q_1,\ldots,q_n)$ and Theorem~\ref{thm:Gorenstein} implies that \emph{if a fake wps $X$ is Gorenstein then}
  \begin{equation*}
    \forall\,i=0,1,\ldots,n\quad q_i\,|\,Q\cdot\1_{n+1}=\sum_{i=0}^n q_i=:|Q|
  \end{equation*}
  which is a well known condition. In this case, $Q$ admits the bunch of cones determined by the  fan $\Si$ of $\P(Q)$, with $\I_\Si(n)=\{\{0\},\{1\},\ldots,\{n\}\}$, so that
   $$\forall\,i\quad Q_{\{i\}}=q_i\ \text{and}\ \q_{\{i\}}= |Q|/q_i>0$$
   Then a Gorenstein weight matrix $Q=(q_0,q_1,\ldots,q_n)$ is always a Fano one.
\end{remark}

\subsection{Multiplicity of a complete toric variety}\label{ssez:molteplicità}

Let $X=X(\Si)$ be a complete toric variety. Set $n:=\dim X, r:=\rk(\Cl(X)), m:=n+r=|\Si(1)|$ and let
\begin{eqnarray*}
  V \in \Z^{n,m}&&\text{be a fan matrix of $X(\Si)$} \\
  Q = \G(V)\in \Z^{r,m}&&\text{be a weight matrix of $X$}
\end{eqnarray*}
Recall the following results (see \cite[\S~2.6]{R-CovMDS}).
\subsubsection{Universal 1-covering}\label{sssez:U1covering} There exists a \emph{universal 1-covering} $\pi:Y\twoheadrightarrow X$, where $Y=Y(\Theta)$ is a complete toric variety with trivial fundamental \'{e}tale group in codimension one: $\pi_1^{\text{ét}}(Y)^{(1)}=0$. In particular, $\pi$ is a \emph{1-covering}, that is a finite surjective morphism which is \'{e}tale in codimension one. Moreover $\pi$ is induced by the natural action of the finite abelian group
\begin{equation*}
  G:=\pi_1^{\text{ét}}(X)^{(1)}\cong\Tors(\Cl(X))\cong N/N_1
\end{equation*}
where $N_1\subset N$ is the sublattice spanned by $\Si(1)\cap N$, meaning that $X\cong Y/G$. This fact allows us to adopt a notation introduced by Kasprzyk \cite[Def.~2.7]{Kasprzyk} for fake wps (i.e. in case $r=1$):

\begin{definition}\label{def:mult}
  The \emph{multiplicity} of $X$, denoted by $\mult X$, is the index of the sublattice $N_1$ spanned by $\Si(1)\cap N$ inside $N$. That is, the order of the acting group $G$
  \begin{equation*}
    \mult X=|G|
  \end{equation*}
\end{definition}

\begin{remark}\label{rem:Univ-costruzione}
  For the important role that the universal $1$-covering $\pi:Y\twoheadrightarrow X$ will play in the following, let us recall how the toric variety $Y$ is constructed.
  \begin{itemize}
    \item Start from $X$, whose fan matrix is $V\in \Z^{n,m}$, whose fan is generated as the fan of faces of maximal cones in $\Si(n)=\{\langle V^I\rangle\,|\,I\in \I_\Si(n)\}$ and whose weight matrix is given by $Q=\G(V)$.
    \item Construct the fan matrix $W=\G(Q)$ and the fan $\Th$ generated as the fan of faces of maximal cones in $\Th(n):=\{\langle W^I\rangle\,|\,I\in\I_\Si(n)\}$.
    \item The toric variety $Y=Y(\Th)$ is the universal 1-covering $\pi:Y\twoheadrightarrow X$, by means of the map of toric varieties $\pi$ induced by the obvious map of fans $\b:N\to N$, such that $\b(\langle W^I\rangle)=\langle V^I\rangle$, for every $I\in\I_\Si(n)$. In particular $Y$ has the same weight matrix $Q$ of $V$, as
        $\G(W)=\G(V)=Q$.
  \end{itemize}
\end{remark}

\subsubsection{Gorenstein universal 1-covering} For a wps $\P(Q)=\P(q_0,\ldots,q_n)$, condition (a) in Theorem~\ref{thm:Gorenstein} is actually equivalent to asking that $\P(Q)$ is Gorenstein, that is
\begin{equation}\label{Gorenstein wps}
  \P(q_0,\ldots,q_n)\ \text{is Gorenstein}\ \Longleftrightarrow\ \forall\,i=0,\ldots,n\quad q_i\,|\,\sum_{i=0}^n q_i
\end{equation}
This condition can be extended to complete toric varieties of rank $r>1$, by means of condition given in item (a) of Theorem~\ref{thm:Gorenstein}.

\begin{proposition}\label{prop:Univ1covGorestein}
Let $\pi\colon Y(\Theta)\longrightarrow X(\Sigma)$ be the
universal $1$-covering of a complete toric variety, constructed
as in Remark~\ref{rem:Univ-costruzione}. Let $Q\in\mathbb Z^{r,m}$ be their common
weight matrix, where $m=n+r$, and recall that
$
\mathcal I_{\Theta}(n)=\mathcal I_{\Sigma}(n)$.
Then
\[
X \text{ is Gorenstein}
\quad\Longrightarrow\quad
Y \text{ is Gorenstein}
\quad\Longrightarrow\quad
Q \text{ is Gorenstein}.
\]
Moreover, assume that $Y$ is Gorenstein, and let
$\beta\colon N\longrightarrow N$ be the lattice homomorphism
inducing $\pi$.
For every $I\in\mathcal I_{\Sigma}(n)$, let $\m_I\in M$
be the unique lattice vector satisfying
$\langle \m_I,\mathbf w_j\rangle=1$, for every $j\notin I$,
where $\mathbf w_j$ denotes the $j$-th column of $W=\mathcal G(Q)$.
Then
\[
X \text{ is Gorenstein}
\quad\Longleftrightarrow\quad
\m_I\in\operatorname{Im}\beta^T
\quad\text{for every }I\in\mathcal I_{\Sigma}(n).
\]
\end{proposition}

\begin{proof} If $X$ is Gorenstein, then $K_X$ is Cartier. Since $\pi$ is
a $1$-covering, there is no ramification divisor, and
\[
K_Y=\pi^*K_X.
\]
Thus $K_Y$ is Cartier and $Y$ is Gorenstein.

Since $Y$ is simply connected in codimension $1$, its class
group is torsion-free, and its divisor sequence takes the form
\[
0\longrightarrow M
\xrightarrow{\operatorname{div}_Y}\mathbb Z^m
\xrightarrow{Q}\mathbb Z^r
\longrightarrow0,
\]
where
\[
\operatorname{div}_Y(\m)
= W^T\cdot\m=
\bigl(
\langle\m,\mathbf w_1\rangle,\ldots,
\langle\m,\mathbf w_m\rangle
\bigr)^T.
\]
By the toric Cartier criterion, $Y$ is Gorenstein if and only if,
for every $I\in\mathcal I_{\Theta}(n)$, there exists $\m_I\in M$
such that
$
\langle\m_I,\mathbf w_j\rangle=1$
 for every $j\notin I$.
Equivalently, for every such $I$, there exists
$\mathbf a_I\in\mathbb Z^m$, supported on $I$, such that
\[
\mathbf 1_m-\mathbf a_I
\in\operatorname{Im}(\operatorname{div}_Y)=\ker Q.
\]
Writing $\mathbf q_I$ for the vector formed by the entries
of $\mathbf a_I$ indexed by $I$, this condition becomes
\[
Q_I\cdot\mathbf q_I=Q\cdot\mathbf 1_m.
\]

Finally, assume that $Y$ is Gorenstein.
The vector $\m_I$ is unique because the rays of the maximal
cone $\langle W^I\rangle$ span $N_{\mathbb Q}$.
Let $\mathbf v_j=\beta(\mathbf w_j)$ be the corresponding
primitive ray generators of $\Sigma$.
The Cartier criterion on $X$ requires, for every $I$, a vector
$\m'_I\in M$ such that
$
\langle\m'_I,\mathbf v_j\rangle=1$
for every $j\notin I$.
Since
\[
\langle\m'_I,\mathbf v_j\rangle
=
\langle\beta^T(\m'_I),\mathbf w_j\rangle,
\]
uniqueness of $\m_I$ shows that this is equivalent to
\[
\m_I=\beta^T(\m'_I).
\]
Hence $X$ is Gorenstein if and only if every $\m_I$ belongs
to $\operatorname{Im}\beta^T$.
\end{proof}

\subsubsection{Divisibility of fan matrices}\label{sssez:divisione} Let $W\in\Z^{n,m}$ be a fan matrix of the universal 1-covering $Y$ of $X$. Recall that there exists a fan matrix $V$ of $X$ such that
$$W=\G(Q)=\G(\G(V))$$
is uniquely determined up to $\GL$-equivalence \cite[\S~3.1]{RT-LA&GD}. Then there exists a unique matrix $B\in\GL_n(\Q)\cap\Z^{n,n}$ giving \emph{the quotient} of the division of $V$ by $W$, that is
\begin{equation*}
  V=B\cdot W
\end{equation*}
$B$ is called the \emph{quotient matrix of $V$}. In particular, by the construction of $\pi:Y\twoheadrightarrow X$, as recalled in Remark~\ref{rem:Univ-costruzione}, there is a commutative diagram
 \begin{equation}\label{div-diagram-covering}
      \xymatrix{&&&0\ar[d]&\\
& 0 \ar[d] & 0 \ar[d] & \ker(\pi^*)=\Tors(\Cl(X)) \ar[d] & \\
0 \ar[r] & M \ar[r]^-{div_X}_-{V^T}\ar[d]^-{\b^T}_-{B^T} &
\mathcal{W}_{\T} (X)=\Z^{|\Si(1)|} \ar[r]^-{d_X}_-{Q\oplus \Ga}\ar[d]_-{\mathbf{I}_{n+r}} & \Cl(X) \ar[r]\ar[d]^-{\pi^*} & 0 \\
0 \ar[r] & M \ar[r]^-{div_Y}_-{W^T}\ar[d]&\mathcal{W}_{\T}(Y)=\Z^{|\Th(1)|}\ar[r]^-{d_Y}_{Q}\ar[d] & \Cl (Y) \ar[r]\ar[d] & 0 \\
 & \coker(\b^T)\cong\Tors(\Cl(X))\ar[d] & 0 & 0 & \\
 &0&&&}
\end{equation}
where $d_X$ and $d_Y$ are the class morphisms on torus invariant Weil divisors. In the diagram, representative matrices of the involved morphisms are made explicit and the isomorphism $\coker(\b^T)\cong\Tors(\Cl(X))$ comes from the Snake Lemma. Then one obtains the following interpretation of $\mult X$
\begin{equation}\label{molteplicità}
  \mult X=\deg(\pi)=|G|=|\coker(\b^T)|=|\det(B)|
\end{equation}

\begin{remark}\label{rem:azioneGa}
  The matrix $\Ga$ appearing in diagram (\ref{div-diagram-covering}) represents the torsion part of the class morphism $d_X:\Weil(X)\to\Cl(X)$. Recalling the Cox's quotient presentation of a non-degenerate toric variety, $\Ga$ encodes the action, of the finite group $G\cong\Tors(\Cl(X))$ on $Y$, giving rise to the quotient $X\cong Y/G$.
\end{remark}

\subsection{Polar duality} Let $X=X(\Si)$ be a Fano toric variety and $Y=Y(\Th)$ its universal 1-covering. Then, their \emph{polar} Fano varieties $X^\circ=X^\circ(\Si^\circ)$ and $Y^\circ=Y^\circ(\Th^\circ)$, respectively, are well defined, after Batyrev \cite{Batyrev94}. Let $V,W,V^\circ,W^\circ$ be fan matrices of $X,Y,X^\circ,Y^\circ$, respectively. Therefore
\begin{eqnarray}\label{polarità}
  \conv(V^\circ) = \De_{-K_X} &\text{and}& \conv(V) = \De_{-K_{X^\circ}} \\
  \nonumber
  \conv(W^\circ) = \De_{-K_Y} &\text{and}& \conv(W) = \De_{-K_{Y^\circ}}
\end{eqnarray}
Call $r^\circ:=\rk(\Cl(X^\circ))$ and $m^\circ:=n+r^\circ=|\Si^\circ(1)|$. Then $V^\circ\in Z^{n,m^\circ}$.

\begin{theorem}\label{thm:1-coveringpolare}
  The lattice map $\b^T:M\to M$ induces a map of fans $\b^T:\Si^\circ\to \Th^\circ$ giving rise to a 1-covering $\pi^T: X^\circ\twoheadrightarrow Y^\circ$ such that
\[
\deg(\pi^T)=\deg(\pi)=|\det B|,
\qquad
Y^\circ\cong X^\circ/\coker\beta.
\]
Moreover, pull-back on divisor classes, as defined e.g. in \cite[\S~1.5]{R-CovMDS}, gives
\[
\coker\beta\cong
\ker\bigl((\pi^T)^*:\Cl(Y^\circ)\to\Cl(X^\circ)\bigr).
\]
  \end{theorem}

  The statement follows from the following two lemmas.

  \begin{lemma}\label{lem:polar}
    Recall the definition of $\I_\Si(n)$ given in Theorem~\ref{thm:Gorenstein}. Then, up to a column permutation,
    \begin{equation*}
      V^\circ=\left(-(V^I)^{*'}\cdot\1_n\right)_{I\in\I_\Si(n)}\in \Z^{n,m^\circ}
    \end{equation*}
    where $(V^I)^{*'}$ is the \emph{transverse} matrix of a $n\times n$ non-singular submatrix of $V^I$ (recall notation in \S~\ref{ssez:Not-matrici}) and $$\v^\circ_I:=-(V^I)^{*'}\cdot\1_n$$
    is the $I$-th column of $V^\circ$. Notice that: $|\I_\Si(n)|=|\Si(n)|=|\Si^\circ(1)|=m^\circ$.
  \end{lemma}

  \begin{lemma}\label{lem:quozientepolare}
    Let $W^\circ$ be a fan matrix of $Y^\circ(\Th^\circ)$. Then $V^\circ$ divides $W^\circ$ and the quotient matrix is given by the transposed matrix $B^T\in\GL_n(\Q)\cap\Z^{n,n}$, that is
    \begin{equation*}
      W^\circ = B^T \cdot V^\circ
    \end{equation*}
    up to a permutation on columns.
  \end{lemma}

  \begin{proof}[Proof of Lemma \ref{lem:polar}]
    If $X$ is $\Q$-factorial, that is $\Si$ is a simplicial fan, then $V^I$ is a $n\times n$ integer matrix, for every $I\in\I_\Si(n)$. In this case, $(V^I)^{*'}$ is defined as the transverse matrix $(V^I)^*$ of $V^I$, that is
    \begin{equation*}
      (V^I)^{*'}:=(V^I)^*=((V^I)^T)^{-1}
    \end{equation*}
    Notice that the column vector $\v^\circ_I:=-(V^I)^{*}\cdot\1_n$ is the unique solution of the Cramer linear system $(V^I)^T\cdot \x=-\1_n$ so giving the $I$-th vertex of the polar polytope $\conv(V^\circ)=\De_{-K_X}$.
    \noindent When $X$ is not $\Q$-factorial, there must exist $I\in \I_\Si(n)$ such that $|I|<r$ and $V^I$ turns out to admit a column number greater than $n$. Anyway, for any choice of a non-singular $n\times n$ submatrix $H$ of $V^I$, the resulting Cramer linear system $H^T\cdot\x=-\1_n$ always admits the same unique solution $\v^\circ_I:=-H^{*}\cdot\1_n$, which does not depend on the choice of $H$ in $V^I$, as it always gives the same $I$-th vertex of $\conv(V^\circ)$.
  \end{proof}

  \begin{proof}[Proof of Lemma \ref{lem:quozientepolare}] Let $B$ the quotient matrix of $V$, that is
  \begin{equation*}
    V=B\cdot W
  \end{equation*}
  Then, for every $I\in\I_\Si(n)$, one has $V^I=B\cdot W^I$. In particular, if $K$ is any non-singular $n\times n$ submatrix of $W^I$ then $H=B\cdot K$ is a non-singular $n\times n$ submatrix of $V^I$. This is enough to guarantee that
  \begin{equation*}
    \v^\circ_I=-H^*\cdot\1_n= - B^*\cdot K^*\cdot \1_n \ \Longrightarrow\ \w^\circ_I=-K^*\cdot \1_n= - B^T\cdot H^*\cdot\1_n=B^T\cdot\v^\circ_I
  \end{equation*}
  \end{proof}

  \begin{proof}[Proof of Theorem~\ref{thm:1-coveringpolare}]
By Lemma~\ref{lem:quozientepolare}, $W^\circ=B^T V^\circ$,
after a compatible ordering of the columns. Thus $\beta^T$
identifies the rational fans $\Si^\circ$ and $\Th^\circ$ and
induces a finite toric morphism $\pi^T$.
Corresponding ray generators are primitive, so the morphism
is unramified in codimension one. Its degree is
$|\det B^T|=|\det B|=\deg(\pi)$.

The kernel of the induced isogeny of dense tori is
$\Hom(\coker\beta,\K^*)$. Hence
\[
Y^\circ\cong X^\circ/\Hom(\coker\beta,\K^*).
\]
Since a finite abelian group is abstractly isomorphic to its
character group over $\K$, transporting this action through
such an isomorphism gives the stated quotient by $\coker\beta$.
The divisor sequences give the commutative diagram
 \begin{equation}\label{dual-div-diagram-covering}
      \xymatrix{&&&0\ar[d]&\\
& 0 \ar[d] & 0 \ar[d] & \ker{\pi^T}^* \ar[d] & \\
0 \ar[r] & N \ar[r]^-{div_{Y^\circ}}_-{(W^\circ)^T}\ar[d]^-{\b}_-{B} &
\mathcal{W}_{\T} (Y^\circ)\cong\Z^{|\Th^\circ(1)|} \ar[r]^-{d_{Y^\circ}}_-{Q^\circ\oplus \Ga_1^\circ}\ar[d]_-{\mathbf{I}_{m^\circ}} & \Cl(Y^\circ) \ar[r]\ar[d]^-{{\pi^T}^*} & 0 \\
0 \ar[r] & N \ar[r]^-{div_{X^\circ}}_-{(V^\circ)^T}\ar[d]&\mathcal{W}_{\T}(X^\circ)\cong\Z^{|\Si^\circ(1)|}
\ar[r]^-{d_{X^\circ}}_{Q^\circ\oplus\Ga_2^\circ}\ar[d] & \Cl (X^\circ) \ar[r]\ar[d] & 0 \\
 & \coker(\b)\ar[d] & 0 & 0 & \\
 &0&&&}
\end{equation}
The Snake Lemma yields $\coker(\b)\cong\ker {\pi^T}^*$,
proving the last assertion.
\end{proof}

\begin{corollary}\label{cor:dualGale}
  $\G(V^\circ)=Q^\circ=\G(W^\circ)$ that is, as for $X$ and $Y$, also their polar Fano varieties admit the same weight matrix.
\end{corollary}

\subsection{The universal 1-covering of the polar variety}\label{ssez:1-cov_polar} As above, let $X=X(\Si)$ be a Fano toric variety and $\pi:Y(\Th)\to X(\Si)$ its universal 1-covering, represented by the quotient matrix $B$ such that $V=B\cdot W$. Let $X^\circ(\Si^\circ)$ be the polar Fano toric variety of $X$ and $Q^\circ=\G(V^\circ)$ its weight matrix. By applying the procedure described in Remark~\ref{rem:Univ-costruzione} to $X^\circ$ one obtains the universal 1-covering
\begin{equation}\label{polar-univ-1cov}
  \pi^\circ: Z^\circ\twoheadrightarrow X^\circ
\end{equation}
Namely, define $\L^\circ=\G(Q^\circ)\in\Z^{n,m^\circ}$ and construct the fan $\Xi^\circ$ of faces of maximal cones
\begin{equation*}
  \Xi^\circ(n):=\left\{\langle (\L^\circ)^I\rangle\,|\, I\in\I_{\Si^\circ}(n)\right\}
\end{equation*}
Following \S~\ref{sssez:divisione}, let $C\in\GL_n(\Q)\cap\Z^{n,n}$ be the quotient matrix of $V^\circ$, that is \begin{equation*}
  V^\circ = C\cdot \L^\circ
\end{equation*}
$C$ represents the natural map of fans $\g:\Xi^\circ\to \Si^\circ$, giving rise to the universal 1-covering in (\ref{polar-univ-1cov}).

\begin{theorem}\label{thm:fattorizzazione}
Let $Z:=(Z^\circ)^\circ$ and $\L:=(\L^\circ)^\circ$.
Then $A:=C^T B$ and its transpose are the quotient matrices
determined by
\[
\L=AW,\qquad W^\circ=A^T\L^\circ.
\]
Calling $\Xi:=(\Xi^\circ)^\circ$ the fan of $Z$, matrices $A$ and $A^T$ represent maps of fans $\a:\Th\to\Xi$ and $\a^T:\Xi^\circ\to\Th^\circ$, respectively, inducing the universal 1-coverings $p: Y\twoheadrightarrow Z$ and $p^T: Z^\circ\to Y^\circ$ and the following dual factorizations of 1-coverings
\begin{equation}\label{fattorizzazioni}
\xymatrix{
Y\ar[dd]_-p\ar[dr]^\pi&&&&
Z^\circ\ar[dd]_-{p^T}\ar[dr]^{\pi^\circ}&\\
&X\ar[dl]^{\pi^{\circ T}}&
\ar@{<~>}[r]^\circ&&&X^\circ\ar[dl]^{\pi^T}\\
Z&&&&Y^\circ&
}
\end{equation}
Moreover,
\[
Z\cong Y/\coker\alpha^T,\qquad
Y^\circ\cong Z^\circ/\coker\alpha,
\]
and
\[
\begin{aligned}
\coker\alpha^T&\cong\Tors(\Cl(Z))
                 \cong\pi_1^{\acute{e}t}(Z)^{(1)},\\
\coker\alpha&\cong\Tors(\Cl(Y^\circ))
                 \cong\pi_1^{\acute{e}t}(Y^\circ)^{(1)}.
\end{aligned}
\]
\end{theorem}

The proof is an immediate consequence of the previous constructions. In particular, matrix products  $A:=C^T\cdot B$ and  $A^T=B^T\cdot C$ induce factorizations $p=\pi^{\circ T}\circ\pi$ and $p^T=\pi^T\circ\pi^\circ$, respectively, between associated maps of fans and toric varieties.

\begin{proof}
Apply Theorem~\ref{thm:1-coveringpolare} to
$\pi^\circ:Z^\circ\to X^\circ$. Its polar $1$-covering
$\pi^{\circ T}:X\to Z$ is induced by $\gamma^T$, and
$\L=C^T V=C^T BW=AW$. Similarly,
$W^\circ=B^T V^\circ=B^T C\L^\circ=A^T\L^\circ$.
The lattice factorizations give the displayed factorizations
of $p$ and $p^T$. Their sources $Y$ and $Z^\circ$ are simply
connected in codimension one, so both $1$-coverings are universal.

The maps on character lattices are $\alpha^T$ and $\alpha$.
Consequently the deck groups are, respectively,
\[
\Hom(\coker\alpha^T,\K^*),\qquad
\Hom(\coker\alpha,\K^*).
\]
Transporting these actions through abstract isomorphisms with
$\coker\alpha^T$ and $\coker\alpha$ proves the quotient
descriptions and the fundamental group identifications.

Finally, the divisor sequences and the Snake Lemma give
\[
0\longrightarrow\coker\alpha^T\longrightarrow\Cl(Z)
\xrightarrow{p^*}\Cl(Y)\longrightarrow0,
\]
\[
0\longrightarrow\coker\alpha\longrightarrow\Cl(Y^\circ)
\xrightarrow{(p^T)^*}\Cl(Z^\circ)\longrightarrow0.
\]
Since $\Cl(Y)$ and $\Cl(Z^\circ)$ are torsion-free, the
finite kernels are exactly the torsion subgroups.
\end{proof}

\subsection{The weight group $G_Q$ and its order}\label{ssez:GQ} Let $Q\in\Z^{r,m}$ be a Fano weight matrix (recall Definition~\ref{def:Fano-Gorenstein_Q}) and consider the lattice morphism $\a:N\to N$, represented by the quotient matrix of $\L=\G(Q^\circ)^\circ$, that is $A\in\GL_n(\Q)\cap\Z^{n,n}$ such that $\L=A\cdot \G(Q)$. The finite abelian group $$G_Q:=\coker(\a^T)$$
is called the \emph{weight group of Q}. The order of this group
\begin{equation*}
  g_Q:=|G_Q|
\end{equation*}
is called the \emph{weight order of $Q$}.

Analogously, the weight group of the polar weight matrix  $Q^\circ\in\Z^{r^\circ,m^\circ}$, which is clearly still Fano by reflexivity, is defined as
\begin{equation*}
  G_{Q^\circ}:=\coker(\a)
\end{equation*}

\begin{proposition}\label{prop:isoGQ}
  Let $Q$ be a Fano weight matrix. Then there exists an isomorphism  $\bar T:G_{Q^\circ}\cong G_Q$
\end{proposition}

\begin{proof}
 $N$ and $M=\Hom(N,\Z)$ are dual lattices, hence isomorphic. The point is, finding an isomorphism $T:N\stackrel{\cong}{\to}M$ restricting to an isomorphism $\im\a\cong\im\a^T$, so that it can pass to an isomorphism of quotients $\bar T:\coker\a\cong\coker\a^T$.

 As above, let $A$ and $A^T$ in $\GL_n(\Q)\cap\Z^{n,n}$ be representative matrices of $\a\in\End(N)$ and $\a^T\in\End(M)$, respectively, with respect to the choice of a basis for $N$ and its dual basis for $M$. Then, recalling notation introduced in \S~\ref{ssez:Not-matrici},
 \begin{equation*}
   \im\a=\mathcal{L}_c(A)\le\Z^n\cong N\quad\text{and}\quad \im \a^T=\mathcal{L}_c\left(A^T\right)=\mathcal{L}_r(A)\le\Z^n\cong M
 \end{equation*}
 Let $\D=\diag(c_1,\ldots,c_n)$ be the Smith normal form of $A$, that is, there exist unimodular matrices $P,Q\in\GL_n(\Z)$ such that $P\cdot A\cdot Q =\D$. In particular, $\D$ is also the Smith normal form of $A^T$, as
\begin{equation*}
  P\cdot A\cdot Q =\D\quad\Longleftrightarrow\quad Q^T\cdot A^T\cdot P^T=\D^T=\D
\end{equation*}
Then
\begin{equation*}
  A\cdot Q=P^{-1}\cdot\D\quad\text{and}\quad A^T\cdot P^T = Q^*\D
\end{equation*}
where $Q^*=(Q^T)^{-1}$ is the transverse matrix of $Q$. By the elementary divisor theorem, the columns of $P^{-1}$ and of $Q^*$ give bases $\mathcal{P}=\{\pp_1,\ldots,\pp_n\}$ of $\Z^n\cong N$ and $\mathcal{Q}=\{\q_1,\ldots,\q_n\}$ of $\Z^n\cong M$, respectively, such that $\{c_1\pp_1,\ldots,c_n\pp_n \}$ is a basis of $\im\a$ and $\{c_1\q_1,\ldots,c_n\q_n\}$ is a basis of $\im\a^T$: for the details the interested reader is referred e.g. to Theorems~1.15 and 1.17 in \cite{RT-LA&GD} and discussion therein. Then isomorphisms $T$ and $\bar T$ can be defined by setting:
\begin{equation*}
  \forall\,i=1,\ldots, n,\quad T(\pp_i):=\q_i\quad,\quad\forall\,x\in N\quad\bar T(x+\im\a):=T(x)+\im\a^T
\end{equation*}
\end{proof}

\subsection{Exact sequences on the weight group}\label{ssez:splittingGQ} Assume the same hypotheses of Theorem~\ref{thm:fattorizzazione} and consider the dual factorizations given in (\ref{fattorizzazioni}). They are induced by the following underlying factorizations of lattice endomorphisms
\begin{equation}\label{fattoreticoli}
    \xymatrix{M\ar[dd]_-{\a^T}\ar[dr]^-\g&&&&N\ar[dd]_-\a\ar[dr]^-\b& \\
                &M\ar[dl]^-{\b^T}&\ar@{<-|}[r]^{T}&&&N\ar[dl]^-{\g^T}\\
                M&&&&N&}
  \end{equation}
Then we have inclusions $i:\im\a^T\hookrightarrow\im\b^T$ and $j:\im\a\hookrightarrow\im\g^T$ inducing surjections
\begin{equation*}
  \xymatrix{\coker\a^T\ar@{->>}[r]^-{\bar i}&\coker\b^T}\quad\text{and}\quad
  \xymatrix{\coker\a\ar@{->>}[r]^-{\bar j}&\coker\g^T}
\end{equation*}
Recall our notation
 \begin{equation*}
   G_Q=\coker\a^T\,,\ G=\coker\b^T\,,\ G_{Q^\circ}=\coker\a\stackrel{\bar T}{\cong} G_Q
 \end{equation*}
and define
\begin{equation*}
  G^\circ:=\coker\g^T
\end{equation*}
so that $X^\circ$ is the geometric quotient $X^\circ\cong Z^\circ/G^\circ$\,.

  \begin{proposition}\label{prop:splittingGQ}
In the above notation, there are short exact sequences
\[
0\longrightarrow\ker\bar i
\longrightarrow G_Q
\xrightarrow{\bar i}G
\longrightarrow0,
\]
\[
0\longrightarrow\ker\bar j
\longrightarrow G_{Q^\circ}
\xrightarrow{\bar j}G^\circ
\longrightarrow0,
\]
together with canonical isomorphisms
\[
\ker\bar i\cong\operatorname{coker}\gamma,
\qquad
\ker\bar j\cong\operatorname{coker}\beta.
\]
In particular, there are abstract, generally noncanonical,
isomorphisms
\[
\ker\bar i\cong G^\circ,
\qquad
\ker\bar j\cong G.
\]
Moreover, $G_{Q^\circ}$ and $G_Q$ carry a perfect pairing
\[
\begin{aligned}
\langle\cdot,\cdot\rangle_\alpha\colon G_{Q^\circ}\times G_Q
&\longrightarrow\mathbb Q/\mathbb Z,\\
([\n],[\m])
&\longmapsto
\langle [\n],[\m]\rangle_\alpha:=
\langle \m,\alpha_{\mathbb Q}^{-1}(\n)\rangle+\mathbb Z,
\end{aligned}
\]
where $\alpha_{\mathbb Q}=\alpha\otimes_{\mathbb Z}\mathbb Q$.
With respect to this pairing, the kernels are mutual annihilators:
\[
(\ker\bar i)^\perp=\ker\bar j,
\qquad
(\ker\bar j)^\perp=\ker\bar i.
\]
Let $T\colon N\xrightarrow{\sim}M$ be the lattice isomorphism
of Proposition~\ref{prop:isoGQ}, and let
\[
\bar T\colon G_{Q^\circ}\xrightarrow{\sim}G_Q
\]
be the induced isomorphism. Define
\[
\bar j':=\bar j\circ\bar T^{-1}
\colon G_Q\longrightarrow G^\circ,
\qquad
\bar i':=\bar i\circ\bar T
\colon G_{Q^\circ}\longrightarrow G.
\]
Then there are short exact sequences
\[
0\longrightarrow\bar T(\ker\bar j)
\longrightarrow G_Q
\xrightarrow{\bar j'}G^\circ
\longrightarrow0,
\]
\[
0\longrightarrow\bar T^{-1}(\ker\bar i)
\longrightarrow G_{Q^\circ}
\xrightarrow{\bar i'}G
\longrightarrow0.
\]
Consequently,
\[
G_Q/\bar T(\ker\bar j)\cong G^\circ,
\qquad
G_{Q^\circ}/\bar T^{-1}(\ker\bar i)\cong G,
\]
and
\[
|G_Q|=|G_{Q^\circ}|=|G|\,|G^\circ|.
\]
In particular, after choosing an isomorphism
$s\colon G\xrightarrow{\sim}\ker\bar j$, the embedding
\[
\iota:=\bar T\circ s\colon G\hookrightarrow G_Q
\]
satisfies
\[
G_Q/\iota(G)\cong G^\circ.
\]
\end{proposition}

\begin{proof}
The factorizations of $\alpha$ and $\alpha^T$ give inclusions
\[
\operatorname{Im}\alpha^T
\subseteq\operatorname{Im}\beta^T,
\qquad
\operatorname{Im}\alpha
\subseteq\operatorname{Im}\gamma^T.
\]
Hence $\bar i$ and $\bar j$ are well-defined and surjective, with
\[
\ker\bar i=
\frac{\operatorname{Im}\beta^T}
     {\operatorname{Im}\alpha^T},
\qquad
\ker\bar j=
\frac{\operatorname{Im}\gamma^T}
     {\operatorname{Im}\alpha}.
\]
Consider the homomorphism
\[
\begin{aligned}
f\colon\operatorname{coker}\gamma
&\longrightarrow\ker\bar i,\\
\m+\operatorname{Im}\gamma
&\longmapsto
\beta^T(\m)+\operatorname{Im}\alpha^T.
\end{aligned}
\]
It is well-defined because
$\beta^T\circ\gamma=\alpha^T$, and it is surjective by the
description of $\ker\bar i$. If $f([\m])=0$, then there exists
$\m'\in M$ such that
\[
\beta^T(\m)=\alpha^T(\m')
         =\beta^T(\gamma(\m')).
\]
Since $\beta^T$ is injective, $\m=\gamma(\m')$.
Thus $[\m]=0$ in $\operatorname{coker}\gamma$, proving that
$f$ is an isomorphism.
Similarly,
\[
\begin{aligned}
g\colon\operatorname{coker}\beta
&\longrightarrow\ker\bar j,\\
    \n+\operatorname{Im}\beta
&\longmapsto
\gamma^T(\n)+\operatorname{Im}\alpha
\end{aligned}
\]
is an isomorphism, by the identity
$\gamma^T\circ\beta=\alpha$ and the injectivity of $\gamma^T$.

Every finite abelian group is
abstractly isomorphic to its dual. Together with $f$ and $g$,
this gives
\[
\ker\bar i\cong G^\circ,
\qquad
\ker\bar j\cong G.
\]
Taking orders in the first exact sequence yields
\[
|G_Q|=|\ker\bar i|\,|G|=|G^\circ|\,|G|.
\]

We next consider the pairing in the statement:
\[
\begin{aligned}
\langle\cdot,\cdot\rangle_\alpha\colon
G_{Q^\circ}\times G_Q
&\longrightarrow\mathbb Q/\mathbb Z,\\
([\n],[\m])
&\longmapsto
\langle \m,\alpha_{\mathbb Q}^{-1}(\n)\rangle+\mathbb Z,
\end{aligned}
\]
where $[\n]$ and $[\m]$ denote the classes modulo
$\operatorname{Im}\alpha$ and $\operatorname{Im}\alpha^T$,
respectively, and
$\alpha_{\mathbb Q}=\alpha\otimes_{\mathbb Z}\mathbb Q$. The brackets on the right denote the natural evaluation
pairing between $M$ and $N_{\mathbb Q}$.
Thus the subscript $\alpha$ distinguishes the induced pairing
on the cokernels from this evaluation pairing.

Replacing $\n$ by $\n+\alpha(\n_0)$ changes its value by
$\langle \m,\n_0\rangle$, while replacing $\m$ by
$\m+\alpha^T(\m_0)$ changes its value by
$\langle \m_0,\n\rangle$. Both are integers, so the pairing
is well-defined.
If $[\n]$ pairs trivially with every $[\m]\in G_Q$, then
\[
\forall\,\m\in M\quad\langle \m,\alpha_{\mathbb Q}^{-1}(\n)\rangle\in\mathbb Z
\]
By the duality between $M$ and $N$, this implies
$\alpha_{\mathbb Q}^{-1}(\n)\in N$, hence
$\n\in\operatorname{Im}\alpha$. The analogous argument applies
to the other factor. Since both groups are finite, the pairing
is perfect.
Recall that
\[
\ker\bar i=
\frac{\operatorname{Im}\beta^T}
     {\operatorname{Im}\alpha^T},
\qquad
\ker\bar j=
\frac{\operatorname{Im}\gamma^T}
     {\operatorname{Im}\alpha}.
\]
Hence their elements can be represented by $\beta^T(\m)$
and $\gamma^T(\n)$, respectively, with $\m\in M$ and $\n\in N$.
Extending the factorization
$\alpha=\gamma^T\circ\beta$ to the rational vector spaces gives
\[
\alpha_{\mathbb Q}
=(\gamma^T)_{\mathbb Q}\circ\beta_{\mathbb Q},
\qquad
\alpha_{\mathbb Q}^{-1}\circ(\gamma^T)_{\mathbb Q}
=\beta_{\mathbb Q}^{-1}.
\]
Using this identity and the defining property of the dual
homomorphism $\beta^T$, we obtain
\[
\begin{aligned}
\big\langle[\gamma^T(\n)],[\beta^T(\m)]\big\rangle_\alpha
&=
\big\langle
\beta^T(\m),\alpha_{\mathbb Q}^{-1}(\gamma^T(\n))
\big\rangle+\mathbb Z\\
&=
\big\langle
\beta^T(\m),\beta_{\mathbb Q}^{-1}(\n)
\big\rangle+\mathbb Z\\
&=
\big\langle
\m,\beta_{\mathbb Q}(\beta_{\mathbb Q}^{-1}(\n))
\big\rangle+\mathbb Z\\
&=
\langle \m,\n\rangle+\mathbb Z\\
&=0\qquad\text{in }\mathbb Q/\mathbb Z,
\end{aligned}
\]
since $\langle \m,\n\rangle\in\mathbb Z$.
Therefore $\ker\bar i$ and $\ker\bar j$ are orthogonal.
Since the pairing is perfect and
\[
|\ker\bar i|\,|\ker\bar j|
=|G_Q|=|G_{Q^\circ}|,
\]
they are mutual annihilators:
\[
(\ker\bar i)^\perp=\ker\bar j,
\qquad
(\ker\bar j)^\perp=\ker\bar i.
\]
Finally, the definitions of $\bar i'$ and $\bar j'$ give
\[
\ker\bar j'=\bar T(\ker\bar j),
\qquad
\ker\bar i'=\bar T^{-1}(\ker\bar i).
\]
Both homomorphisms are surjective, so the remaining exact
sequences and quotient identifications follow. For the chosen
isomorphism $s\colon G\xrightarrow{\sim}\ker\bar j$, one has
$$\iota(G)=\bar T(\ker\bar j)$$
proving the last assertion.
\end{proof}

\subsection{Bounding the multiplicity of a Fano toric variety}
We are now in a position to propose a first bound on the multiplicity of a Fano toric variety. The following is a first generalization of the Conrads' result Theorem~\ref{thm:Conrads} for Gorenstein (hence Fano) fake wps, that is the case $r=1$.

\begin{corollary}\label{cor:Conrads}
Let $X$ be a Fano toric variety and $\pi:Y\twoheadrightarrow X$ its universal 1-covering. Then $X\cong Y/G$ with $G\leq G_Q$, where $Q$ is a weight matrix of $X$ (and $Y$). In particular,
\begin{equation*}
  \mult X\,|\,g_Q
\end{equation*}
and
\begin{eqnarray*}
  \mult X=1 &\Longleftrightarrow& X\cong Y \\
  \mult X=g_Q &\Longleftrightarrow& X\cong Z
\end{eqnarray*}
where $Z=(Z^\circ)^\circ$ is the polar Fano toric variety of the universal 1-covering $Z^\circ$ of the polar Fano toric variety $X^\circ$.
\end{corollary}

\begin{remark}
  Corollary~\ref{cor:Conrads} follows immediately by \S~\ref{sssez:U1covering} and the Definition~\ref{def:mult} of multiplicity. In fact, if $V$ and $W$ are fan matrices of $X$ and $Y$ respectively, then
  \begin{equation*}
    \mult X=|\det B|\quad\text{and}\quad g_Q=|\det A|
  \end{equation*}
   where $B$ is the quotient matrix of $V$ and $A$ the quotient matrix of a fan matrix $\L$ of the polar Fano toric variety $Z$ of the universal 1-covering $Z^\circ$ of the polar toric variety $X^\circ$ of $X$. By Theorem~\ref{thm:fattorizzazione}, there exists a unique $C\in\GL_n(\Q)\cap \Z^{n,n}$, which is the quotient matrix of $V^\circ$, such that $A=C^T\cdot B$, so proving that $\det B\,|\,\det A$\,.
\end{remark}

\subsection{Weight modulus and degree of a Fano toric variety} Let $Q\in\mathbb Z^{r,m}$ be a reduced $W$-matrix, with
$m=n+r$, and let
\[
W=\mathcal G(Q)=(\mathbf w_1,\ldots,\mathbf w_m).
\]
Set $P:=\operatorname{Conv}(W)$ and define the modulus of $Q$ by
\begin{equation}\label{Q-modulo}
  |Q|:=n!\operatorname{Vol}(P)
\end{equation}
where the volume is normalized so that a fundamental
parallelepiped of the ambient lattice has volume one.

\begin{proposition}\label{prop:modulus}
The modulus $|Q|$ is independent of the choice of the
integral Gale dual $W$.
Let $\widehat\Theta\in\SF(W)$ be such that
the primitive ray generators of every maximal cone of
$\widehat\Theta$ lie in a common facet of $P$. Set
\[
\widehat{\I}:=
\left\{
I\subseteq\{1,\ldots,m\}
\ \middle|\
|I|=r,\
\langle W^I\rangle\in\widehat\Theta(n)
\right\}.
\]
Then
\[
|Q|=\sum_{I\in\widehat{\I}}|\det Q_I|.
\]
In particular, this sum is independent of the choice of
$\widehat\Theta$ among the fans satisfying the above condition.
\end{proposition}

\begin{proof}
Any two integral Gale duals of $Q$ differ by left
multiplication by a unimodular matrix. The corresponding
polytopes therefore have the same lattice volume, proving
the first assertion.

Since $W$ is $F$-complete, the origin lies in the interior
of $P$. For $I\in\widehat{\I}$, set
\[
\sigma_I:=\langle W^I\rangle,
\qquad
\Delta_I:=
\operatorname{Conv}
\bigl(\{\mathbf0\}\cup\{\mathbf w_j\mid j\notin I\}\bigr).
\]
We claim that
\[
\sigma_I\cap P=\Delta_I.
\]
Indeed, by hypothesis there is a facet $F_I$ of $P$
containing all the generators $\mathbf w_j$, $j\notin I$.
Choose a supporting linear form $\ell_I$ such that
\[
\ell_I\leq1\text{ on }P,
\qquad
\ell_I=1\text{ on }F_I.
\]
If
\[
\mathbf x=\sum_{j\notin I}a_j\mathbf w_j
\in\sigma_I\cap P,
\qquad a_j\geq0,
\]
then
\[
\sum_{j\notin I}a_j=\ell_I(\mathbf x)\leq1,
\]
so $\mathbf x\in\Delta_I$. The reverse inclusion follows
from convexity and the inclusion $\mathbf0\in P$.

The completeness of $\widehat\Theta$ now implies that
the simplices $\Delta_I$ cover $P$ and have pairwise
disjoint interiors. Consequently,
\[
n!\operatorname{Vol}(P)
=
\sum_{I\in\widehat{\I}}
n!\operatorname{Vol}(\Delta_I)
=
\sum_{I\in\widehat{\I}}|\det W^I|.
\]
Since $W=\mathcal G(Q)$ is a $CF$-matrix, integral Gale
duality gives
\[
|\det W^I|=|\det Q_I|
\]
for every $I\in\widehat{\I}$; see \cite[Cor.~3.3]{RT-LA&GD}.
This proves the asserted formula.
\end{proof}

  \begin{remark}\label{rem:r=1 Q-modulo}
    If the rank $r=1$ then $Q=(q_1,q_2,\ldots,q_{n+1})$ is a reduced and well-formed weight system of the wps $\P(q_1,q_2,\ldots,q_{n+1})$ and
    \begin{equation*}
      |Q|=\sum_{i=1}^{n+1} q_i
    \end{equation*}
   Indeed, writing $W=(\mathbf w_1,\ldots,\mathbf w_{n+1})$,
the relation
\[
\sum_{i=1}^{n+1}q_i\mathbf w_i=\mathbf0,
\qquad q_i>0,
\]
together with $\operatorname{rk}W=n$, shows that $P$
is an $n$-dimensional simplex containing the origin
in its interior. Its facets are
\[
F_i=\operatorname{Conv}
(\mathbf w_1,\ldots,\widehat{\mathbf w_i},\ldots,\mathbf w_{n+1}),
\qquad 1\leq i\leq n+1.
\]
The fan over these facets satisfies the required condition,
and its index set consists of the singletons $\{i\}$.
Hence
\[
|Q|
=
\sum_{i=1}^{n+1}|\det Q_{\{i\}}|
=
\sum_{i=1}^{n+1}q_i.
\]
  \end{remark}

  \begin{remark}
    Let $X$ be an $n$-dimensional Fano toric variety. The self-intersection of the anticanonical divisor is usually called the \emph{degree of $X$},  that is
     \begin{equation*}
       \deg X:= (-K_X)^n
     \end{equation*}
      It is also well known (see the Corollary at page 111 in Fulton's book \cite{Fulton}) that, for a nef Cartier divisor $D$ on a complete toric variety,
the top self-intersection is given by
$D^n=n!\operatorname{Vol}(\Delta_D)$. In particular, this fact applies to computing the degree of $X$, namely
    \begin{equation}\label{autointersezione}
      \deg X=(-K_X)^n= n! \Vol\left(\De_{-K_X}\right)
    \end{equation}
    Let $V$ be a fan matrix of $X$. Then, recalling (\ref{polarità}),  $\D_{-K_X}=\conv(V^\circ)$, so that
    \begin{equation}\label{autointersezione2}
      (-K_X)^n=n! \Vol\left(\conv(V^\circ)\right)
    \end{equation}
  \end{remark}

  We can now state and prove the following bounds relating multiplicity, weight order and degree of a Fano toric variety, its polar Fano and their universal 1-coverings.

  \begin{theorem}\label{thm:Fano-bounds}
    Let $X$ be a Fano toric variety, $V$ its fan matrix and $Q=\G(V)$ its weight matrix. Let $\pi:Y\twoheadrightarrow X$ be the universal 1-covering of $X$ and $\pi^\circ:Z^\circ\twoheadrightarrow X^\circ$ be the universal 1-covering of the polar Fano toric variety $X^\circ$. Recall notation and factorizations introduced in Theorem~\ref{thm:fattorizzazione} and let $Q^\circ=\G(V^\circ)$ be the weight matrix of the polar variety $X^\circ$. Then:
    \begin{enumerate}
      \item $\mult X=|\det B|= \displaystyle{(-K_{X^\circ})^n\over |Q|}=\displaystyle{\deg X^\circ\over |Q|}$
      \item $\mult(X^\circ)=|\det C|= \displaystyle{(-K_X)^n\over |Q^\circ|}=\displaystyle{\deg X\over |Q^\circ|}$
      \item $g_Q=|\det A|=\mult X\cdot\mult(X^\circ)= \displaystyle{(-K_Y)^n\over |Q^\circ|}=\displaystyle{(-K_{Z^\circ})^n\over |Q|}$
    \end{enumerate}
  \end{theorem}

  As a consequence, the previous Corollary~\ref{cor:Conrads} can be improved as follows

  \begin{corollary}\label{cor:bounds}
    Same notation as above, then:
    \begin{enumerate}
      \item $\mult X\,\left|\,\displaystyle{(-K_Y)^n\over |Q^\circ|}\right.=\displaystyle{(-K_{Z^\circ})^n\over |Q|}$ and equality holds if and only if $X\cong Z$,
      \item $\mult X^\circ\,\left|\,\displaystyle{(-K_Y)^n\over |Q^\circ|}\right.=\displaystyle{(-K_{Z^\circ})^n\over |Q|}$ and equality holds if and only if $X^\circ\cong Y^\circ$,
      \item $|Q^\circ|\quad |\quad (-K_X)^n\quad |\quad (-K_Y)^n=g_Q\cdot |Q^\circ|\ $ and
      \begin{eqnarray*}
        (-K_X)^n = |Q^\circ| &\Longleftrightarrow& X\cong Z\\
        (-K_X)^n = g_Q\cdot|Q^\circ| &\Longleftrightarrow& X\cong Y
      \end{eqnarray*}
      \item $|Q|\quad|\quad(-K_{X^\circ})^n\quad|\quad(-K_{Z^\circ})^n=g_Q\cdot |Q|\ $ and
      \begin{eqnarray*}
        (-K_{X^\circ})^n = |Q| &\Longleftrightarrow& X^\circ\cong Y^\circ\\
        (-K_{X^\circ})^n = g_Q\cdot|Q| &\Longleftrightarrow& X^\circ\cong Z^\circ
      \end{eqnarray*}
      \item $(-K_X)^n = |Q^\circ|\ \Longleftrightarrow\ (-K_{X^\circ})^n = g_Q\cdot|Q|$
      \item $(-K_X)^n = g_Q\cdot|Q^\circ|\ \Longleftrightarrow\ (-K_{X^\circ})^n = |Q|$
    \end{enumerate}
  \end{corollary}

  \begin{remark}
    Notice that item (1) in Corollary~\ref{cor:bounds} generalizes the Conrads' result Theorem~\ref{thm:Conrads}: in fact, for $r=1$, one has $Q^\circ=Q$ \cite[Lemma~5.3]{Conrads}. Then, for $r=1$ one obtains:
    \begin{equation*}
      \mult X\ \left|\ {(-K_Y)^n\over \sum_{i=0}^n q_i}\right.
    \end{equation*}
    where $Y=\P(q_0,\ldots,q_n)$. In this case, recalling that $(-K_Y)^n=|Q|^n/\prod_{i=0}^n q_i$, item (3) in Theorem~\ref{thm:Fano-bounds} gives the following explicit computation of the weight order
    \begin{equation*}
      g_Q={|Q|^{n-1}\over \prod_{i=0}^n q_i}={\left(\sum_{i=0}^n q_i\right)^{n-1}\over \prod_{i=0}^n q_i}
    \end{equation*}
  \end{remark}

  \begin{proof}[Proof of Thm.~\ref{thm:Fano-bounds}] Item (1). The first equality follows from (\ref{molteplicità}). Moreover, by \cite[Thm.~2.4, Prop.~3.1]{RT-LA&GD} one has
  \begin{equation}\label{multX}
    \mult X=|\Tors(\Cl(X))|=[\Z^n:\Ls_c(V)]
  \end{equation}
  On the other hand, by applying (\ref{autointersezione2}) to the polar Fano toric variety $X^\circ$ one gets
  \begin{equation}\label{autointersezione duale}
    (-K_{X^\circ})^n=n! \Vol\left(\conv(V)\right)
  \end{equation}
  By Gale duality and recalling the definition (\ref{Q-modulo}) of the modulus $|Q|$
  \begin{equation*}
    W=\G(\G(V))\ \Longrightarrow\ \Vol\left(\conv(V)\right)=[\Z^n:\Ls_c(V)]\Vol\left(\conv(W)\right)
  \end{equation*}
  Then
  \begin{equation}\label{VolV}
    n!\Vol(\conv(V))=[\Z^n:\Ls_c(V)]\,|Q|
  \end{equation}
Therefore, (\ref{multX}), (\ref{VolV}) and then (\ref{autointersezione duale}), give
  \begin{equation*}
    \mult X=[\Z^n:\Ls_c(V)]={n!\Vol(\conv(V))\over |Q|}={(-K_{X^\circ})^n\over |Q|}
  \end{equation*}
  Item (2) is obtained in the same way, by exchanging the roles of $X$ and $X^\circ$.\\
  Item (3). The first equality is the definition of the weight order $g_Q$, as
  \begin{equation*}
    g_Q=|\coker(\a^T)|=|\det A|
  \end{equation*}
  Then, recalling that $A=C^T\cdot B$, by Theorem~\ref{thm:fattorizzazione}, the previous items (1) and (2) give the second equality. Last two equalities are obtained recalling that
  \begin{equation*}
    |G_Q|=\mult(Z)\quad\text{and}\quad|G_{Q^\circ}|=\mult(Y^\circ)
  \end{equation*}
  Then, Proposition~\ref{prop:isoGQ} and item (2) applied to $Y^\circ$ and item (1) applied to $Z$, end up the proof.
  \end{proof}

  \begin{proof}[Proof of Cor.~\ref{cor:bounds}]
    (1) and (2) follow immediately from Theorem~\ref{thm:Fano-bounds}~(3). The first divisibility in (3) comes from Theorem~\ref{thm:Fano-bounds}~(2), while the second one is still obtained by Theorem~\ref{thm:Fano-bounds}~(3). The same for (4), which follows by Theorem~\ref{thm:Fano-bounds}~(1) and then (3). Finally (5) and (6) are immediate consequences of previous items (3) and (4).
  \end{proof}

  We are now in a position to give a generalization of Theorem~\ref{thm:AKLN}.

  In the following statement, $\{s_n\}$ denotes the Sylvester sequence, iteratively defined by setting
  \begin{equation}\label{sylvester}
    s_1:=2\ ,\quad\forall\,n\ge 2\quad s_n:=\prod_{i=1}^{n-1}s_i +1
  \end{equation}
  Moreover, $[\,\cdot\,]$ denotes the integer part of a rational number.

  \begin{theorem}\label{thm:bounds}
Let $X$ be an $n$-dimensional Fano toric variety of rank $r$,
with $n\geq2$, and set
$
r':=\max\{r,r^\circ\},
$
where $r^\circ$ is the rank of the polar partner $X^\circ$.
Then:
\begin{enumerate}
\item If $n=2$, then
\[
\operatorname{mult}X
\leq
\left[\frac{9}{2+r'}\right].
\]
Recalling the degree classification of Fano toric surfaces
(see e.g.\cite[\S~8.3]{CLS} and consider the dual picture of Fig.~2) equality holds if and only if
one of the following occurs:
\begin{itemize}
\item
$r'=1$ and
$
X\cong\mathbb P^2/(\mathbb Z/3\mathbb Z),
$
corresponding to item~3 of the classification;
\item
$r'=2$ and
$
X\cong
(\mathbb P^1\times\mathbb P^1)/(\mathbb Z/2\mathbb Z),
$
corresponding to item~4a of the classification;
\item
$r'\geq3$ and $X$ is isomorphic to one of the surfaces
in items~6b, 5a and~7a when $r'=3$, or to the surface
in item~6a when $r'=4$.
\end{itemize}

\item If $n=3$, then
\[
\operatorname{mult}X
\leq
\left[\frac{72}{3+r'}\right].
\]

\item If $n\geq4$, then
\[
\operatorname{mult}X
\leq
\left[\frac{2(s_n-1)^2}{n+r'}\right].
\]
\end{enumerate}
\end{theorem}

  This result is based on the following

  \begin{theorem}[Thm.~4.6 in \cite{Kasprzyk10}]\label{thm:KX3} Let $X$ be a 3-dimensional, $\Q$-Fano, toric variety of index $k$, with at worst canonical singularities. Then $$(-kK_X)^3\le 72\,k^3$$
  If $(-K_X)^3= 72$ then $X$ is isomorphic to $\P(1,1,1,3)$ or $\P(1,1,4,6)$.
  \end{theorem}

  \begin{theorem}[Cor.~1.3 in \cite{BKN}]\label{thm:KXn} Let $X$ be an $n$-dimensional, $\Q$-Fano toric variety of index $k$, with $n\ge 4$ and at worst canonical singularities. Then
\begin{equation*}
  (-kK_X)^n \le 2\,k^n\,(s_n -1)^2
\end{equation*}
with equality if and only if $X\cong\P\left(1,1,{2(s_n-1)\over s_1},\ldots,{2(s_n-1)\over s_{n-1}}\right)$\,.
  \end{theorem}

\begin{lemma}\label{lem:volumeverticesfacets}
Let $N$ be a lattice of rank $n\geq2$, and let
$P\subseteq N_{\mathbb R}$ be an $n$-dimensional lattice
polytope such that $\mathbf0\in\operatorname{int}(P)$.
Denote by $f_0(P)$ and $f_{n-1}(P)$ the numbers of vertices
and facets of $P$, respectively. Then
\[
n!\operatorname{Vol}(P)
\geq \max\{f_0(P),f_{n-1}(P)\},
\]
where $\operatorname{Vol}$ is measured so that a fundamental
parallelepiped of $N$ has volume one.
\end{lemma}

\begin{proof}
Choose a triangulation $\mathcal T$ of $\partial P$
using only the vertices of $P$, for instance by restricting
a pulling triangulation of $P$ to its boundary.
Let $t$ be the number of maximal simplices of $\mathcal T$.
Since $\mathbf0\in\operatorname{int}(P)$, coning these
simplices from $\mathbf0$ gives a triangulation of $P$
into $t$ non-degenerate lattice simplices.
Each of them has positive integral normalized volume.
Therefore
\[
n!\operatorname{Vol}(P)\geq t.
\]
Every facet of $P$ contains at least one maximal simplex
of $\mathcal T$, and each maximal simplex is contained
in exactly one facet of $P$. Hence
\[
t\geq f_{n-1}(P).
\]
To bound the number of vertices, observe that every vertex
of $\mathcal T$ belongs to at least $n$ maximal simplices.
Indeed, its link is a triangulated $(n-2)$-sphere, which
has at least $n$ facets.
For $n=2$, this link consists of two points.
For $n\geq3$, choose a facet of the link. Each of its
$n-1$ codimension-one faces belongs to exactly one further
facet, and these further facets are pairwise distinct.
Together with the chosen facet, they give at least $n$
facets of the link.
Now count incidences between vertices and maximal simplices
of $\mathcal T$. Each maximal simplex has exactly $n$
vertices, while every vertex belongs to at least $n$
maximal simplices. Consequently,
\[
nt\geq n f_0(P),
\]
and therefore $t\geq f_0(P)$.
Combining the preceding inequalities yields
\[
n!\operatorname{Vol}(P)
\geq t
\geq\max\{f_0(P),f_{n-1}(P)\}.
\]
\end{proof}

\begin{proof}[Proof of Thm.~\ref{thm:bounds}] Let $Y$ and $Z^\circ$ be the universal $1$-coverings
of $X$ and $X^\circ$, respectively, with fan matrices
\[
W=\mathcal G(Q),
\qquad
\Lambda^\circ=\mathcal G(Q^\circ).
\]
By Corollary~\ref{cor:bounds}~(1) we know that
\begin{equation*}
  \mult X\,|\,\displaystyle{(-K_Y)^n\over |Q^\circ|}={(-K_{Z^\circ})^n\over |Q|}
\end{equation*}
Recall the volume expressions for the weight moduli:
\[
|Q|
=n!\operatorname{Vol}\bigl(\operatorname{Conv}(W)\bigr),
\qquad
|Q^\circ|
=n!\operatorname{Vol}
       \bigl(\operatorname{Conv}(\Lambda^\circ)\bigr).
\]
The polytopes $\operatorname{Conv}(W)$ and
$\operatorname{Conv}(\Lambda^\circ)$ are linearly
isomorphic to $\operatorname{Conv}(V)$ and
$\operatorname{Conv}(V^\circ)$, respectively.
Since $X$ and $X^\circ$ are Fano, their fans are the
face fans of these latter polytopes. Consequently,
\[
\begin{aligned}
f_0\bigl(\operatorname{Conv}(W)\bigr)
&=n+r,
&
f_{n-1}\bigl(\operatorname{Conv}(W)\bigr)
&=n+r^\circ,\\
f_0\bigl(\operatorname{Conv}(\Lambda^\circ)\bigr)
&=n+r^\circ,
&
f_{n-1}\bigl(\operatorname{Conv}(\Lambda^\circ)\bigr)
&=n+r.
\end{aligned}
\]
Both are lattice polytopes containing $\mathbf0$ in
their interiors. Applying
Lemma~\ref{lem:volumeverticesfacets} therefore gives
\[
|Q|\geq n+r',
\qquad
|Q^\circ|\geq n+r'.
\]
Hence
\begin{equation}\label{eq:boundmultFano}
\operatorname{mult}X
\leq
\min\left\{
\frac{(-K_{Z^\circ})^n}{n+r'},
\frac{(-K_Y)^n}{n+r'}
\right\}.
\end{equation}

Assume $n=2$. By the degree classification of Fano
toric surfaces,
\[
(-K_{Z^\circ})^2\leq9,
\qquad
(-K_Y)^2\leq9.
\]
Substituting in \eqref{eq:boundmultFano} and taking
integer parts proves the inequality in~(1).
The cases of equality follow from the same classification.

Assume $n=3$. Both $Y$ and $Z^\circ$ are Fano toric
varieties and, being Gorenstein, have at worst canonical
singularities. Applying Theorem~\ref{thm:KX3} gives
\[
(-K_{Z^\circ})^3\leq72,
\qquad
(-K_Y)^3\leq72.
\]
Then \eqref{eq:boundmultFano}, followed by taking integer
parts, proves~(2).

Assume $n\ge 4$ and apply the upper bound given by Theorem~\ref{thm:KXn} to $Y$ and $Z^\circ$ so that
\[
(-K_{Z^\circ})^n\leq2(s_n-1)^2,
\qquad
(-K_Y)^n\leq2(s_n-1)^2.
\]
Substituting in \eqref{eq:boundmultFano} and taking
integer parts proves~(3).
\end{proof}

\begin{remark}\label{rem:sharpness}
  Upper bounds given by Theorem~\ref{thm:bounds} are sharp if $n=2$, as the reader can easily check by the degree classification of Fano toric surfaces \cite[\S~8.3, Fig.~2]{CLS}. But for $n\ge 3$ the given upper bounds are not generally sharp: for instance, if $n=3$ and $r=1$ then Thm.~\ref{thm:bounds}~(2) gives
   \begin{equation*}
    \mult X\le \left[{72\over 4}\right] =18
   \end{equation*}
   while Thm.~\ref{thm:AKLN}~(i) gives $\mult X\le 16$, which is better.

    Upper bounds given in Thm.~\ref{thm:bounds} can be, possibly, improved, although it is not clear how to perform such an improvement. In fact, inequalities given by Theorems~\ref{thm:KX3} and \ref{thm:KXn} are sharp, in the sense that equalities are obtained in the stated cases. Moreover, also the lower bound given by Lemma~\ref{lem:volumeverticesfacets} is sharp, in the sense that equality can be attained, for instance by the modulus of varieties considered in Example~\ref{ex:QFanoCanonica2}, where $|Q|=6=n+r'$. What is not sharp is their ratio!

  Anyway, Theorem~\ref{thm:bounds} gives upper bounds stratified over the rank $r$.
  In principle, Theorem~\ref{thm:bounds} implies that
  \begin{equation}\label{limite}
    \forall\, n\ge 2\,,\ \exists\, r_0\gg 1 : \forall\, r\ge r_0\quad \mult X=1
  \end{equation}
  which is precisely what is experienced from examples, although the high non-sharpness of those upper bounds gives values of $r_0$ that likely can never be achieved by any Fano toric variety of dimension $n\ge 3$. For instance:
  \begin{itemize}
    \item if $n=3$ then $r_0=34$, by item (2) of Theorem~\ref{thm:bounds},
    \item if $n=4$ then $r_0=1761$, by item (3) of the same theorem.
  \end{itemize}
Actually, in the $\Q$-factorial case, C.~Casagrande in \cite{Casagrande06} proved that $r\le 2n$ and equality can be reached if $n$ is even: in particular if $n=3$ then $r\le 5\ll 34$ and if $n=4$ then $r\le 8\ll 1761$. Also dropping $\Q$-factoriality, answers in dimension $\le 4$ can be given by the classification of reflexive polytopes by Kreuzer-Skarke, via PALP \cite{KS}, \cite{PALP}:
 \begin{itemize}
   \item for $n=3$ the maximum number of vertices of a reflexive polytope is 14 and it is reached by the unique polytope $P=\conv(V)$ with
       \begin{equation*}
         V=\left(
             \begin{array}{cccccccccccccc}
               1 & -1 & 0 & 0 & 0 & 0 & 1 & -1 & 1 & -1 & 0 & 0 & 1 & -1 \\
               0 & 0 & 1 & -1 & 0 & 0 & 1 & -1 & 0 & 0 & 1 & -1 & 1 & -1 \\
               0 & 0 & 0 & 0 & 1 & -1 & 0 & 0 & 1 & -1 & 1 & -1 & 1 & -1 \\
             \end{array}
           \right)
       \end{equation*}
       then $r=r'=11\ll 34$; the Fano toric variety $X$ determined by the fan spanned by $P$ has $\mult X=1$ satisfying the upper bound $\mult X\le \displaystyle\left[72\over 14\right]=5$;
   \item for $n=4$ the maximum number of vertices of a reflexive polytope is 36 and it is reached by the unique polytope $H\times H$, where $H$ is the hexagon whose vertices span the fan of the blow up of $\P^2$ in three points: then $r=r'=32\ll 1761$; also in this case $\mult X=1$, where $X$ is the Fano toric variety determined by the fan spanned by $H\times H$, satisfying the upper bound $\mult X\le \displaystyle{\left[2\cdot42^2\over 36\right]}=98$.
 \end{itemize}
 One way to obtain useful information by Theorem~\ref{thm:bounds} is by producing finite lists of upper bounds, one for any dimension. For instance, when $n=2$ the list is already given in item (1), while if $n=3$ then item (2) of Theorem~\ref{thm:bounds} is equivalent to the following finite list of inequalities:
  \begin{itemize}
    \item if $r'=2$ then $\mult X\le \displaystyle{\left[{72\over 5}\right]}= 14$,
    \item if $r'=3$ then $\mult X\le \displaystyle{{72\over 6}}= 12$,
    \item if $r'=4$ then $\mult X\le \displaystyle{\left[{72\over 7}\right]}= 10$,
    \item if $r'=5$ then $\mult X\le \displaystyle{{72\over 8}}= 9$,
    \item if $r'=6$ then $\mult X\le \displaystyle{{72\over 9}}= 8$,
    \item if $r'=7$ then $\mult X\le \displaystyle{\left[{72\over 10}\right]}= 7$,
    \item if $r'=8$ then $\mult X\le \displaystyle{\left[{72\over 11}\right]}= 6$,
    \item if $r'=9$ then $\mult X\le \displaystyle{{72\over 12}}= 6$,
    \item if $r'=10$ then $\mult X\le \displaystyle{\left[{72\over 13}\right]}= 5$,
    \item if $r'=11$ then $\mult X\le \displaystyle{\left[{72\over 14}\right]}= 5$.
  \end{itemize}
\end{remark}

\section{A classification of Fano toric varieties and their polar partners}\label{sez:Fano}

Results developed in the previous \S~\ref{sez:bounds} allows us to start from a Fano weight matrix $Q$ to get a complete topological classification of Fano toric varieties admitting $Q$ as a weight matrix and their polar Fano toric varieties, hence admitting $Q^\circ$ as a weight matrix.

\begin{remark}\label{rem:QFano iff}
  A first obvious consideration is that:
  \begin{itemize}
    \item \emph{$Q$ is a Fano weight matrix if and only if $Q^\circ$ is a Fano weight matrix}
  \end{itemize}
  This is a direct consequence of Batyrev's reflexivity of polytopes $\conv(\G(Q))$ and $\conv(\G(Q)^\circ)$, recalling that by definition, $Q^\circ=\G(\G(Q)^\circ)$. In fact, recalling Definition~\ref{def:Fano-Gorenstein_Q} and construction exhibited in \S~\ref{ssez:1-cov_polar}, choose Fano toric varieties $Y=Y(\Theta)$, with fan matrix $W=\G(Q)$, and $Z^\circ=Z^\circ(\Xi^\circ)$, with fan matrix $\L^\circ:=\G(Q^\circ)$. Then:
  \begin{eqnarray*}
    Q\ \text{is Fano} &\Longrightarrow & Y\ \text{is Fano}\ \stackrel{\text{Reflexivity}}{\Longrightarrow}\ Y^\circ\ \text{is Fano}\ \stackrel{\text{Prop.~\ref{prop:Univ1covGorestein}}}{\Longrightarrow}\ Q^\circ\ \text{is Fano} \\
    Q^\circ\ \text{is Fano} & \Longrightarrow & Z^\circ\ \text{is Fano}\ \stackrel{\text{Reflexivity}}{\Longrightarrow}\ Z\ \text{is Fano}\ \stackrel{\text{Prop.~\ref{prop:Univ1covGorestein}}}{\Longrightarrow}\ Q\ \text{is Fano}
  \end{eqnarray*}
\end{remark}

\begin{theorem}\label{thm:classificazione}
Let $Q$ be a Fano weight matrix and let $Y=Y(\Theta)$ have
fan matrix $W=\mathcal G(Q)$. Let $Z^\circ\to Y^\circ$
be the universal $1$-covering, and set $Z=(Z^\circ)^\circ$.
The groups $G_Q$ and $G_{Q^\circ}$ admit compatible actions
on $Y$ and $Z^\circ$, with quotients $Z$ and $Y^\circ$,
respectively, for which the following assertions hold.
\begin{enumerate}
\item
The Fano toric varieties admitting $Q$ as a weight matrix
are precisely the quotients $X=Y/H$, with $H\leq G_Q$.
The map $Y\to X$ is the universal $1$-covering,
with deck transformation group $H$, and
\[
\pi_1^{\acute{e}t}(X)^{(1)}\cong H,\qquad
\mult X=|H|.
\]
Subgroups are in bijection with intermediate factorizations
$Y\to X\to Z$ of $p:Y\to Z$, up to isomorphisms commuting
with the $1$-covering maps.

\item
Polar duality induces an inclusion-reversing bijection
$H\mapsto H^\circ$ between the subgroup lattices of
$G_Q$ and $G_{Q^\circ}$, such that
\[
(Y/H)^\circ\cong Z^\circ/H^\circ,\qquad
H^\circ\cong G_Q/H.
\]
The map $Z^\circ\to(Y/H)^\circ$ is the universal
$1$-covering, with deck transformation group $H^\circ$.
In particular,
\[
\mult X\,\mult(X^\circ)=g_Q=\mult Z.
\]

\item
The fan $\Theta$ is uniquely determined: it is the face
fan of $P:=\operatorname{Conv}(W)$.
After forgetting the $1$-covering maps, toric isomorphism
classes correspond to the orbits of subgroups of $G_Q$
under the induced action of
\[
\operatorname{Aut}(N,P)
:=\{u\in\operatorname{GL}(N)\mid u(P)=P\}.
\]
Changing the lattice basis transports the construction
by a toric isomorphism.
\end{enumerate}
\end{theorem}

\begin{proof}
Since $Q$ is Fano, there exists a fan $\Theta$ over
$W=\mathcal G(Q)$ such that $Y(\Theta)$ is Fano.
The columns of $W$ generate its ambient lattice, so $Y$
is simply connected in codimension $1$.
Moreover, ampleness of $-K_Y$ implies that $\Theta$ is
the fan of cones over the faces of
$P=\operatorname{Conv}(W)$.
In particular, this fan is uniquely determined by $W$.
Theorem~\ref{thm:fattorizzazione} gives the dual universal $1$-coverings
\[
p\colon Y\longrightarrow Z,
\qquad
p^T\colon Z^\circ\longrightarrow Y^\circ.
\]
Their groups of deck transformations are canonically
\[
\operatorname{Hom}(\operatorname{coker}\alpha^T,\mathbb K^*)
=G_Q^\vee\quad\text{and}\quad
\operatorname{Hom}(\operatorname{coker}\alpha,\mathbb K^*)
=G_{Q^\circ}^\vee,
\]
respectively, where
$F^\vee:=\Hom(F,\K^*)$, for a finite abelian group $F$.
Choose a lattice isomorphism $T:N\to M$ as in
Proposition~\ref{prop:isoGQ}, with induced isomorphism
$\bar T:G_{Q^\circ}\to G_Q$, and fix a group embedding
$\varepsilon:\Q/\Z\hookrightarrow\K^*$.

Such an embedding exists because $\K$ is algebraically
closed of characteristic zero. Indeed, for each prime $p$,
choose a primitive $p$-th root of unity $\zeta_{p,1}$ and,
recursively, choose $\zeta_{p,r+1}$ with
$\zeta_{p,r+1}^{\,p}=\zeta_{p,r}$.
Then $\zeta_{p,r}$ has order $p^r$, and the compatible maps
\[
\varepsilon_p:\Z[1/p]/\Z\longrightarrow\K^*,
\qquad
\frac{a}{p^r}+\Z\longmapsto\zeta_{p,r}^{\,a},
\]
are injective. The primary decomposition
$\Q/\Z=\bigoplus_{p\ {\rm prime}}\Z[1/p]/\Z$
allows us to combine them by multiplication into
$\varepsilon$. This map is injective because roots of unity
of relatively prime orders generate subgroups with trivial
intersection. Its image is the group of all roots of unity
in $\K$. The choice is generally noncanonical; it converts
the $\Q/\Z$-valued pairing into multiplicative characters
with values in $\K^*$.

Using the pairing of Proposition~\ref{prop:splittingGQ}, define
\[
\begin{aligned}
\rho:G_Q&\longrightarrow G_Q^\vee,
&\rho(a)(a')&:=
\varepsilon\bigl(\langle\bar T^{-1}(a),a'\rangle_\alpha\bigr),\\
\rho^\circ:G_{Q^\circ}&\longrightarrow G_{Q^\circ}^\vee,
&\rho^\circ(b)(b')&:=
\varepsilon\bigl(\langle b',\bar T(b)\rangle_\alpha\bigr).
\end{aligned}
\]
The perfectness of the pairing implies that $\rho$ and
$\rho^\circ$ are isomorphisms. They therefore define the
actions used in the statement, with quotients
\[
Z\cong Y/G_Q,
\qquad
Y^\circ\cong Z^\circ/G_{Q^\circ}.
\]
For item~(1), let $H\leq G_Q$ and set $X:=Y/H$.
The quotient map factors $p$ as
\[
Y\xrightarrow{\pi}X\xrightarrow{q}Z.
\]
Both maps are finite toric $1$-coverings, since their
composition is a $1$-covering. Hence
\[
-K_X=q^*(-K_Z).
\]
Since $-K_Z$ is ample and Cartier, so is $-K_X$.
Thus $X$ is Fano.

The absence of ramification in codimension $1$ means
that the primitive ray generators are preserved by the
corresponding lattice maps. Their integral relations
are therefore unchanged, so $Q$ is a weight matrix of $X$.
Since $Y$ is simply connected in codimension $1$, $\pi$
is the universal $1$-covering of $X$.

Conversely, let $X$ be a Fano toric variety admitting
$Q$ as a weight matrix. By Proposition~\ref{prop:Univ1covGorestein}, its universal $1$-covering is
Fano and has fan matrix $\mathcal G(Q)$, up to an
integral change of basis.
By uniqueness of the Fano fan over $W$, the source of this $1$-covering is torically isomorphic to $Y$.
The factorization of Theorem~\ref{thm:fattorizzazione} then gives an
intermediate factorization
\[
Y\longrightarrow X\longrightarrow Z
\]
of $p$.
The Galois correspondence for this finite quotient
identifies $X$ with $Y/H$ for a subgroup $H\leq G_Q$.
It also gives the asserted bijection when the maps
from $Y$ and to $Z$ are retained.

For item~(2), consider the lattice factorizations
associated with $X$ and $X^\circ$.
The duals of the induced surjections
\[
\bar i\colon G_Q\longrightarrow G,
\qquad
\bar j\colon G_{Q^\circ}\longrightarrow G^\circ
\]
identify $G^\vee$ and $(G^\circ)^\vee$ with the deck
groups of
\[
Y\longrightarrow X,
\qquad
Z^\circ\longrightarrow X^\circ,
\]
inside the deck groups of $p$ and $p^T$, respectively.
The image of $G^\vee$ in $G_Q^\vee$ consists precisely
of the characters trivial on $\ker\bar i$.
By the definition of $\rho$, this gives
\[
\bar T^{-1}(H)
=
(\ker\bar i)^\perp
=
\ker\bar j,
\]
where the last equality follows from
Proposition~\ref{prop:splittingGQ}.
Consequently,
\[
H=\bar T(\ker\bar j)=\iota(G).
\]
Similarly, the image of $(G^\circ)^\vee$ in
$G_{Q^\circ}^\vee$ consists precisely of the characters
trivial on $\ker\bar j$.
Define $H^\circ$ to be its inverse image under $\rho^\circ$.
The annihilator identities give
\[
\bar T^{-1}\bigl((\ker\bar j)^\perp\bigr)
=
\bar T^{-1}(\ker\bar i)
=
\bar T^{-1}\!\left(
\bigl(\bar T^{-1}(H)\bigr)^\perp
\right)
=
H^\circ.
\]
Thus the quotient by the restricted action is exactly
\[
X^\circ\cong Z^\circ/H^\circ.
\]
Here $G=\coker\beta^T$ and $G^\circ=\coker\gamma^T$.
The descriptions of $H$ and $H^\circ$, together with
Proposition~\ref{prop:splittingGQ}, give
\[
0\longrightarrow\bar T(H^\circ)\longrightarrow G_Q
\xrightarrow{\bar i}G\longrightarrow0,
\]
\[
0\longrightarrow H\longrightarrow G_Q
\xrightarrow{\bar j'}G^\circ\longrightarrow0.
\]
Thus $G_Q/H\cong G^\circ$. Since $G^\circ$ is abstractly
isomorphic to its character group, which is the deck group
of $Z^\circ\to X^\circ$, we obtain
$H^\circ\cong G^\circ\cong G_Q/H$.
Perfectness of the pairing shows that annihilation is
an inclusion-reversing bijection, and hence the same
holds for $H\mapsto H^\circ$.
It also gives
\[
|H|\,|H^\circ|=|G_Q|=g_Q.
\]
Since the quotient maps from $Y$ and $Z^\circ$ are
universal $1$-coverings, these orders are the respective
multiplicities, proving the stated identity.

For item~(3), uniqueness of the Fano fan over $W$ was
established at the beginning of the proof.
An integral change of basis transports the fans,
lattice maps and group actions, giving torically
isomorphic quotient constructions.

Finally, every element of $\operatorname{Aut}(N,P)$
induces an automorphism of $Y$ and, by polar duality
and the universal $1$-covering construction, an automorphism
of the $1$-covering $p$.
It therefore acts on its deck group.
Subgroups in the same orbit give torically isomorphic
quotients.

Conversely, a toric isomorphism
$
Y/H\xrightarrow{\sim}Y/H'
$
induces an isomorphism of their lattices carrying
primitive ray generators to primitive ray generators.
Restricting to the sublattices generated by these
vectors gives an automorphism of the lattice $N$ of the universal $1$-covering preserving $P$.
The resulting automorphism of $Y$ carries the deck
subgroup of $Y\to Y/H$ to that of $Y\to Y/H'$.
Thus it carries $H$ to $H'$ under the fixed identification,
proving the orbit description.
\end{proof}

\begin{remark}
For a fixed Fano weight matrix $Q$, Theorem~\ref{thm:classificazione} organizes
the associated Fano toric varieties through the intermediate
$1$-coverings
\[
Y\longrightarrow X\longrightarrow Z.
\]
Once the group actions have been fixed, these factorizations
are parameterized by the subgroups $H\leq G_Q$.
The embedding of $H$ in $G_Q$, together with the specified
action on $Y$, records the quotient construction.
After forgetting the $1$-covering maps, toric isomorphism classes
are obtained by taking the orbits described in item~(3).

The topological content of this description is expressed
by the abstract isomorphisms
\[
G\cong\pi_1^{\acute{e}t}(X)^{(1)},
\qquad
G^\circ\cong\pi_1^{\acute{e}t}(X^\circ)^{(1)}.
\]
Polar duality is reflected in the perfect pairing between
$G_{Q^\circ}$ and $G_Q$ and in the mutual annihilation
of the canonical kernels:
\[
(\ker\bar i)^\perp=\ker\bar j,
\qquad
(\ker\bar j)^\perp=\ker\bar i.
\]
The corresponding exact sequences give
\[
G_Q/H\cong G^\circ,
\qquad
\operatorname{mult}X\,
\operatorname{mult}(X^\circ)=g_Q
\]
Thus $g_Q$ controls the multiplicities of all the polar
partners associated with the fixed weight matrix $Q$.

The Gorenstein condition supplies the integrality constraint
underlying this finite classification. Its role also explains
the relevance of the Gorenstein index when passing to the
$\mathbb Q$-Fano setting considered in Section~4.
\end{remark}

\begin{remark}[Comparison with Batyrev's classification]\label{rem:Batyrev}
Let
$
P:=\operatorname{Conv}(\mathcal G(Q)),
$
and let $N$ be the lattice of its Fano toric universal
$1$-covering $Y$, with dual lattice $M$.
Set
\[
L:=\left\langle\operatorname{Vert}(P^\circ)\right\rangle_{\mathbb Z}
\subseteq M
\]
and
\[
N^{\mathrm{max}}
:=
\left\{
\n\in N_{\mathbb Q}
\ \middle|\
\langle\m,\n\rangle\in\mathbb Z
\text{ for every }\m\in L
\right\}.
\]
Batyrev's theory of minimal and maximal reflexive pairs
describes the admissible ambient lattices by
\[
N\subseteq N'\subseteq N^{\mathrm{max}};
\]
see \cite[Prop.~5.2.8 and Cor.~5.2.9]{Batyrev94}.
In the notation of the present paper,
\[
G_Q\cong M/L,
\qquad
G_{Q^\circ}\cong N^{\mathrm{max}}/N.
\]
Under these identifications, the pairing of
Proposition~\ref{prop:splittingGQ} becomes evaluation modulo
$\mathbb Z$, consistently with the dual finite groups
introduced in \cite[Def.~5.3.5]{Batyrev94}.
Theorem~\ref{thm:classificazione} therefore expresses this lattice correspondence
in terms of weight groups and their subgroups.

Batyrev also studies weight systems for reflexive simplices
in \cite[\S~5.4--5.5]{Batyrev94}.
The distinctive feature of the present approach is its
organization around a reduced weight matrix $Q$ of arbitrary
rank. Integral Gale duality and polar duality recover the
universal $1$-coverings and the maximal Fano quotients,
while the resulting lattice morphisms provide explicit
cokernel presentations
\[
G_Q=\operatorname{coker}\alpha^T,
\qquad
G_{Q^\circ}=\operatorname{coker}\alpha.
\]
Together with the divisor class morphisms in Cox's
construction, these presentations give access to the
quotient actions and to the subgroup correspondence
describing polar partners.

This formulation also connects the classification with
the numerical invariants studied here. In particular,
Theorem~\ref{thm:Fano-bounds} gives
\[
\deg(X^\circ)=|Q|\,\operatorname{mult}X,
\qquad
\deg X=|Q^\circ|\,\operatorname{mult}(X^\circ).
\]
The weight data thus organize both the finite $1$-covering
structure and the degrees of the associated varieties.

The contribution of this approach lies in this explicit
integration of weight matrices, Cox quotient data, polar
duality and multiplicity. For each given Fano weight matrix,
it provides a finite construction of the associated
quotients and their polar partners, forming the basis
of the Weight Matrices Analysis proposed in Section~7.1.
\end{remark}

Theorem~\ref{thm:classificazione} can be easily checked in dimension 2 by means of the well known classification by degree of Fano toric surfaces \cite[\S~8.3, Fig.~2]{CLS}.

\begin{corollary}\label{cor:dim2}
  Setting $n=2$, up to weight equivalence, all the 2-dimensional Fano weight matrices are the following:
  \begin{itemize}
    \item $r=1$, then:
    \begin{enumerate}
      \item $Q=\left(
                 \begin{array}{ccc}
                   1 & 1 & 1 \\
                 \end{array}
               \right)=Q^\circ$ with $G_Q\cong\Z/3\Z$ and associated Fano toric surfaces given by degrees 3 and 9 in the degree classification, admitting multiplicities 3 and 1, respectively; in particular, 9 is $\P^2$ and is the universal 1-covering;
      \item $Q=\left(
                 \begin{array}{ccc}
                   1 & 1 & 2 \\
                 \end{array}
               \right)=Q^\circ$ with $G_Q\cong\Z/2\Z$ and associated Fano toric surfaces given by cases $4c$ and $8c$ in the degree classification, admitting multiplicities 2 and 1, respectively; in particular, $8c$ is $\P(1,1,2)$ and is the universal 1-covering;
      \item $Q=\left(
                 \begin{array}{ccc}
                   1 & 2 & 3 \\
                 \end{array}
               \right)=Q^\circ$ with $G_Q=\{0\}$ and the unique associated Fano toric surface $6d=\P(1,2,3)$, clearly admitting multiplicity 1;
    \end{enumerate}
    \item $r=2$, then:
    \begin{enumerate}
      \item $Q=\left(
                 \begin{array}{cccc}
                   1 & 1 & 0 & 0 \\
                   0 & 0 & 1 & 1 \\
                 \end{array}
               \right)=Q^\circ$ with $G_Q\cong\Z/2\Z$ and associated Fano toric surfaces given by cases $4a$ and $8a$ in the degree classification, admitting multiplicities 2 and 1, respectively; in particular, $8a$ is $\P^1\times\P^1$ and is the universal 1-covering;
      \item $Q=\left(
                 \begin{array}{cccc}
                   1 & 1 & 1 & 0 \\
                   0 & 0 & 1 & 1 \\
                 \end{array}
               \right)$ with $G_Q=\{0\}$ and the unique associated Fano toric surface given by $8b$, admitting multiplicity 1;
      \item $Q=\left(
                 \begin{array}{cccc}
                   1 & 1 & 1 & 1 \\
                   0 & 1 & 2 & 3 \\
                 \end{array}
               \right)$ with $G_Q=\{0\}$ and the unique associated Fano toric surface given by $4b$, admitting multiplicity 1; in particular this is the polar weight $Q^\circ$ with $Q$ as in the previous case (2);
      \item $Q=\left(
                 \begin{array}{cccc}
                   1 & 1 & 1 & 0 \\
                   0 & 1 & 2 & 1 \\
                 \end{array}
               \right)$ with $G_Q=\{0\}$ and the unique associated Fano toric surface given by $7b$, admitting multiplicity 1;
     \item $Q=\left(
                 \begin{array}{cccc}
                   1 & 1 & 2 & 0 \\
                   1 & 2 & 1 & 1 \\
                 \end{array}
               \right)$ with $G_Q=\{0\}$ and the unique associated Fano toric surface given by $5b$, admitting multiplicity 1; in particular this is the polar weight $Q^\circ$ with $Q$ as in the previous case (4);
     \item $Q=\left(
                 \begin{array}{cccc}
                   1 & 1 & 0 & 0 \\
                   0 & 1 & 1 & 2 \\
                 \end{array}
               \right)=Q^\circ$ with $G_Q=\{0\}$ and the unique associated Fano toric surface given by $6c$, admitting multiplicity 1;
    \end{enumerate}
    \item $r=3$, then
    \begin{enumerate}
      \item $Q=\left(
                 \begin{array}{ccccc}
                   1 & 1 & 0 & 0 & 0 \\
                   0 & 1 & 1 & 1 & 0 \\
                   0 & 0 & 0 & 1 & 1
                 \end{array}
               \right)$ with $G_Q=\{0\}$ and the unique associated Fano toric surface given by $7a$, admitting multiplicity 1;
      \item $Q=\left(
                 \begin{array}{ccccc}
                   1 & 1 & 0 & 0 & 0 \\
                   0 & 1 & 2& 1 & 0 \\
                   0 & 0 & 1 & 1 & 1
                 \end{array}
               \right)$ with $G_Q=\{0\}$ and the unique associated Fano toric surface given by $5a$, admitting multiplicity 1; in particular this is the polar weight $Q^\circ$ with $Q$ as in the previous case (1);
      \item $Q=\left(
                 \begin{array}{ccccc}
                   1 & 1 & 0 & 0 & 0 \\
                   0 & 1 & 1& 1 & 0 \\
                   0 & 0 & 1 & 2 & 1
                 \end{array}
               \right)=Q^\circ$ with $G_Q=\{0\}$ and the unique associated Fano toric surface given by $6b$, admitting multiplicity 1;
    \end{enumerate}
    \item $r=4$, then $Q=\left(
                           \begin{array}{cccccc}
                             1 & 1 & 0 & 0 & 0 & 0 \\
                             0 & 1 & 1 & 1 & 0 & 0 \\
                             0 & 0 & 1 & 0 & 1 &  0\\
                             0 & 0 & 0 & 1 & 0 & 1 \\
                           \end{array}
                         \right)=Q^\circ$ with $G_Q=\{0\}$ and the unique associated Fano toric surface given by $6a$, admitting multiplicity 1.
  \end{itemize}
\end{corollary}

The reader can find a further application of the Classification Theorem~\ref{thm:classificazione} in Example~\ref{ex:QFanoCanonica2}, namely, with $G_Q\cong\Z/3\Z\cong G_{Q^\circ}$ and $Y, Z$ and their polar partners $Z^\circ, Y^\circ$, with $Z\cong Y/G_Q$ and $Y^\circ\cong Z^\circ/G_{Q^\circ}$: these are 3-dimensional Fano toric varieties of rank $r=2$ and $r^\circ=3$, respectively.

\begin{example}[Blow-up at two points of $\P^3$ and small contractions]
\label{ex:blupP3}
This example illustrates Theorem~\ref{thm:classificazione}
when the weight group has nontrivial proper subgroups.

Let $\phi:\widehat X\to\P^3$ be the blow-up at two distinct
points $p,q$, which we may take to be torus invariant.
Let $l$ be the line joining them, and let $\widehat l$
be its strict transform. The anticanonical morphism gives
a small contraction $f:\widehat X\to X$ of $\widehat l$.
The resulting diagram is
\[
\xymatrix{
&\widehat X\ar[dl]_-f\ar[dr]^\phi&\\
X\ar@{-->}[rr]^\varphi&&\P^3
}.
\]
Here $\varphi=\phi\circ f^{-1}$ is a rational map:
$\phi$ does not factor through $f$ as a morphism, since
$\phi(\widehat l)=l$, whereas $f(\widehat l)$ is a point.
Moreover, from $K_{\widehat{X}}=\phi^*(K_{\P^3})+2 E_p + 2 E_q$ it follows that
\[
-K_{\widehat X}\cdot\widehat l=4-2-2=0.
\]
Thus $\widehat X$ is not Fano. As we shall see, it is weak
Fano, and $X$ is a non-$\Q$-factorial Fano threefold.
A fan matrix of $\widehat X$ is
\[
V=
\begin{pmatrix}
1&0&0&0&-1&1\\
0&1&0&0&-1&1\\
0&0&1&-1&-1&1
\end{pmatrix}.
\]
The rays generated by $\v_4$ and $\v_6$ correspond to the
exceptional divisors $E_p$ and $E_q$ of $\phi$. Writing
$\langle i,j,k\rangle=\langle\v_i,\v_j,\v_k\rangle$,
the maximal cones of the smooth complete fan of the
blow-up are
\[
\begin{aligned}
\widehat\Si(3)=\{&
\langle2,4,5\rangle,\langle2,3,5\rangle,
\langle1,4,5\rangle,\langle1,3,5\rangle,\\
&\langle1,2,4\rangle,\langle1,2,6\rangle,
\langle2,3,6\rangle,\langle1,3,6\rangle\}.
\end{aligned}
\]
A weight matrix is
\[
Q=\G(V)=
\begin{pmatrix}
1&1&1&0&1&0\\
0&0&1&1&0&0\\
0&0&0&0&1&1
\end{pmatrix}.
\]

\begin{figure}
\centering
\includegraphics[width=8truecm]{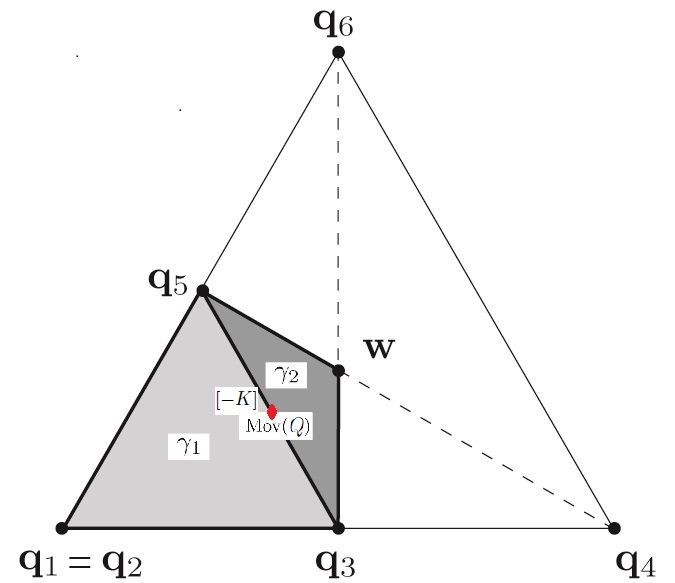}
\caption{\label{fig:mov}
Example~\ref{ex:blupP3}: a section of $\Mov(V)$ and its
chambers inside $\overline{\Eff}(V)$, cut out by
$\sum_{i=1}^3x_i=1$.}
\end{figure}

Let $\q_i$ be the $i$-th column of $Q$, and set
$\w=(1,1,1)^T$. In the notation of Fig.~\ref{fig:mov},
\[
\begin{aligned}
\overline{\Eff}(\widehat X)
&=\langle\q_1,\q_4,\q_6\rangle,\\
\Mov(\widehat X)
&=\langle\q_1,\q_3,\w,\q_5\rangle,\\
\Nef(\widehat X)
&=\g_2=\langle\q_3,\q_5,\w\rangle.
\end{aligned}
\]
The other chamber is
$\g_1=\langle\q_1,\q_3,\q_5\rangle$.
The anticanonical class satisfies
\[
[-K_{\widehat X}]
=Q\1_6=(4,2,2)^T
=2\q_3+2\q_5
\in\Relint(\g_1\cap\g_2),
\]
where
\[
\g_1\cap\g_2=\langle\q_3,\q_5\rangle.
\]
In particular, $-K_{\widehat X}$ is nef but not ample.
Then
$\widehat X$ is weak Fano.

The curve $\widehat l$ corresponds to the wall
$\langle1,2\rangle$. Its contraction replaces
$\langle1,2,4\rangle$ and $\langle1,2,6\rangle$ by their
union $\langle1,2,4,6\rangle$. Thus the fan of $X$ has
maximal cones
\begin{equation}\label{Sigma(n)}
\begin{aligned}
\Si(3)=\{&
\langle2,4,5\rangle,\langle2,3,5\rangle,
\langle1,4,5\rangle,\langle1,3,5\rangle,\\
&\langle1,2,4,6\rangle,
\langle2,3,6\rangle,\langle1,3,6\rangle\}.
\end{aligned}
\end{equation}
It is the face fan of $\conv(V)$. The integral polar
vertices displayed below show that $X$ is Gorenstein;
hence $X$ is Fano. The nonsimplicial cone
$\langle1,2,4,6\rangle$ makes it non-$\Q$-factorial.
Under the identification of class groups induced by $f$,
\[
\Nef(X)=\langle\q_3,\q_5\rangle,\qquad
\rk\Cl(X)=3,\qquad \rk\Pic(X)=2.
\]

The chamber $\g_1$ corresponds to another smooth weak
Fano threefold $\widehat X'$. Its fan is obtained by
replacing the two cones
$\langle1,2,4\rangle,\langle1,2,6\rangle$ with
$\langle1,4,6\rangle,\langle2,4,6\rangle$.
There is a second small contraction
$f':\widehat X'\to X$, and the birational map
$(f')^{-1}\circ f$ is a flop. The relation
\[
\v_1+\v_2=\v_4+\v_6
\]
gives the standard local toric flop
\cite[Example~15.5.14]{CLS}.
A common resolution is
$\widetilde X=\operatorname{Bl}_{\widehat l}\widehat X$,
whose exceptional divisor is $\P^1\times\P^1$:
\[
\xymatrix{
&\widetilde X\ar[dl]\ar[dr]&\\
\widehat X\ar[dr]_-f
\ar@{-->}[rr]^{(f')^{-1}\circ f}
&&\widehat X'\ar[dl]^{f'}\\
&X&
}.
\]
Its fan is obtained by adding the ray generated by
\[
\v_1+\v_2=\v_4+\v_6=(1,1,0)^T.
\]
We now apply Theorem~\ref{thm:classificazione} to $X$.
Since the columns of $V$ generate $N$, the class group
$\Cl(X)$ is free and $X$ is its own universal
$1$-covering. Thus $X$ plays the role of $Y$ in that
theorem, and $X^\circ$ is the maximal Fano quotient
associated with the polar weight matrix.

The polar fan matrix and its Gale dual are
\[
V^\circ=
\begin{pmatrix}
-1&-1&-1&-1&1&1&3\\
-1&1&1&3&-1&-1&-1\\
1&-1&1&-1&-1&1&-1
\end{pmatrix},
\]
\[
Q^\circ=\G(V^\circ)=
\begin{pmatrix}
1&1&1&0&0&0&1\\
0&1&0&0&0&1&0\\
0&0&1&0&1&0&0\\
1&0&0&1&1&1&0
\end{pmatrix}.
\]
A torsion matrix is
\begin{equation}\label{Gammaduale}
\Ga^\circ=
\begin{pmatrix}
\overline0&\overline0&\overline0&\overline0&
\overline1&\overline0&\overline1\\
\overline0&\overline0&\overline0&\overline0&
\overline0&\overline1&\overline1
\end{pmatrix},
\end{equation}
with entries in $\Z/2\Z$.

The universal $1$-covering
$\pi^\circ:Z^\circ\to X^\circ$ has fan matrix
\[
\L^\circ=\G(Q^\circ)=
\begin{pmatrix}
1&0&0&-1&0&0&-1\\
0&1&0&1&0&-1&-1\\
0&0&1&1&-1&0&-1
\end{pmatrix}.
\]
Indeed,
\[
V^\circ=C\L^\circ,\qquad
C=
\begin{pmatrix}
-1&-1&-1\\
-1&1&1\\
1&-1&1
\end{pmatrix}.
\]
The Smith normal form of $C$ is $\diag(1,2,2)$, so
\[
G_Q\cong G_{Q^\circ}\cong
(\Z/2\Z)^2,\qquad g_Q=4.
\]
Ordering the columns compatibly with $V$, a fan matrix
of the maximal Fano quotient $Z=(Z^\circ)^\circ$ is
\[
\L=C^T V=
\begin{pmatrix}
-1&-1&1&-1&1&-1\\
-1&1&-1&1&1&-1\\
-1&1&1&-1&-1&1
\end{pmatrix}.
\]
With this ordering, $Q\L^T=0$, and the maximal cones of
the fan $\Xi$ are given by the same index sets as
in~\eqref{Sigma(n)}, now over $\L$.
A torsion matrix of $Z$ is
\[
\Ga=
\begin{pmatrix}
\overline0&\overline1&\overline1&
\overline0&\overline1&\overline1\\
\overline0&\overline0&\overline0&
\overline1&\overline0&\overline1
\end{pmatrix}.
\]
The five subgroups of $G_Q\cong(\Z/2\Z)^2$ are
\begin{equation}\label{sottogruppi}
\begin{gathered}
\{0\},\qquad
G_1:=\langle(\overline1,\overline0)\rangle,\qquad
G_1':=\langle(\overline1,\overline1)\rangle,\\
G_2:=\langle(\overline0,\overline1)\rangle,\qquad G_Q.
\end{gathered}
\end{equation}
Exponentiating $\Ga$ gives the action on $X$, in Cox
coordinates,
\[
\begin{aligned}
(\overline a,\overline b)\cdot[x_1:\cdots:x_6]
={}&[x_1:(-1)^a x_2:(-1)^a x_3:\\
&\hspace{12mm}(-1)^b x_4:(-1)^a x_5:
(-1)^{a+b}x_6].
\end{aligned}
\]
All subgroup actions below are restrictions of this
action. The quotients by $\{0\}$ and $G_Q$ are $X$ and
$Z$, respectively.
The quotient $X_1:=X/G_1$, obtained from the first row
of $\Ga$, has fan matrix
\[
V_1=
\begin{pmatrix}
1&0&0&0&-1&1\\
0&1&1&-1&-2&2\\
0&0&2&-2&-2&2
\end{pmatrix}.
\]
The quotient $X_2:=X/G_2$, obtained from the second row,
has fan matrix
\[
V_2=
\begin{pmatrix}
1&1&0&0&-2&2\\
0&2&0&0&-2&2\\
0&0&1&-1&-1&1
\end{pmatrix}.
\]
In both cases the maximal cones are described by the
index sets in~\eqref{Sigma(n)}.
These two varieties are not isomorphic, although
\[
\mult X_1=\mult X_2=2.
\]
Indeed, the nonregular two-dimensional cones of their
fans are, respectively,
\[
\begin{array}{c|l}
X_1&
\langle1,5\rangle,\ \langle1,6\rangle,\
\langle2,3\rangle,\ \langle2,4\rangle\\
X_2&
\langle3,5\rangle,\ \langle3,6\rangle,\
\langle4,5\rangle.
\end{array}
\]
Consequently their singular loci have four and three
one-dimensional irreducible components, respectively.

For the diagonal subgroup $G_1'$, the action is given
by the sum of the two rows of $\Ga$. The quotient
$X_1':=X/G_1'$ has fan matrix
\[
V_1'=
\begin{pmatrix}
1&0&1&-1&-2&2\\
0&1&0&0&-1&1\\
0&0&2&-2&-2&2
\end{pmatrix}.
\]
This time $X_1'\cong X_1$. Explicitly,
\[
V_1'=U V_1 S,\qquad
U=
\begin{pmatrix}
1&0&-1\\
0&1&-1\\
0&0&-1
\end{pmatrix}\in\GL_3(\Z),
\]
where $S$ is the permutation matrix associated with
$(3\ 5)(4\ 6)$; this permutation preserves the
collection of index sets in~\eqref{Sigma(n)}.
Thus the Fano toric varieties admitting $Q$ as a
weight matrix are, up to isomorphism,
\[
X,\qquad X_1\cong X_1',\qquad X_2,\qquad Z,
\]
with multiplicities $1,2,2,4$, respectively.
The five subgroups therefore give five intermediate
quotient maps, but only four isomorphism classes of
varieties.

Consider now the polar quotients. The matrix
$\Ga^\circ$ in~\eqref{Gammaduale} specifies the action
of $G_{Q^\circ}\cong(\Z/2\Z)^2$ on $Z^\circ$:
\[
\begin{aligned}
(\overline c,\overline d)\cdot[z_1:\cdots:z_7]
={}&[z_1:z_2:z_3:z_4:\\
&\hspace{12mm}(-1)^c z_5:(-1)^d z_6:
(-1)^{c+d}z_7].
\end{aligned}
\]
The polar partner $X_1^\circ$ has fan matrix
\[
V_1^\circ=
\begin{pmatrix}
-1&-1&-1&-1&1&1&3\\
-1&1&1&3&-1&-1&-1\\
1&-1&0&-2&0&1&0
\end{pmatrix},
\]
and $X_2^\circ$ has fan matrix
\[
V_2^\circ=
\begin{pmatrix}
-1&-1&-1&-1&1&1&3\\
0&1&1&2&-1&-1&-2\\
1&-1&1&-1&-1&1&-1
\end{pmatrix}.
\]
These matrices can be obtained from
$V_i^\circ=(B_i^T)^{-1}V^\circ$, where $V_i=B_iV$.
With the generators specified by $\Ga$ and $\Ga^\circ$,
the polar subgroup correspondence is
\[
\begin{array}{c|ccc}
H&G_1&G_1'&G_2\\ \hline
H^\circ&
\langle(\overline1,\overline0)\rangle&
\langle(\overline0,\overline1)\rangle&
\langle(\overline1,\overline1)\rangle.
\end{array}
\]
In particular,
\[
X_1^\circ\cong
Z^\circ/\langle(\overline1,\overline0)\rangle,\qquad
X_2^\circ\cong
Z^\circ/\langle(\overline1,\overline1)\rangle.
\]
Thus $X_1^\circ$ is obtained from the first row of
$\Ga^\circ$, and $X_2^\circ$ from the sum of its two
rows. The second row gives $(X_1')^\circ\cong X_1^\circ$.
The quotients corresponding to $\{0\}$ and
$G_{Q^\circ}$ are $Z^\circ$ and $X^\circ$.
Hence the Fano toric varieties of weight matrix
$Q^\circ$ are
\[
Z^\circ,\qquad X_1^\circ\cong(X_1')^\circ,\qquad
X_2^\circ,\qquad X^\circ,
\]
with multiplicities $1,2,2,4$, respectively.

The dual factorizations of
Theorem~\ref{thm:fattorizzazione} are
\[
\xymatrix@C=12pt{
&X\ar[dl]_-{\pi_1}\ar[dd]_-p\ar[dr]^{\pi_2}&&&&
&Z^\circ\ar[dl]_-{\pi_1^\circ}
\ar[dd]_-{p^T}\ar[dr]^{\pi_2^\circ}&\\
X_1\ar[dr]_-{\pi_1^{\circ T}}&&
X_2\ar[dl]^{\pi_2^{\circ T}}&
\ar@{<~>}[r]^\circ&&
X_1^\circ\ar[dr]_-{\pi_1^T}&&
X_2^\circ\ar[dl]^{\pi_2^T}\\
&Z&&&&&X^\circ&
}.
\]

Finally, the weight moduli are
\[
|Q|=3!\Vol(\conv(V))=8,\qquad
|Q^\circ|=3!\Vol(\conv(\L^\circ))=12.
\]
The anticanonical degrees are
\[
\begin{aligned}
(-K_X)^3
&=3!\Vol(\conv(V^\circ))=48=4|Q^\circ|,\\
(-K_{X_1})^3
&=3!\Vol(\conv(V_1^\circ))=24=2|Q^\circ|,\\
(-K_{X_2})^3
&=3!\Vol(\conv(V_2^\circ))=24=2|Q^\circ|,\\
(-K_Z)^3
&=3!\Vol(\conv(\L^\circ))=12=|Q^\circ|,\\
(-K_{Z^\circ})^3
&=3!\Vol(\conv(\L))=32=4|Q|,\\
(-K_{X_1^\circ})^3
&=3!\Vol(\conv(V_1))=16=2|Q|,\\
(-K_{X_2^\circ})^3
&=3!\Vol(\conv(V_2))=16=2|Q|,\\
(-K_{X^\circ})^3
&=3!\Vol(\conv(V))=8=|Q|.
\end{aligned}
\]
These identities agree with
Theorem~\ref{thm:Fano-bounds} and
Corollary~\ref{cor:bounds}. In every polar pair, the
product of the multiplicities is $g_Q=4$.
Here $r=3$ and $r^\circ=4$, so
$r'=4$, consistently with the upper bound in
Theorem~\ref{thm:bounds}\textup{(2)}.
\end{example}

\section{Extending bounds to $\mathbb Q$-Fano toric varieties}
\label{sez:QFano}

Throughout this section, $X=X(\Sigma)$ is an $n$-dimensional
$\mathbb Q$-Fano toric variety, with $n\geq 2$. Thus $X$ is
projective and $-kK_X$ is ample and Cartier for some positive
integer $k$: in particular, the least of them is the Gorenstein index of $X$. We do not assume $\mathbb Q$-factoriality unless
explicitly stated.

We retain the notation
\[
n=\dim X,\qquad r=\operatorname{rk}\operatorname{Cl}(X),
\qquad m=n+r=|\Sigma(1)|,
\]
and let
\[
V=(\mathbf v_1,\ldots,\mathbf v_m)\in\mathbb Z^{n,m},
\qquad Q=\mathcal G(V)\in\mathbb Z^{r,m}
\]
be a fan matrix and a weight matrix of $X$, respectively.
Since $X$ is $\mathbb Q$-Fano, $\Sigma$ is the fan consisting
of cones over the proper faces of $\operatorname{Conv}(V)$,
and every column of $V$ is a vertex of this polytope.

Let $V^\circ\in\mathbb Q^{n,m^\circ}$ have as columns all the
distinct vertices of
\[
\Delta_{-K_X}
=\operatorname{Conv}(V)^\circ
=\left\{\m\in M_{\mathbb R}\ \middle|\
\langle\m,\mathbf v_j\rangle\geq-1,\quad 1\leq j\leq m\right\}.
\]
Equivalently, its columns are the rational Cartier data
$\m_\sigma$, for $\sigma\in\Sigma(n)$, characterized by
\[
\langle\m_\sigma,\mathbf v_j\rangle=-1
\qquad\text{whenever }\langle\mathbf v_j\rangle\in\sigma(1).
\]
Ampleness implies that these points are distinct vertices.
In particular,
\[
m^\circ=|\Sigma(n)|,\qquad
r^\circ=m^\circ-n,\qquad
\Delta_{-kK_X}=\operatorname{Conv}(kV^\circ),
\]
and $k$ is the least positive integer for which $kV^\circ$
has integral entries.

Let $\pi:Y=Y(\Theta)\longrightarrow X$ be the universal
$1$-covering and choose $W=\mathcal G(Q)$ as a fan matrix
of $Y$. In suitable lattice bases, $\pi$ is represented by
the injective lattice map $\beta:N\longrightarrow N$ with
matrix $B\in\operatorname{GL}_n(\mathbb Q)\cap\mathbb Z^{n,n}$,
where $V=BW$.
The lifted fan $\Theta$ is the face fan of
$\operatorname{Conv}(W)$. Thus $Y$ is also $\mathbb Q$-Fano.
Define $W^\circ$ to be the matrix of distinct vertices of
$\Delta_{-K_Y}=\operatorname{Conv}(W)^\circ$. With compatible
column orderings, polarity gives
\begin{equation}\label{eq:cover-polars}
V=BW,\qquad W^\circ=B^T V^\circ.
\end{equation}
In particular, $W^\circ$ is defined from the anticanonical
polytope of $Y$, independently of the index of $X$.

For a rational matrix, its integral Gale dual means a
matrix whose rows form a basis of its integral kernel.
Clearing denominators does not change this kernel.
We therefore set
\begin{equation}\label{eq:polar-weight}
Q^\circ:=\mathcal G(V^\circ)=\mathcal G(kV^\circ)
\in\mathbb Z^{r^\circ,m^\circ}.
\end{equation}
By \eqref{eq:cover-polars}, we may choose the same Gale
dual for $W^\circ$. The matrix $Q^\circ$ is a W-matrix,
possibly non-reduced, and may be chosen with non-negative
entries.

\begin{proposition}\label{prop:indici}
More generally, let $\pi:Y\longrightarrow X$ be the universal
$1$-covering of a complete toric variety.
Then $X$ is $\mathbb Q$-Gorenstein if and only if $Y$ is
$\mathbb Q$-Gorenstein, and $X$ is $\mathbb Q$-Fano if and
only if $Y$ is $\mathbb Q$-Fano.
In the $\mathbb Q$-Gorenstein case, the Gorenstein index
$\widehat k$ of $Y$ divides the Gorenstein index $k$ of $X$.
\end{proposition}

\begin{proof}
The fans are identified by the rational lattice isomorphism
$\beta_{\mathbb Q}$, and corresponding primitive ray
generators satisfy $\beta(\mathbf w_j)=\mathbf v_j$.
Consequently, rational Cartier data for $K_X$ and $K_Y$
correspond through $\beta_{\mathbb Q}^T$. This proves the
equivalence of the $\mathbb Q$-Gorenstein conditions.

Since $\pi$ is unramified in codimension one, in this case
\[
K_Y=\pi^*K_X.
\]
Thus $kK_Y$ is Cartier whenever $kK_X$ is Cartier, proving
$\widehat k\mid k$.
Finally, ampleness of a $\mathbb Q$-Cartier divisor is
equivalent to ampleness of its pullback by a finite
surjective morphism. Applying this to $-K_X$ proves the
assertion about $\mathbb Q$-Fano varieties.
\end{proof}

\begin{remark}\label{rem:cover-index}
Proposition~\ref{prop:indici} extends to complete
toric varieties the index divisibility for fake weighted
projective spaces recalled in \cite[Cor.~13]{Bauerle}.
In the $\mathbb Q$-Fano setting, it also follows directly
from
\[
kW^\circ=B^T(kV^\circ)\in\mathbb Z^{n,m^\circ}.
\]
\end{remark}

\subsection{The weight group of a $\mathbb Q$-Fano toric variety}
\label{subsec:weight-group}

Define
\begin{equation}\label{eq:lambda-circ}
\Lambda^\circ:=\mathcal G(Q^\circ)
\in\mathbb Z^{n,m^\circ}.
\end{equation}
Both $kV^\circ$ and $\widehat kW^\circ$ are integral
F-matrices with Gale dual $Q^\circ$. Their common
double Gale dual $\Lambda^\circ$ is a CF-matrix, although
its columns need not be primitive.

We extend the volume definition of the weight modulus to
this possibly non-reduced W-matrix by setting
\[
|Q^\circ|:=n!\operatorname{Vol}
                 \bigl(\operatorname{Conv}(\Lambda^\circ)\bigr).
\]
This is a positive integer, independent of the choice of
the CF-matrix $\mathcal G(Q^\circ)$ up to unimodular
equivalence. For the reduced weight matrix $Q$, we retain
\[
|Q|=n!\operatorname{Vol}\bigl(\operatorname{Conv}(W)\bigr).
\]
Here $\operatorname{Vol}$ is measured so that a fundamental
parallelepiped of the relevant lattice has volume one.
The definition of $|Q^\circ|$ does not require
$\Lambda^\circ$ to be a fan matrix.

\begin{definition}[Index of an F-matrix]\label{def:matrix-index}
Let $V$ be an F-matrix and let $V^\circ$ be a matrix
whose columns are the vertices of
$\operatorname{Conv}(V)^\circ$.
The \emph{index} of $V$ is the least positive integer
$\ell$ such that $\ell V^\circ$ has integral entries.

This index is invariant under unimodular changes of
lattice coordinates and permutations of columns.
If $V$ is the fan matrix of a $\mathbb Q$-Fano toric
variety, it equals the Gorenstein index of that variety.
\end{definition}

\begin{proposition}\label{prop:lambda-index}
The CF-matrix $\Lambda^\circ$ has the same index
$\widehat k$ as the universal $1$-covering $Y$ of $X$.
\end{proposition}

\begin{proof}
Since $\widehat kW^\circ$ and $\Lambda^\circ$ have the same
integral kernel and $\Lambda^\circ$ is a CF-matrix, there
is a unique matrix
$A^T\in\operatorname{GL}_n(\mathbb Q)\cap\mathbb Z^{n,n}$
such that
\begin{equation}\label{eq:A}
\widehat kW^\circ=A^T\Lambda^\circ.
\end{equation}
Here integrality follows from the fact that the columns
of $\Lambda^\circ$ generate the entire lattice.

Taking polars and choosing compatible column orderings gives
\begin{equation}\label{eq:lambda}
\Lambda:=(\Lambda^\circ)^\circ
=\frac{1}{\widehat k}AW.
\end{equation}
If $\widetilde k$ is the index of $\Lambda^\circ$, then
$\widehat k\Lambda$ is integral, so
$\widetilde k\mid\widehat k$.

Conversely, $\widetilde k\Lambda$ is integral.
Since the columns of $W$ generate the entire lattice,
$W$ admits an integral right inverse. Hence
$\widetilde A=(\widetilde k/\widehat k)A$ has integral entries
and
\begin{equation}\label{eq:A-tilde}
\widetilde k\Lambda=\widetilde A W,
\qquad
\widetilde A\in\operatorname{GL}_n(\mathbb Q)
                    \cap\mathbb Z^{n,n}.
\end{equation}
Taking polars of this identity yields
\[
\widetilde A^T\Lambda^\circ=\widetilde kW^\circ.
\]
The left-hand side is integral, so $\widehat k\mid\widetilde k$
by the definition of the index of $Y$. Thus
$\widetilde k=\widehat k$.
\end{proof}

In the chosen bases, $A^T$ represents an injective lattice
endomorphism $\alpha^T:M\longrightarrow M$.

\begin{remark}\label{rem:factor}
Similarly, there is a unique matrix
$C\in\operatorname{GL}_n(\mathbb Q)\cap\mathbb Z^{n,n}$
such that
\begin{equation}\label{eq:C}
kV^\circ=C\Lambda^\circ.
\end{equation}
Let $\gamma:M\longrightarrow M$ be the lattice endomorphism
represented by $C$. Combining
\eqref{eq:cover-polars}, \eqref{eq:A}, and
\eqref{eq:C}, we obtain
\begin{equation}\label{eq:factorization}
h\alpha^T=\beta^T\circ\gamma,
\qquad h:=\frac{k}{\widehat k}\in\mathbb Z_{>0}.
\end{equation}
Equivalently, $hA^T=B^TC$.
We call $h$ the \emph{factor} of $X$.

For fixed compatible lattice identifications,
\eqref{eq:factorization} induces a surjective morphism
\[
\bar i_h:\operatorname{coker}(h\alpha^T)
                 \longrightarrow\operatorname{coker}\beta^T,
\qquad
\ker\bar i_h=
\frac{\operatorname{Im}\beta^T}
     {\operatorname{Im}(h\alpha^T)}.
\]
The map induced by $\beta^T$ gives a canonical isomorphism
$\operatorname{coker}\gamma\cong\ker\bar i_h$.
In particular, the factorization gives a quotient map onto
$\operatorname{coker}\beta^T$, without choosing an embedding
of that cokernel into $\operatorname{coker}(h\alpha^T)$.
\end{remark}

\begin{definition}\label{def:weight-group}
The \emph{weight group} and the \emph{weight order} associated
with $Q$ are
\[
G_Q:=\operatorname{coker}\alpha^T,
\qquad
g_Q:=|G_Q|=|\det A|.
\]
For a positive integer $h$, define the \emph{$h$-extension}
of $G_Q$ and its order by
\[
\widehat G_Q^{\,h}:=\operatorname{coker}(h\alpha^T),
\qquad
\widehat g_Q^{\,h}:=|\widehat G_Q^{\,h}|=h^n g_Q.
\]
For a fixed weight matrix $Q$ admitting a $\mathbb Q$-Fano
realization, the face fan of $\operatorname{Conv}(\mathcal G(Q))$
determines $Y$ up to toric isomorphism. Hence $\widehat k$,
$G_Q$, and $\widehat G_Q^{\,h}$ are independent, up to the
appropriate isomorphisms, of the chosen $\mathbb Q$-Fano
quotient $X$ with weight matrix $Q$.
\end{definition}

There is a canonical short exact sequence
\begin{equation}\label{eq:h-extension}
0\longrightarrow G_Q
\xrightarrow{\bar h}\widehat G_Q^{\,h}
\xrightarrow{\bar r_h}M/hM
\longrightarrow0,
\end{equation}
where
\[
\begin{aligned}
\bar h\bigl(\m+\operatorname{Im}\alpha^T\bigr)
   &=h\m+\operatorname{Im}(h\alpha^T),\\
\bar r_h\bigl(\m+\operatorname{Im}(h\alpha^T)\bigr)
   &=\m+hM.
\end{aligned}
\]
Indeed, multiplication by $h$ is injective on $M$, and
$\ker\bar r_h=hM/\operatorname{Im}(h\alpha^T)$.
Moreover, $\alpha^T$ induces a canonical isomorphism
\[
M/hM\xrightarrow{\sim}
\operatorname{Im}\alpha^T/\operatorname{Im}(h\alpha^T),
\]
and $M/hM\cong(\mathbb Z/h\mathbb Z)^n$ after choosing a basis.
There is also the natural short exact sequence
\[
0\longrightarrow
\frac{\operatorname{Im}\alpha^T}
     {\operatorname{Im}(h\alpha^T)}
\longrightarrow\widehat G_Q^{\,h}
\xrightarrow{\bar q_h}G_Q
\longrightarrow0,
\]
where $\bar q_h$ is induced by the identity on $M$.

\begin{remark}\label{rem:polar-cover}
If $\Lambda^\circ$ is reduced, let $\Xi^\circ$ be the fan
of cones over the proper faces of
$\operatorname{Conv}(\Lambda^\circ)$, and set
\[
Z^\circ:=X(\Xi^\circ).
\]
Then $Z^\circ$ is a $\mathbb Q$-Fano toric variety,
simply connected in codimension one, with Gorenstein
index $\widehat k$. Indeed, $\Lambda^\circ$ is a reduced
CF-matrix and
\[
\Delta_{-K_{Z^\circ}}=\operatorname{Conv}(\Lambda).
\]
Proposition~\ref{prop:lambda-index} gives its index.

When $\widehat k=1$, $Y$ is Gorenstein Fano,
$\Lambda^\circ$ is reduced, and $Z^\circ$ agrees with
the universal $1$-covering of the polar variety $Y^\circ$.
We retain the terminology \emph{polar $1$-covering} for
this auxiliary variety when it exists. For general $X$,
this terminology does not assert the existence of a
polar toric variety $X^\circ$ or of a $1$-covering
$Z^\circ\longrightarrow X^\circ$.

In rank one, after indexing polar vertices by the
opposite facets, Conrads' result \cite[Lemma~5.3]{Conrads} gives
$Q^\circ=Q$, with compatible choices of Gale duals.
We may then take $\Lambda^\circ=W$, so that
$Z^\circ\cong Y$.

In general, $\Xi^\circ$ need not be simplicial.
Consequently, $Z^\circ$ need not be $\mathbb Q$-factorial,
even when $X$ and $Y$ are $\mathbb Q$-factorial.
\end{remark}

\begin{definition}\label{def:formal-degree}
With $\Lambda$ as in \eqref{eq:lambda}, we set
\[
(-K_{Z^\circ})^n
:=n!\operatorname{Vol}\bigl(\operatorname{Conv}(\Lambda)\bigr)
\in\mathbb Q_{>0}.
\]
When $Z^\circ$ exists as in
Remark~\ref{rem:polar-cover}, this is its anticanonical
self-intersection. Otherwise, the notation denotes only
the displayed combinatorial quantity; it does not
represent an intersection number on a toric variety.
\end{definition}

\subsection{Bounds on $\mathbb Q$-Fano varieties}
\label{subsec:bounds}

The volume definition of the weight modulus gives the
following extension of the degree identities from Section~2.

\begin{theorem}\label{thm:degree}
Let $X$ be an $n$-dimensional $\mathbb Q$-Fano toric variety,
with fan matrix $V$, weight matrix $Q$, and universal
$1$-covering $\pi:Y\longrightarrow X$.
Let $k$ and $\widehat k$ be the indices of $X$ and $Y$,
respectively, and put $h=k/\widehat k$.
With the notation introduced above, the following hold:
\begin{enumerate}
\item
$
\operatorname{mult}X
=|\det B|
=\displaystyle \frac{n!\operatorname{Vol}(\operatorname{Conv}(V))}{|Q|}
$
\item
$
(-kK_X)^n=|\det C|\,|Q^\circ|.
$
\item
$
g_Q=|\det A|
=\displaystyle \frac{(-\widehat kK_Y)^n}{|Q^\circ|}
=\displaystyle \frac{\widehat k^{\,n}(-K_{Z^\circ})^n}{|Q|}.
$
\item
$
\widehat g_Q^{\,h}=h^n g_Q
=\displaystyle \frac{(-kK_Y)^n}{|Q^\circ|}
=\displaystyle \frac{k^n(-K_{Z^\circ})^n}{|Q|}.
$
\item
$
\operatorname{mult}X\,\displaystyle
\frac{(-kK_X)^n}{|Q^\circ|}
=\widehat g_Q^{\,h}.
$
\end{enumerate}
Here $(-K_{Z^\circ})^n$ has the meaning specified in
Definition~\ref{def:formal-degree}.
\end{theorem}

\begin{proof}
The degree of the $1$-covering is
$\operatorname{mult}X=|\det B|$, and $V=BW$ gives
\[
n!\operatorname{Vol}(\operatorname{Conv}(V))
=|\det B|\,|Q|.
\]
This proves (1).
Since $-kK_X$ is ample and Cartier,
\[
(-kK_X)^n
=n!\operatorname{Vol}(\operatorname{Conv}(kV^\circ)).
\]
Equation~\eqref{eq:C} therefore proves (2).
Likewise, \eqref{eq:A} gives
\[
(-\widehat kK_Y)^n=|\det A|\,|Q^\circ|,
\]
whereas \eqref{eq:lambda} gives
\[
\widehat k^{\,n}
 n!\operatorname{Vol}(\operatorname{Conv}(\Lambda))
=|\det A|\,|Q|.
\]
These are the equalities in (3). Multiplication by $h^n$
proves (4). Finally,
\[
|\det B|\,|\det C|=h^n|\det A|
\]
by \eqref{eq:factorization}; using (1) and (2) proves (5).
\end{proof}

\begin{corollary}\label{cor:divisibility}
Under the hypotheses and notation of
Theorem~\ref{thm:degree}, the following hold:
\begin{enumerate}
\item
$
\operatorname{mult}X\ \bigm|\ \widehat g_Q^{\,h}
=\displaystyle\frac{(-kK_Y)^n}{|Q^\circ|}
=\displaystyle\frac{k^n(-K_{Z^\circ})^n}{|Q|}.
$
\item
$
(-kK_X)^n\ \bigm|\ (-kK_Y)^n
=\widehat g_Q^{\,h}|Q^\circ|
$
and
$$
(-kK_X)^n=\widehat g_Q^{\,h}|Q^\circ|
\quad\Longleftrightarrow\quad X\cong Y
$$
\item
$
|Q^\circ|\ \bigm|\ (-kK_X)^n.
$
\end{enumerate}
\end{corollary}

\begin{proof}
All the quantities in the divisibility assertions are
positive integers. Parts (1) and (3) follow from
Theorem~\ref{thm:degree}(2),(5).
For (2), the projection formula gives
\[
\operatorname{mult}X\,(-kK_X)^n=(-kK_Y)^n.
\]
Thus equality holds precisely when the $1$-covering has degree
one, equivalently when $\pi$ is an isomorphism.
If $X\cong Y$ abstractly, their anticanonical degrees
coincide, so the same identity also forces
$\operatorname{mult}X=1$.
\end{proof}

\begin{theorem}\label{thm:canonical-bound}
Let $X$ be an $n$-dimensional $\mathbb Q$-Fano toric variety
of rank $r$ and Gorenstein index $k$, with at worst
canonical singularities. Let
\[
r^\circ=\operatorname{rk}Q^\circ,\qquad
r'=\max(r,r^\circ).
\]
Then:
\begin{enumerate}
\item If $n=2$, then $k=1$ and
\[
\operatorname{mult}X
\leq\left[\frac{9}{2+r'}\right].
\]
\item If $n=3$, then
\[
\operatorname{mult}X
\leq\left[\frac{72k^3}{3+r'}\right].
\]
\item If $n\geq4$, then
\[
\operatorname{mult}X
\leq\left[
\frac{2(s_n-1)^2k^n}{n+r'}
\right],
\]
\end{enumerate}
where $s_1=2$ and
$s_j=1+\prod_{i=1}^{j-1}s_i$ for $j\geq2$ and $[\,\cdot\,]$ denotes the integer part of a rational number.
\end{theorem}

\begin{proof}
First of all, recall that Lemma~\ref{lem:volumeverticesfacets} gives
\[
n!\operatorname{Vol}(R)
\geq\max\{f_0(R),f_{n-1}(R)\},
\]
where $f_0$ and $f_{n-1}$ denote the numbers of vertices
and facets, respectively.

Apply this to $R=\Delta_{-kK_X}$. The $\mathbb Q$-Fano
hypothesis implies that $R$ has $n+r$ facets and
$n+r^\circ$ vertices. Therefore
\[
(-kK_X)^n=n!\operatorname{Vol}(R)\geq n+r'.
\]
The projection formula now yields
\[
\operatorname{mult}X
=\frac{(-kK_Y)^n}{(-kK_X)^n}
\leq\frac{(-kK_Y)^n}{n+r'}.
\]
The universal $1$-covering $Y$ is canonical.
Indeed, the sublattice generated by the rays of $\Sigma$
cannot introduce additional interior lattice points in
$\operatorname{Conv}(V)$, whose only interior lattice
point is $\mathbf0$. Since $\widehat k\mid k$, the divisor $-kK_Y$ is Cartier.
The degree bounds recalled in Section~2 give
\[
(-kK_Y)^n\leq
\begin{cases}
9k^2,&n=2,\\
72k^3,&n=3,\\
2(s_n-1)^2k^n,&n\geq4.
\end{cases}
\]
with the last bound supplied by \cite[Cor.~1.3]{BKN}.
These inequalities prove (2) and (3).
Finally, a canonical toric surface is Gorenstein, so
$k=1$ in dimension two, giving (1).
\end{proof}

\begin{example}\label{ex:QFanoCanonica}
Consider the fan matrix and weight matrix
\[
V=
\begin{pmatrix}
1&1&-2&0&0\\
0&3&-3&1&-1\\
0&0&0&2&-2
\end{pmatrix},
\qquad
Q=
\begin{pmatrix}
1&1&1&0&0\\
0&0&0&1&1
\end{pmatrix}.
\]
There is a unique complete simplicial fan $\Sigma$ over
$V$, with maximal cones
\[
\Sigma(3)=
\left\{
\langle\mathbf v_i,\mathbf v_j,\mathbf v_\ell\rangle
\ \middle|\ 1\leq i<j\leq3,\quad \ell\in\{4,5\}
\right\}.
\]
Let $X=X(\Sigma)$. Its anticanonical class in
$\operatorname{Cl}(X)\otimes\mathbb R$ is
\[
[-K_X]_{\mathbb R}=
\begin{pmatrix}3\\2\end{pmatrix}
\in\operatorname{relint}\operatorname{Nef}(X),
\qquad
\operatorname{Nef}(X)=
\left\langle
\begin{pmatrix}1\\0\end{pmatrix},
\begin{pmatrix}0\\1\end{pmatrix}
\right\rangle.
\]
Thus $X$ is $\mathbb Q$-Fano. Its polar vertex matrix is
\[
V^\circ=
\begin{pmatrix}
-1&-1&-1&-1&2&2\\
0&0&1&1&-1&-1\\
-\frac12&\frac12&-1&0&0&1
\end{pmatrix}.
\]
Consequently, $X$ is non-Gorenstein and has index $k=2$.

We may choose
\[
W=\mathcal G(Q)=
\begin{pmatrix}
1&0&-1&0&0\\
0&1&-1&0&0\\
0&0&0&1&-1
\end{pmatrix},
\qquad
B=
\begin{pmatrix}
1&1&0\\
0&3&1\\
0&0&2
\end{pmatrix},
\]
so that $V=BW$ and the universal $1$-covering is
$Y=\mathbb P^2\times\mathbb P^1$.
In particular,
\[
\widehat k=1,\qquad h=2,\qquad
G\cong\mathbb Z/6\mathbb Z,\qquad
\operatorname{mult}X=6.
\]
Let $\xi$ be a primitive sixth root of unity.
The torsion row
\[
\Gamma=\begin{pmatrix}0&2&1&4&1\end{pmatrix}
\]
gives the action
\[
\bar a\cdot
\bigl([x_1:x_2:x_3],[x_4:x_5]\bigr)
=
\bigl([x_1:\xi^{2a}x_2:\xi^a x_3],
      [\xi^{4a}x_4:\xi^a x_5]\bigr).
\]
Its quotient is $X$, and the quotient map is unramified
in codimension one. Moreover,
\[
\operatorname{int}(\operatorname{Conv}(V))\cap N
=\{\mathbf0\},
\]
so $X$ has canonical singularities.
The polar vertex matrix of $Y$ is
\[
W^\circ=B^TV^\circ=
\begin{pmatrix}
-1&-1&-1&-1&2&2\\
-1&-1&2&2&-1&-1\\
-1&1&-1&1&-1&1
\end{pmatrix}.
\]
We may take
\[
Q^\circ=\mathcal G(2V^\circ)=
\begin{pmatrix}
1&1&0&2&2&0\\
0&2&1&1&2&0\\
1&1&1&1&1&1
\end{pmatrix},
\qquad
\Lambda^\circ=\mathcal G(Q^\circ)=
\begin{pmatrix}
1&1&0&0&-1&-1\\
0&2&0&2&-3&-1\\
0&0&1&1&-1&-1
\end{pmatrix}.
\]
Thus $r=2$, $r^\circ=3$, and $r'=3$.
\begin{figure}
\begin{center}
\includegraphics[width=10truecm]{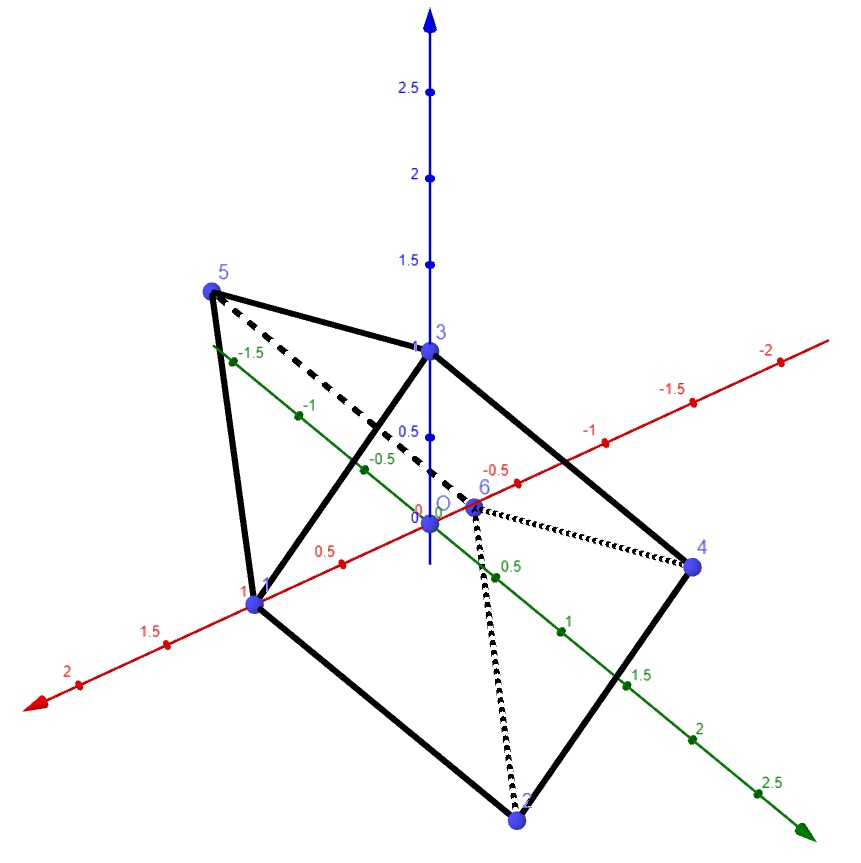}
\caption{\label{fig:politopo} Example \ref{ex:QFanoCanonica}: the polytope $\conv(\L^\circ)$ spanning the fan $\Xi^\circ$. }
\end{center}
\end{figure}
For these choices,
\[
A^T=
\begin{pmatrix}
-1&0&-1\\
-1&0&2\\
-1&1&-1
\end{pmatrix},
\qquad
C=
\begin{pmatrix}
-2&0&-2\\
0&0&2\\
-1&1&-2
\end{pmatrix}.
\]
They satisfy
\[
W^\circ=A^T\Lambda^\circ,\qquad
2V^\circ=C\Lambda^\circ,\qquad
2A^T=B^TC.
\]
Their Smith normal forms give
\[
G_Q\cong\mathbb Z/3\mathbb Z,\qquad
\widehat G_Q^{\,2}
\cong(\mathbb Z/2\mathbb Z)^2\oplus\mathbb Z/6\mathbb Z.
\]
In particular,
\[
g_Q=3,\qquad \widehat g_Q^{\,2}=24.
\]
The matrix $\Lambda^\circ$ is a reduced CF-matrix,
and its polar vertex matrix is
\[
\Lambda=
\begin{pmatrix}
-1&-1&-1&1&2\\
0&0&1&-1&0\\
-1&2&-1&1&-1
\end{pmatrix}.
\]
This agrees with $AW$ up to a permutation of columns.
Hence the variety $Z^\circ=X(\Xi^\circ)$ of
Remark~\ref{rem:polar-cover} exists and is Gorenstein
Fano.
However, $Z^\circ$ is not $\mathbb Q$-factorial, as can be deduced from Fig.~\ref{fig:politopo}: the fan $\Xi^\circ$ whose rays are generated by the columns of $\Lambda^\circ$ contains three non-simplicial cones spanned by the three quadrangular faces.


Finally, the normalized volumes give
\[
|Q|=6,\qquad |Q^\circ|=18,
\]
and
\[
\begin{aligned}
(-K_Y)^3&=54=g_Q|Q^\circ|,\\
(-K_{Z^\circ})^3&=18=g_Q|Q|,\\
(-2K_X)^3&=72=2^3\,9,\\
(-2K_Y)^3&=432
\end{aligned}
\]
In particular,
\[
\operatorname{mult}X\,
\frac{(-2K_X)^3}{|Q^\circ|}
=6\cdot\frac{72}{18}=24=\widehat g_Q^{\,2},
\]
and
\[
6\cdot72=432.
\]
These identities verify Theorem~\ref{thm:degree} and
Corollary~\ref{cor:divisibility}.
The bound of Theorem~\ref{thm:canonical-bound}(2) is
also satisfied:
\[
\operatorname{mult}X=6
\leq\left[\frac{72\cdot2^3}{3+3}\right]=96.
\]
\end{example}

\subsection{Dropping the canonical condition for fake wps}
\label{subsec:fake-wps}

In rank one, the universal $1$-covering is a weighted
projective space. We apply the degree bounds of
B\"auerle \cite[Thm.~1]{Bauerle} and his volume bounds
\cite[Thm.~1.1]{Bauerle25} to the multiplicity identities
established above. The resulting estimates retain the
dependence on the Gorenstein index of the universal
$1$-covering and on the weight modulus, without assuming
canonical singularities.

For every positive integer $k$, set
\begin{equation}\label{eq:sylvester-index}
s_{k,1}=k+1,\qquad
s_{k,j}=k\prod_{i=1}^{j-1}s_{k,i}+1\quad(j\geq2),
\qquad
t_{k,j}=s_{k,j}-1\quad(j\geq1).
\end{equation}
For $n\geq2$ and every positive integer $\ell$, write
\[
b_n(\ell)=
\begin{cases}
9,&n=2,\ \ell=1,\\[2pt]
\dfrac{2(\ell+1)^2}{\ell},&n=2,\ \ell\geq2,\\[6pt]
\dfrac{2t_{\ell,n}^{\,2}}{\ell^{n+1}},&n\geq3.
\end{cases}
\]
Thus $b_n(\ell)$ is the anticanonical degree bound in
\cite[Thm.~1]{Bauerle}.

\begin{proposition}\label{thm:fake-wps}
Let $X$ be an $n$-dimensional fake weighted projective
space, with $n\geq2$, Gorenstein index $k$, and reduced
weight matrix $Q=(q_0,\ldots,q_n)$.
Let $\pi:Y\longrightarrow X$ be its universal
$1$-covering, and let $\widehat k$ be the Gorenstein
index of $Y$. Then:
\begin{enumerate}
\item
\[
\operatorname{mult}X
\leq\left[
\frac{k^n b_n(\widehat k)}{n+1}
\right]
\leq\left[
\frac{k^n b_n(k)}{n+1}
\right].
\]
\item
\[
\operatorname{mult}X\leq
\begin{cases}
\left[\dfrac{9}{|Q|}\right],
&(n,k)=(2,1),\\[8pt]
\left[\dfrac{2t_{k,n}^{\,2}}{k^2|Q|}\right],
&(n,k)\ne(2,1).
\end{cases}
\]
\end{enumerate}
Here $|Q|=q_0+\cdots+q_n\geq n+1$.
\end{proposition}

\begin{proof}
Since $-kK_X$ is ample and Cartier, its associated
polytope $R=\Delta_{-kK_X}$ is a lattice simplex
containing the origin in its interior.
Lemma~\ref{lem:volumeverticesfacets} therefore gives
\[
(-kK_X)^n=n!\operatorname{Vol}(R)\geq n+1.
\]
Since $\widehat k\mid k$, put $h=k/\widehat k$.
Both $-\widehat kK_Y$ and $-kK_Y$ are Cartier, and
$\pi^*(-kK_X)=-kK_Y$. The projection formula and the
homogeneity of Cartier intersection numbers yield
\[
\operatorname{mult}X\,(-kK_X)^n
=(-kK_Y)^n=h^n(-\widehat kK_Y)^n.
\]
Since $Y$ is a weighted projective space of Gorenstein
index $\widehat k$, \cite[Thm.~1]{Bauerle}, expressed
using the Cartier divisor $-\widehat kK_Y$, gives
\[
(-\widehat kK_Y)^n
\leq\widehat k^{\,n}b_n(\widehat k).
\]
Consequently,
\[
\operatorname{mult}X
\leq\frac{h^n(-\widehat kK_Y)^n}{n+1}
\leq\frac{k^n b_n(\widehat k)}{n+1}.
\]
Taking integer parts proves the first inequality in (1).

We recall that $\widehat k\mid k$ and check that
$b_n(\ell)$ is non-decreasing on the positive integers.
For $n=2$, we have $b_2(1)=b_2(2)=9$, and
$2(\ell+1)^2/\ell$ is increasing for $\ell\geq2$.
For $n\geq3$, extend the defining polynomials to real
$x\geq1$. Then
\[
b_n(x)=2\prod_{j=1}^{n-1}\frac{s_{x,j}^{\,2}}{x}.
\]
Each polynomial $p(x)=s_{x,j}$ has non-negative
coefficients, constant coefficient $1$, and linear
coefficient $1$. Consequently,
\[
2xp'(x)-p(x)\geq x-1\geq0,
\]
so $p(x)^2/x$ is non-decreasing for $x\geq1$.
This proves the second inequality in (1).

For (2), let $V$ be a fan matrix of $X$ and put
$P=\operatorname{Conv}(V)$.
The simplex $P$ has primitive lattice vertices,
contains the origin in its interior, and has
Gorenstein index $k$.
Theorem~\ref{thm:degree}\textup{(1)} gives
\[
\operatorname{mult}X
=\frac{n!\operatorname{Vol}(P)}{|Q|}.
\]
Applying \cite[Thm.~1.1]{Bauerle25} to $P$ proves (2).
The formula for $|Q|$ follows from
Remark~\ref{rem:r=1 Q-modulo}.
\end{proof}

\begin{remark}\label{rem:degree-conjecture}
B\"auerle \cite[Thm.~1.2]{Bauerle25} proves the following
multiplicity bounds, depending only on $n$ and $k$:
\[
\operatorname{mult}X\leq
\begin{cases}
16k^2,&n=3,\ k\in\{1,2\},\\[2pt]
128,&(n,k)=(4,1),\\[2pt]
\dfrac{3t_{k,n-1}^{\,2}}{k},&\text{otherwise}.
\end{cases}
\]
Proposition~\ref{thm:fake-wps} records applications of
the degree and volume bounds that retain the additional
data $\widehat k$ and $|Q|$.
Replacing these data by $k$ and $n+1$, respectively,
does not improve the displayed bounds in terms of
dimension and Gorenstein index alone.

The degree argument also applies in higher rank.
Indeed, for any $n$-dimensional $\mathbb Q$-Fano toric
variety of rank $r$ and Gorenstein index $k$, put
$r'=\max(r,r^\circ)$, where
$r^\circ=\operatorname{rk}Q^\circ$.
The polytope $\Delta_{-kK_X}$ has $n+r^\circ$ vertices
and $n+r$ facets, so Lemma~\ref{lem:volumeverticesfacets}
and the projection formula give
\[
(-kK_X)^n\geq n+r',
\qquad
\operatorname{mult}X
\leq\frac{(-kK_Y)^n}{n+r'}=\frac{h^n(-\widehat{k}K_Y)^n}{n+r'}.
\]
This argument does not require canonical singularities.
If the universal $1$-covering satisfies
$(-\widehat{k}K_Y)^n\leq \widehat{k}^n b_n(\widehat k)$, the monotonicity proved
above implies
\[
\operatorname{mult}X
\leq\left[\frac{k^n b_n(k)}{n+r'}\right].
\]
Such an extension of the degree bounds to higher rank
is not established here; it would imply the following
conjecture. The volume and multiplicity bounds for
simplices in \cite{Bauerle25} do not establish this
extension.
\end{remark}

\begin{conjecture}\label{conj:higher-rank}
Let $X$ be an $n$-dimensional $\mathbb Q$-Fano toric
variety of rank $r$ and Gorenstein index $k\geq2$.
Set
\[
r^\circ=\operatorname{rk}Q^\circ,\qquad
r'=\max(r,r^\circ).
\]
Then:
\begin{enumerate}
\item If $n=2$, then
\[
\operatorname{mult}X
\leq\left[\frac{2k(k+1)^2}{2+r'}\right].
\]
\item If $n=3$, then
\[
\operatorname{mult}X
\leq\left[\frac{2t_{k,3}^{\,2}}{k(3+r')}\right].
\]
\item If $n\geq4$, then
\[
\operatorname{mult}X
\leq\left[\frac{2t_{k,n}^{\,2}}{k(n+r')}\right].
\]
\end{enumerate}
The sequence $t_{k,n}$ is defined in
\eqref{eq:sylvester-index}.
\end{conjecture}

\section{A classification of $\mathbb Q$-Fano toric varieties}
\label{sez:Qclassification}

The present section gives a $\mathbb Q$-Fano version of the
classification in Theorem~3.2. In this setting, one fixes both a
$\mathbb Q$-Fano weight matrix $Q$ and a positive integer $h$,
representing the factor introduced in Remark~\ref{rem:factor}. The classification
requires a primitivity condition on the rays and a condition ensuring
that the factor is exactly $h$.

Let $Q$ be a $\mathbb Q$-Fano weight matrix and let
$Y=Y(\Theta)$ be the $\Q$-Fano toric variety with fan matrix $W=\mathcal G(Q)$.
Write $\widehat k$ for its Gorenstein index, and retain
$G_Q$, its $h$-extension $\widehat G_Q^{\,h}$, and
$\bar h:G_Q\hookrightarrow\widehat G_Q^{\,h}$
from Section~\ref{sez:QFano} and, in particular, Definition~\ref{def:weight-group} and exact sequence (\ref{eq:h-extension}).
We take the exponent of the trivial group to be $1$.

\begin{theorem}\label{s5:thm:classification}
The groups $\widehat G_Q^{\,h}$, for $h\geq1$, admit
effective actions on $Y$ by torus translations, whose
restrictions along $\bar h$ agree with the action of $G_Q$,
and for which the following assertions hold.
When $\widehat k=1$, the action of $G_Q$ is that of
Theorem~\ref{thm:classificazione}.

For $H\leq\widehat G_Q^{\,h}$, set
$H_1:=H\cap\bar h(G_Q)$. Call $H$ \emph{admissible of
factor $h$} if
\begin{enumerate}
\item[\textup{(a)}] $H$ acts freely in codimension one on $Y$;
\item[\textup{(b)}] $\exp(H/H_1)=h$.
\end{enumerate}
Then:
\begin{enumerate}
\item
The $\mathbb Q$-Fano toric varieties with weight matrix
$Q$ and factor $h$ are precisely the quotients $X=Y/H$
by admissible subgroups. Their Gorenstein index is
$k=h\widehat k$, and $Y\to X$ is the universal
$1$-covering, with deck transformation group $H$ and
\[
\pi_1^{\acute{e}t}(X)^{(1)}\cong H,\qquad
\mult X=|H|.
\]
Admissible subgroups classify marked quotient maps
$Y\to X$, meaning that isomorphisms commute with the
maps from the fixed source $Y$.

\item
The fan $\Theta$ is the uniquely determined face fan
of $P:=\operatorname{Conv}(W)$.
After forgetting the marking, toric isomorphism classes
correspond to the orbits of admissible subgroups under
the induced action of $\operatorname{Aut}(N,P)$.
Changing the lattice basis gives torically isomorphic
constructions.

\item
For every admissible $H$, the quotient $X_1:=Y/H_1$ has
factor $1$ and gives a factorization of $1$-coverings
\[
Y\xrightarrow{\widehat\pi_1}X_1
\xrightarrow{\pi_1}X=Y/H,\qquad
\deg\pi_1=|H/H_1|,\qquad
h\mid\deg\pi_1\mid h^n.
\]
The map $Y\to X_1$ is the universal $1$-covering.
The deck transformation groups of the two maps are
$H_1$ and $H/H_1$, respectively.
The $1$-covering $X_1\to X$ is independent of the auxiliary
group identifications: every intermediate $1$-covering
$X'\to X$ of factor $1$ factors through it.
\end{enumerate}
\end{theorem}

\begin{definition}\label{s5:def:unitary-cover}
The $1$-covering $\pi_1:X_1\to X$ in
Theorem~\ref{s5:thm:classification}\textup{(3)}
is called the \emph{unitary $1$-covering} of $X$.
\end{definition}

\begin{proof}[Proof of Theorem~\ref{s5:thm:classification}]

Write $F^\vee:=\Hom(F,\K^*)$ for any finite abelian group
$F$, and set $L:=\operatorname{Im}\alpha^T\subseteq M$.
By the construction in Section~\ref{sez:QFano}, $L$ is
generated by the columns of $\widehat kW^\circ$, and
\[
G_Q=M/L,\qquad \widehat G_Q^{\,h}=M/hL.
\]
The quotient $M\to M/hL$ identifies
$(\widehat G_Q^{\,h})^\vee$ with a subgroup of
$\T=\Hom(M,\K^*)$, the acting torus of $Y$.

Choose a lattice isomorphism $T:N\to M$ satisfying
$T(\operatorname{Im}\alpha)=\operatorname{Im}\alpha^T$,
and a group embedding $\varepsilon:\Q/\Z\hookrightarrow\K^*$,
whose existence was explained in the proof of
Theorem~\ref{thm:classificazione}.
Existence of $T$ follows from the Smith normal form
argument of Proposition~\ref{prop:isoGQ}, which applies
to these endomorphisms as well. When $\widehat k=1$, use
the same choices as in the proof of
Theorem~\ref{thm:classificazione}.

Use this one $T$ for every $h\geq1$, and denote its
induced isomorphism by
\[
\bar T_h:\coker(h\alpha)\xrightarrow{\sim}
\widehat G_Q^{\,h}.
\]
The perfect pairing construction of
Proposition~\ref{prop:splittingGQ}, applied to $h\alpha$,
gives
\[
\langle[\n],[\m]\rangle_{h\alpha}
=\langle\m,(h\alpha)_{\Q}^{-1}(\n)\rangle+\Z.
\]
Define
\begin{equation}\label{s5:eq:self-duality}
\begin{aligned}
\rho_h:\widehat G_Q^{\,h}
&\xrightarrow{\sim}(\widehat G_Q^{\,h})^\vee,\\
\rho_h(a)(a')&:=
\varepsilon\bigl(\langle\bar T_h^{-1}(a),a'\rangle_{h\alpha}\bigr).
\end{aligned}
\end{equation}
These isomorphisms transport the canonical torus actions
to $\widehat G_Q^{\,h}$.

Recall that $\bar h$ is induced by multiplication by $h$
and that $\bar q_h:\widehat G_Q^{\,h}\to G_Q$ is induced
by the identity on $M$. The required compatibility is
\begin{equation}\label{s5:eq:compatibility}
\rho_h\circ\bar h=\bar q_h^\vee\circ\rho_1.
\end{equation}

We first verify the compatibility in
\eqref{s5:eq:compatibility}. For $\m,\m'\in M$, cancellation of
$h$ in the rational pairing gives
\[
\begin{aligned}
\rho_h\big(\bar h(\m+L)\big)(\m'+hL)
&=\varepsilon\!\left(
\big\langle\m',\alpha_{\mathbb Q}^{-1}(T^{-1}(\m))\big\rangle
+\mathbb Z\right)\\
&=\rho_1(\m+L)\big(\bar q_h(\m'+hL)\big).
\end{aligned}
\]
This proves the asserted identity.

Consider any subgroup $H\leq\widehat G_Q^{\,h}$ and set
\[
\begin{aligned}
M_H&:=\{\m\in M\mid
\rho_h(a)(\m+hL)=1\ \text{for every }a\in H\},\\
N_H&:=\{\n\in N_{\Q}\mid
\langle\m,\n\rangle\in\Z\ \text{for every }\m\in M_H\}.
\end{aligned}
\]
Then $hL\subseteq M_H\subseteq M$ and $N\subseteq N_H$.

The quotient by the finite subgroup $\rho_h(H)\subseteq\T$
is the toric variety defined by the same rational fan $\Theta$ in the
lattice $N_H$. The ramification index along the divisor corresponding
to $\mathbb R_{\geq0}\mathbf w_j$ is the positive integer $e_j$
such that $\mathbf w_j=e_j\mathbf w_{j,H}$, where
$\mathbf w_{j,H}$ is primitive in $N_H$.
Since the action is free on the dense torus, condition~\textup{(a)}
is equivalent to $e_j=1$ for every $j$, or, equivalently, to
\[
N_H\cap\Q\mathbf w_j=\Z\mathbf w_j
\qquad(1\leq j\leq m).
\]
This gives the explicit lattice test for admissibility.
Under condition~\textup{(a)}, the quotient map is therefore
a $1$-covering.

Assume this condition. The ray generators are then unchanged in
$N_{\mathbb Q}$. Their relations give the same weight matrix $Q$,
and the anticanonical polytope remains
$P^\circ=\operatorname{Conv}(W^\circ)$, now considered in
$(M_H)_{\mathbb Q}=M_{\mathbb Q}$. Since $\Theta$ is its normal fan,
the quotient is $\mathbb Q$-Fano. Let $k_H$ be its Gorenstein index.
By Proposition~\ref{prop:indici}, $\widehat k\mid k_H$. Using the description of
$L$ as the lattice generated by the columns of
$\widehat k W^\circ$, we obtain
\begin{equation}\label{s5:eq:factor-lattice}
\frac{k_H}{\widehat k}
=\min\{d\geq1\mid dL\subseteq M_H\}
=\exp\big((L+M_H)/M_H\big).
\end{equation}
Evaluation induces a perfect pairing between $M/M_H$ and $H$.
The annihilator in $H$ of $(L+M_H)/M_H$ consists of those elements
whose characters are trivial on $L$. By
\eqref{s5:eq:compatibility}, this is exactly
$H\cap\bar h(G_Q)=H_1$. Therefore
\[
(H/H_1)^\vee\cong(L+M_H)/M_H,
\]
and \eqref{s5:eq:factor-lattice} becomes
\begin{equation}\label{s5:eq:factor-exponent}
\frac{k_H}{\widehat k}=\exp(H/H_1).
\end{equation}
In particular, the factor divides $h$, and it equals $h$ precisely
when condition~\textup{(b)} holds.

The columns of $W$ generate $N$, so $Y$ is simply connected in
codimension~$1$, by the universal $1$-covering construction of
\S~\ref{ssez:molteplicità}. Hence, for an admissible subgroup, the quotient map is
the universal $1$-covering and its deck group is $H$. This proves
the direct assertions in~\textup{(1)}.

Conversely, let $X$ be a $\mathbb Q$-Fano toric variety of weight
matrix $Q$ and factor $h$. Identify its universal $1$-covering with
$Y$, and retain the morphisms $\beta$ and $\gamma$ of Section~\ref{sez:QFano}. Recall (\ref{eq:factorization}):
the factorization $h\alpha^T=\beta^T\circ\gamma$ gives a
surjection on cokernels and, by duality, an injection
\begin{equation}\label{s5:eq:quotient-and-dual}
\begin{aligned}
\bar i_h\colon\widehat G_Q^{\,h}
&\twoheadrightarrow G:=\operatorname{coker}\beta^T,\\
\bar i_h^\vee\colon G^\vee
&\hookrightarrow(\widehat G_Q^{\,h})^\vee.
\end{aligned}
\end{equation}
The second map identifies $G^\vee$ with the deck group acting on
$Y$. Consequently the subgroup labelling $X$ is
\[
H:=\rho_h^{-1}\big(\operatorname{Im}\bar i_h^\vee\big)
\leq\widehat G_Q^{\,h}.
\]
In particular, $H\cong G^\vee\cong G$, where the last isomorphism
is abstract and generally noncanonical.
Equivalently, let
\[
\bar j_h\colon\operatorname{coker}(h\alpha)
\twoheadrightarrow\operatorname{coker}\gamma^T
\]
be induced by $h\alpha=\gamma^T\circ\beta$. The annihilator
calculation of Proposition~\ref{prop:splittingGQ}, applied to these factorizations,
gives
\[
H=\bar T_h(\ker\bar j_h).
\]
The lattice $M_H$ is precisely $\operatorname{Im}\beta^T$.
Since the quotient map is a $1$-covering and $X$ has factor $h$,
conditions~\textup{(a)} and~\textup{(b)} follow.
This also proves the bijection for marked quotient maps.
In Cox coordinates, the same deck action is obtained by
exponentiating a torsion matrix $\Gamma$, as in diagram~(\ref{div-diagram-covering}),
Remark~\ref{rem:azioneGa}, and Examples~\ref{ex:blupP3} and~\ref{ex:QFanoCanonica}.

For~\textup{(2)}, ampleness of $-K_Y$ implies that the maximal
cones of $\Theta$ are the cones over the facets of
$P=\operatorname{Conv}(W)$. Thus $\Theta$ is unique. Moreover,
a toric isomorphism between two admissible quotients identifies
their primitive ray generators. Since these generate the copies
of $N$ inside their respective lattices $N_H$, the isomorphism
restricts to an element of $\operatorname{Aut}(N,P)$.
Conversely, an element of $\operatorname{Aut}(N,P)$ transporting
one subgroup to another induces a toric isomorphism of the
corresponding quotients. This proves the orbit description, just
as in Theorem~\ref{thm:classificazione}\textup{(3)}.

Finally, let $H$ be admissible and consider $H_1$.
The annihilator of an intersection of subgroups is the sum of
their annihilators. Together with
\eqref{s5:eq:compatibility}, this yields the following equality
of character lattices inside $M$:
\begin{equation}\label{s5:eq:unitary-lattice}
M_{H_1}=M_H+L.
\end{equation}
Hence
\[
M_H\subseteq M_{H_1}\subseteq M,
\qquad
N\subseteq N_{H_1}\subseteq N_H.
\]
Primitivity of the columns of $W$ in $N_H$ implies their
primitivity in $N_{H_1}$ as well. Both maps in
\[
Y\xrightarrow{\widehat\pi_1}X_1:=Y/H_1
\xrightarrow{\pi_1}X:=Y/H
\]
are therefore $1$-coverings. The first one is universal
because $Y$ is simply connected in codimension one; hence
$H_1\cong\pi_1^{\acute{e}t}(X_1)^{(1)}$.
Since $L\subseteq M_{H_1}$,
formula~\eqref{s5:eq:factor-lattice} gives that $X_1$ has
index $\widehat k$, hence factor $1$.

To express this in the lattice morphisms used above, choose
$\beta_1^T$ with image $M_{H_1}$. There are injective lattice
morphisms $\gamma_1^T$ and $(\gamma')^T$ such that
\begin{equation}\label{s5:eq:unitary-factorizations}
\beta^T=\beta_1^T\circ\gamma_1^T,
\qquad
\alpha^T=\beta_1^T\circ(\gamma')^T.
\end{equation}
If $B_1$, $C_1$ and $C'$ represent $\beta_1$, $\gamma_1$ and
$\gamma'$, respectively, and $V_1=B_1W$, then
\begin{equation}\label{s5:eq:unitary-matrices}
\begin{gathered}
B^T=B_1^TC_1^T,
\qquad
A^T=B_1^T(C')^T,\\
V=C_1V_1,
\qquad
\widehat k\Lambda=C'V_1.
\end{gathered}
\end{equation}
In particular,
\[
\widehat k V_1^\circ=(C')^T\Lambda^\circ
\]
is integral, in agreement with the index computation.

The deck groups fit into the short exact sequence
\begin{equation}\label{s5:eq:unitary-deck-sequence}
0\longrightarrow H_1\longrightarrow H
\longrightarrow H/H_1\longrightarrow0.
\end{equation}
By~(\ref{eq:h-extension}), the last group is a subgroup of
$\widehat G_Q^{\,h}/\bar h(G_Q)\cong M/hM$.
Its order therefore divides $h^n$, and its exponent is $h$ by
\eqref{s5:eq:factor-exponent}. This proves the assertions on
$\deg\pi_1$.

If $X'\longrightarrow X$ is an intermediate $1$-covering dominated
by $Y$, its character lattice $M'$ satisfies
$M_H\subseteq M'\subseteq M$. If $X'$ has factor $1$, then
$L\subseteq M'$, whence
$M_{H_1}=M_H+L\subseteq M'$. Thus the $1$-covering factors through
$X_1\longrightarrow X$. Description~\eqref{s5:eq:unitary-lattice}
also shows that this construction is independent of $T$ and
$\varepsilon$. This proves~\textup{(3)}.
\end{proof}

\begin{remark}\label{s5:rem:admissibility}
In the Gorenstein situation of Theorem~\ref{thm:classificazione}, one has
$\widehat k=h=1$, and every subgroup of $G_Q$ is admissible.
For a general $\mathbb Q$-Fano weight matrix, the primitivity
condition can exclude subgroups even when $h=1$.
For $h>1$, one must also distinguish factor exactly $h$ from
factor dividing $h$; equation~\eqref{s5:eq:factor-exponent}
makes this distinction explicit.

The quotient by any finite subgroup of the dense torus still
exists as a normal projective toric variety. If the subgroup
is not admissible, however, the quotient map need not be a
$1$-covering, and the quotient need not have the prescribed
weight matrix or factor.
For the whole group $\widehat G_Q^{\,h}$, the quotient lattice
has character lattice $hL$, and the images of the original
ray generators are represented, in suitable bases, by
\[
hAW=h\widehat k\Lambda.
\]
This matrix need not be reduced; for $h>1$ every column is
divisible by $h$. The fan matrix of the actual toric quotient
is obtained by replacing each column by its primitive generator.
Thus non-reducedness detects ramification in codimension~$1$.

Finally, the inclusion
$\operatorname{Im}(h\alpha^T)\subseteq\operatorname{Im}\beta^T$
gives the surjection $\bar i_h$ in
\eqref{s5:eq:quotient-and-dual}. The associated subgroup is
obtained by dualizing and then applying $\rho_h^{-1}$.
\end{remark}

\begin{corollary}\label{s5:cor:parameter-space}
For a $\mathbb Q$-Fano weight matrix $Q$ and $h\geq1$, let
\[
\mathcal H_Q^h:=
\{H\leq\widehat G_Q^{\,h}\mid
H\text{ is admissible of factor }h\}.
\]
This finite set parametrizes the marked $\mathbb Q$-Fano
quotients of weight matrix $Q$ and factor $h$.
Their toric isomorphism classes are parametrized by
$\mathcal H_Q^h/\operatorname{Aut}(N,P)$.
There is a well-defined map
\[
\begin{aligned}
\mathcal H_Q^h&\longrightarrow\mathcal H_Q^1,\\
H&\longmapsto\bar h^{-1}\big(H\cap\bar h(G_Q)\big),
\end{aligned}
\]
which sends a quotient to its unitary $1$-covering and descends
to toric isomorphism classes. This map is not necessarily
surjective.
\end{corollary}

\begin{proof}
The assertions about parametrization and the unitary $1$-covering
follow from Theorem~\ref{s5:thm:classification} and
\eqref{s5:eq:compatibility}.

For the failure of surjectivity, consider $Q=(1,1,1)$, so that
$Y=\mathbb P^2$, $\widehat k=1$, and $G_Q\cong\mathbb Z/3\mathbb Z$.
One has $\mathcal H_Q^1\neq\varnothing$, whereas
$\mathcal H_Q^2=\varnothing$.
Indeed, an admissible subgroup of factor $2$ would have even
order, by \eqref{s5:eq:factor-exponent}, and hence would contain
an involution. Every nontrivial involution in the dense torus
of $\mathbb P^2$ fixes a coordinate line pointwise: represent
it by a diagonal matrix with eigenvalues in $\{1,-1\}$ and
observe that two of the three eigenvalues coincide.
Its quotient is therefore ramified along a divisor,
contrary to condition~\textup{(a)}.
\end{proof}

\begin{remark}[Relationship with existing literature]
\label{s5:rem:bibliographical-comparison}
When $\widehat k=h=1$, Theorem~\ref{s5:thm:classification}
recovers the classification of Theorem~3.2 and its relation with
Batyrev's theory of minimal and maximal reflexive pairs
\cite[Prop.~5.2.8 and Cor.~5.2.9]{Batyrev94};
see also Remark~\ref{rem:Batyrev}.
The description of a toric variety as a finite abelian quotient of
its universal $1$-covering, together with the correspondence between
intermediate $1$-coverings and intermediate lattices, is established
in \cite[Thm~2.2, Rem.~2.4 and Cor.~2.5]{RT-QUOT}
for complete $\mathbb Q$-factorial toric varieties.
The universal $1$-covering construction is extended to non-degenerate
toric varieties, without assuming $\mathbb Q$-factoriality, in
\cite[Thm~3.17]{R-CovMDS}.
These results provide the foundation for the classification
by $1$-coverings developed here.

For $r=1$, the varieties under consideration are fake weighted
projective spaces, and classification procedures at a prescribed
Gorenstein index are already available.
For surfaces, see
\cite[Prop.~3.11--3.12 and Algorithm~3.14]
{HHHS}.
In arbitrary dimension, B\"auerle gives classification algorithms
for lattice simplices with prescribed unit fraction partitions and
local Gorenstein indices; imposing primitivity of the vertices
yields the corresponding fake weighted projective spaces
\cite[Prop.~5.1--5.2, Algorithms~5.5 and~5.7,
Rem.~5.8]{Bauerle25}.
Thus the rank-one case of Theorem~\ref{s5:thm:classification}
provides a group-theoretic description of a classification problem
already addressed by these algorithms.

There is also a precise relation between the factor $h$ and
B\"auerle's factor of a unit fraction partition
\cite[Prop.~16, Def.~20 and Prop.~21]{Bauerle}.
Indeed, let $Q=(q_0,\ldots,q_n)$ be the reduced weight vector of
$X$, and let $Y=\mathbb P(Q)$ be its universal $1$-covering.
Writing $|Q|=\sum_{i=0}^n q_i$ and $k=h\widehat k$, set
\[
\mathcal A_X=(a_0,\ldots,a_n),
\qquad
a_i=\frac{k|Q|}{q_i},
\qquad
\mathcal A_Y=
\left(\frac{\widehat k|Q|}{q_0},\ldots,
      \frac{\widehat k|Q|}{q_n}\right).
\]
Then $\sum_i a_i^{-1}=1/k$ and
$\mathcal A_X=h\mathcal A_Y$.
Since $\mathcal A_Y$ is reduced in the sense of
\cite[Def.~20 and Prop.~21]{Bauerle}, we obtain
\[
\lambda_{\mathcal A_X}
:=\gcd(k,a_0,\ldots,a_n)
=h\,\gcd\!\left(
\widehat k,
\frac{\widehat k|Q|}{q_0},\ldots,
\frac{\widehat k|Q|}{q_n}
\right)
=h.
\]
Consequently, the ratio of Gorenstein indices used here gives a
geometric interpretation of that arithmetic factor.

A related approach through degree matrices and torsion data is
developed by Ghirlanda: see \cite[Thm~3.3]{Ghirlanda}
for the computation of the Gorenstein index of a fake weighted
projective space, and \cite[Algorithm~4.17]{Ghirlanda}
for the classification of the Gorenstein ones.
This comparison concerns the rank-one case and the use of
divisor class group data to describe the quotient action.

The specific contribution of Theorem~\ref{s5:thm:classification}
lies in combining these perspectives into a finite group
parametrization for a fixed weight matrix $Q$ of arbitrary rank
$r=\operatorname{rk}\operatorname{Cl}(X)$ and an exact factor $h$,
including the non-$\mathbb Q$-factorial case.
With the duality identifications fixed above, the marked quotients
are described by subgroups of the single finite group
\[
\widehat G_Q^{\,h}=M/hL,
\qquad L=\operatorname{Im}\alpha^T,
\]
subject to the primitivity of the ray generators $\mathbf w_j$
in $N_H$ and the exact index condition
\[
\exp(H/H_1)=h,
\qquad H_1=H\cap\bar h(G_Q).
\]
Passing to the $\operatorname{Aut}(N,P)$-orbits gives the toric
isomorphism classes after forgetting the marking.
The exponent condition expresses the toric Cartier criterion as
\[
\frac{k}{\widehat k}
=\min\{d\geq1\mid dL\subseteq M_H\}
=\exp\big((L+M_H)/M_H\big)
=\exp(H/H_1).
\]
Moreover, the intermediate $1$-covering $X_1\longrightarrow X$
constructed in Theorem~\ref{s5:thm:classification}\textup{(3)}
is intrinsically determined by
\[
M_{H_1}=M_H+L.
\]
Its universal property among intermediate $1$-coverings of factor $1$
and the divisibilities
$h\mid\deg(X_1\longrightarrow X)\mid h^n$
make explicit the part of the $1$-covering responsible for the
increase of the Gorenstein index.
The resulting formulation relates the index factor, quotient
actions and multiplicity $\operatorname{mult}X=|H|$ within the
same weight-group construction, extending the rank-one picture
to arbitrary class group rank.
\end{remark}

\begin{example}\label{s5:ex:weighted-plane}
This example illustrates the construction in the proof of
Theorem~\ref{s5:thm:classification}. Up to a unimodular change
of basis and a permutation of the rays, it is the example
considered by B\"auerle in~\cite[Ex.~3]{Bauerle}.

Let $X$ be the complete toric surface with fan matrix
\[
V=\begin{pmatrix}1&9&-7\\0&16&-12\end{pmatrix}.
\]
There is a unique complete fan over these three rays.
Its weight matrix and universal $1$-covering are
\[
Q=(1\ \ 3\ \ 4),
\qquad
Y=\mathbb P(1,3,4),
\qquad
W=\mathcal G(Q)=
\begin{pmatrix}1&1&-1\\0&4&-3\end{pmatrix}.
\]
Both varieties are $\mathbb Q$-Fano and $\mathbb Q$-factorial.
The quotient matrix satisfies
\[
B^T=\begin{pmatrix}1&0\\2&4\end{pmatrix},
\qquad V=BW,
\]
so that $\operatorname{mult}X=4$ and
$G:=\operatorname{coker}\beta^T\cong\mathbb Z/4\mathbb Z$.
Neither variety is canonical: $\operatorname{Conv}(W)$ has
two interior lattice points and $\operatorname{Conv}(V)$ has
nine.
The polar polytopes have vertex matrices
\[
W^\circ=
\begin{pmatrix}-1&-1&7\\0&2/3&-2\end{pmatrix},
\qquad
V^\circ=
\begin{pmatrix}-1&-1&7\\1/2&2/3&-4\end{pmatrix}.
\]
Thus $\widehat k=3$, $k=6$, and $h=2$.
After reversing the columns of these two polar matrices,
one may take $Q^\circ=Q$ and $\Lambda^\circ=W$.
The factorizations of Section~\ref{sez:QFano} are then represented by
\[
A^T=\begin{pmatrix}21&-6\\-6&2\end{pmatrix},
\qquad
C=\begin{pmatrix}42&-12\\-24&7\end{pmatrix},
\qquad
2A^T=B^TC.
\]
The Smith normal forms of $B^T$, $A^T$, and $2A^T$ are,
respectively,
\[
\operatorname{diag}(1,4),
\qquad
\operatorname{diag}(1,6),
\qquad
\operatorname{diag}(2,12).
\]
Consequently,
\[
G_Q\cong\mathbb Z/6\mathbb Z,
\qquad
\widehat G_Q^{\,2}\cong
\mathbb Z/2\mathbb Z\oplus\mathbb Z/12\mathbb Z.
\]
Sequence~(\ref{eq:h-extension}) becomes
\[
0\longrightarrow\mathbb Z/6\mathbb Z
\xrightarrow{\bar h}
\mathbb Z/2\mathbb Z\oplus\mathbb Z/12\mathbb Z
\xrightarrow{\bar r_h}(\mathbb Z/2\mathbb Z)^2
\longrightarrow0.
\]
Since $A$ is symmetric, choose $T$ to identify the displayed
coordinate bases of $N$ and $M$. Compatible Smith
identifications can then be chosen such that
\[
\bar h(1)=(0,2),
\qquad
H=\langle(0,3)\rangle\cong\mathbb Z/4\mathbb Z
\subseteq\widehat G_Q^{\,2}
\]
is the subgroup defining $X$. In particular,
\[
H_1=H\cap\bar h(G_Q)=\langle(0,6)\rangle
\cong\mathbb Z/2\mathbb Z,
\qquad
\exp(H/H_1)=2.
\]

A torsion matrix for $X$ is $\Gamma=(0\ \ 3\ \ 1)$.
For a primitive fourth root of unity $\xi$, the corresponding
action on $Y$ can be written as
\[
a\cdot[x_1:x_2:x_3]
=[x_1:\xi^{3a}x_2:\xi^a x_3],
\qquad a\in\mathbb Z/4\mathbb Z.
\]
Restriction to $H_1$ gives
\[
a\cdot[x_1:x_2:x_3]
=[x_1:\eta^a x_2:\eta^a x_3],
\qquad
\eta=\xi^2,\quad a\in\mathbb Z/2\mathbb Z.
\]
Its quotient is the unitary $1$-covering $X_1\longrightarrow X$,
where, with the polar columns ordered compatibly with $\Lambda^\circ$,
\[
V_1=\begin{pmatrix}1&1&-1\\0&8&-6\end{pmatrix},
\qquad
V_1^\circ=
\begin{pmatrix}7&-1&-1\\-1&1/3&0\end{pmatrix}.
\]
The columns of $V_1$ are primitive, and the Gorenstein index
of $X_1$ is $3$. Hence $X_1$ has factor $1$,
$\operatorname{mult}X_1=2$, and $\deg\pi_1=2$.
The polygon $\operatorname{Conv}(V_1)$ has three interior
lattice points, so $X_1$ is also noncanonical.

The matrices in \eqref{s5:eq:unitary-matrices} can be taken to be
\[
B_1^T=\begin{pmatrix}1&0\\0&2\end{pmatrix},
\qquad
C_1^T=\begin{pmatrix}1&0\\1&2\end{pmatrix},
\qquad
(C')^T=\begin{pmatrix}21&-6\\-3&1\end{pmatrix}.
\]
They satisfy $B^T=B_1^TC_1^T$ and $A^T=B_1^T(C')^T$,
as required.

The subgroups of $G_Q\cong\mathbb Z/6\mathbb Z$ are
\[
\{0\},\qquad\langle3\rangle,
\qquad\langle2\rangle,\qquad G_Q,
\]
but
\[
\mathcal H_Q^1=\{\{0\},\langle3\rangle\}.
\]
Indeed, the image of the original ray matrix in the quotient
by the whole group is
\[
AW=\begin{pmatrix}21&-3&-3\\-6&2&0\end{pmatrix},
\]
which is not reduced. To consider the subgroup of order $3$,
one may represent the $G_Q$-action, after a choice of generator,
by $\Gamma_Q=(0\ \ 3\ \ 5)$. Its order-three subgroup acts as
\[
a\cdot[x_1:x_2:x_3]=[x_1:x_2:\zeta^{2a}x_3],
\qquad a\in\mathbb Z/3\mathbb Z,
\]
where $\zeta$ is a primitive third root of unity.
The images of the original ray generators in this quotient
lattice are represented by
\[
V_{\langle2\rangle}=
\begin{pmatrix}3&3&-3\\0&4&-3\end{pmatrix},
\]
whose first and third columns are not primitive. Reducing them
gives the actual fan matrix
\[
V_{\langle2\rangle}^{\mathrm{red}}=
\begin{pmatrix}1&3&-1\\0&4&-1\end{pmatrix},
\qquad
\mathcal G\big(V_{\langle2\rangle}^{\mathrm{red}}\big)
=(1\ \ 1\ \ 4).
\]
Thus this toric quotient exists, but has weight matrix
$(1,1,4)$, and its quotient map is ramified along divisors.
This explains the exclusion of $\langle2\rangle$ and $G_Q$
from $\mathcal H_Q^1$.

Finally, $Q^\circ=Q$ up to the stated column ordering and
$\Lambda^\circ=W$ is reduced. Hence the polar universal
$1$-covering of Remark~\ref{rem:polar-cover} exists, with $Z^\circ\cong Y$.
By Remark~\ref{rem:r=1 Q-modulo} and Theorem~\ref{thm:degree},
\[
|Q|=|Q^\circ|=1+3+4=8,
\qquad
g_Q=6,
\qquad
\widehat g_Q^{\,2}=24.
\]
Since $-3K_Y$, $-3K_{X_1}$, and $-6K_X$ are Cartier,
their self-intersections are
\[
\begin{aligned}
(-3K_Y)^2&=48=g_Q|Q^\circ|=(-3K_{Z^\circ})^2,\\
(-3K_{X_1})^2&=24,
&\operatorname{mult}X_1&=2,\\
(-6K_X)^2&=48,
&\operatorname{mult}X&=4.
\end{aligned}
\]
In particular,
\[
\operatorname{mult}X\,
\frac{(-6K_X)^2}{|Q^\circ|}
=4\,\frac{48}{8}=24=\widehat g_Q^{\,2},
\]
as required by Theorem~\ref{thm:degree}\textup{(5)} and Corollary~\ref{cor:divisibility}.
The estimates requiring canonical singularities do not
apply to these surfaces.
However, Proposition~\ref{thm:fake-wps}\textup{(2)}
applies. Since $X_1$ and $X$ have weight modulus $8$
and Gorenstein indices $3$ and $6$, respectively, it gives
\[
\begin{aligned}
\operatorname{mult}X_1=2
&\leq
\left[\frac{2(3+1)^2}{8}\right]
=4,\\
\operatorname{mult}X=4
&\leq
\left[\frac{2(6+1)^2}{8}\right]
=12.
\end{aligned}
\]
For comparison, the bounds depending only on dimension
and Gorenstein index recalled in
Remark~\ref{rem:degree-conjecture} give $9$ and $18$,
respectively.
This illustrates the benefit of retaining the weight
modulus in the estimate.
\end{example}

\begin{example}[Example~\ref{ex:QFanoCanonica} continued]
\label{ex:QFanoCanonica2}
Retain the weight matrix and universal $1$-covering
\[
Q=\begin{pmatrix}1&1&1&0&0\\0&0&0&1&1\end{pmatrix},
\qquad Y=\mathbb P^2\times\mathbb P^1
\]
of Example~\ref{ex:QFanoCanonica}.
We classify all $\mathbb Q$-Fano toric varieties with
weight matrix $Q$ and factor $h=2$.
Since $\widehat k=1$, these are precisely the varieties
with these weights and Gorenstein index $2$.
Recall that
\[
G_Q\cong\mathbb Z/3\mathbb Z,
\qquad
\widehat G_Q^{\,2}
\cong(\mathbb Z/2\mathbb Z)^2\oplus\mathbb Z/6\mathbb Z.
\]

Let $\T=\Hom(M,\K^*)\cong(\K^*)^3$ be the acting
torus of $Y$. We regard its elements as automorphisms of
$Y$. In the homogeneous coordinates used here,
\[
(a,b,c)\cdot([x_1:x_2:x_3],[x_4:x_5])
=([x_1:ax_2:bx_3],[x_4:cx_5]).
\]
Let $\zeta$ be a primitive third root of unity and let
$\rho$ be the automorphism corresponding to
$(\zeta,\zeta^2,1)$, namely
\[
\rho\bigl([x_1:x_2:x_3],[x_4:x_5]\bigr)
=\bigl([x_1:\zeta x_2:\zeta^2x_3],[x_4:x_5]\bigr).
\]
The notation $\langle\rho\rangle$ denotes the cyclic
subgroup generated by $\rho$:
\[
\langle\rho\rangle=\{\mathrm{id}_Y,\rho,\rho^2\}
\cong\Z/3\Z.
\]
More generally, angle brackets around automorphisms denote
the subgroup they generate, with composition as group law.

Set
\[
E:=\T[2]=\{t\in\T\mid t^2=1\}
=\{(a,b,c)\mid a,b,c\in\{1,-1\}\}
\cong(\Z/2\Z)^3.
\]
Thus $\T[2]$ denotes the subgroup of elements of order
dividing $2$ in the same acting torus $\T$.
With the compatible actions of
Theorem~\ref{s5:thm:classification}, we identify
$\widehat G_Q^{\,2}$ with its image in $\T$. Then
\[
\widehat G_Q^{\,2}=\langle\rho\rangle\times E,
\qquad
\bar h(G_Q)=\langle\rho\rangle.
\]
Indeed, the acting group is the inverse image of
$\langle\rho\rangle$ under $t\mapsto t^2$.
Squaring is an automorphism of $\langle\rho\rangle$,
so this inverse image is $\langle\rho\rangle E$;
the product is direct because the two factors have
coprime orders.
For $i=1,2,3$, let $\varepsilon_i$ act by changing the
signs of $x_i$ and $x_5$, leaving the other coordinates
fixed. In the displayed coordinates on $\T$, these are
\[
\varepsilon_1=(-1,-1,-1),\qquad
\varepsilon_2=(-1,1,-1),\qquad
\varepsilon_3=(1,-1,-1).
\]

We claim that the admissible subgroups are exactly
\[
\mathcal H_Q^2=
\bigl\{
\langle\varepsilon_i\rangle,
\ \langle\rho,\varepsilon_i\rangle
\ \bigm|\ i=1,2,3
\bigr\}.
\]
To prove this, any subgroup $H$ decomposes as
\[
H=H_1\times U,
\qquad
H_1\leq\langle\rho\rangle,
\qquad U\leq E,
\]
because the two factors have coprime orders.
The exponent condition is equivalent to $U\ne\{1\}$.
A nontrivial involution on $\mathbb P^2$ fixes a line,
whereas a nontrivial involution on $\mathbb P^1$ fixes
two points. Thus an involution of $Y$ fixes a divisor
if it acts trivially on either factor, and has fixed
locus of codimension at least two if it acts nontrivially
on both factors.

Consequently, the projection of $U$ to the two-torsion
subgroup of the torus of $\mathbb P^1$ must be injective.
It follows that $|U|=2$, and its generator must act
nontrivially on $\mathbb P^2$.
The elements of $E$ with nontrivial projection to the
$\mathbb P^1$ factor are
$\varepsilon_1$, $\varepsilon_2$, $\varepsilon_3$ and
$\varepsilon_1\varepsilon_2\varepsilon_3$.
The last of these acts as the scalar $-1$ on
$\mathbb P^2$, and is therefore trivial on that factor.
Thus the three possibilities are precisely
$U=\langle\varepsilon_i\rangle$.
Moreover, $\rho$ and $\rho^2$ have fixed loci of
codimension two in $Y$.
Hence all six displayed subgroups act freely in
codimension one: a nontrivial stabilizer of a divisor
would contain an element of prime order.
This proves the claim.

Permutations of the first three homogeneous coordinates
are induced by lattice automorphisms preserving
$\operatorname{Conv}(W)$. They normalize
$\langle\rho\rangle$ and permute the three
$\varepsilon_i$ transitively.
There are therefore two orbits of admissible subgroups,
distinguished by their orders $2$ and $6$.
Thus there are six marked quotients and exactly two
toric isomorphism classes, represented by
\[
X':=Y/\langle\varepsilon_3\rangle,
\qquad
X:=Y/\langle\rho,\varepsilon_3\rangle.
\]
Their multiplicities are $2$ and $6$, respectively.
Taking $\zeta=\xi^2$, the automorphism
$\rho\varepsilon_3$ is the generator of the order-six
action in Example~\ref{ex:QFanoCanonica}, so the second
quotient is the variety $X$ already considered there.
A fan matrix for the first quotient is
\[
V'=B'W,\qquad
B'=\begin{pmatrix}1&0&1\\0&1&1\\0&0&2\end{pmatrix},
\]
\[
V'=\begin{pmatrix}
1&0&-1&1&-1\\
0&1&-1&1&-1\\
0&0&0&2&-2
\end{pmatrix}.
\]
All columns are primitive and $|\det B'|=2$.

In factor $1$, the two subgroups of $G_Q$ give precisely
the Fano varieties $Y$ and
$Z=Y/\langle\rho\rangle=(Z^\circ)^\circ$.
For $X'$, the subgroup $H_1$ is trivial; for $X$, it is
$\langle\rho\rangle$.
Their unitary $1$-coverings are therefore
$Y\longrightarrow X'$ and $Z\longrightarrow X$,
both of degree $2$.
The quotient maps form the commutative diagram
\[
\xymatrix{
Y\ar[r]^{3}\ar[d]_{2}&Z\ar[d]^{2}\\
X'\ar[r]_{3}&X,
}
\]
where the numbers are the degrees of the displayed $1$-coverings.
In particular, the deck sequence for the unitary
$1$-covering of $X$ is
\[
0\longrightarrow\mathbb Z/3\mathbb Z
\longrightarrow\mathbb Z/6\mathbb Z
\longrightarrow\mathbb Z/2\mathbb Z
\longrightarrow0.
\]

The Cartier intersection numbers are
\[
(-K_Y)^3=54,\qquad (-K_Z)^3=18,
\]
and the projection formula gives
\[
(-2K_{X'})^3=\frac{2^3\cdot54}{2}=216,
\qquad
(-2K_X)^3=\frac{2^3\cdot54}{6}=72.
\]
Thus both quotient classes satisfy
\[
\begin{aligned}
\operatorname{mult}X'\,
\frac{(-2K_{X'})^3}{|Q^\circ|}
&=2\,\frac{216}{18}=24,\\
\operatorname{mult}X\,
\frac{(-2K_X)^3}{|Q^\circ|}
&=6\,\frac{72}{18}=24,
\end{aligned}
\]
in agreement with Theorem~\ref{thm:degree} and
$\widehat g_Q^{\,2}=24$.

The polar classification in factor $1$ remains governed
by Theorem~\ref{thm:classificazione}.
The two Fano varieties of weight matrix $Q^\circ$ are
$Z^\circ$ and $Y^\circ$, with
\[
\operatorname{mult}Z=\operatorname{mult}Y^\circ=3,
\qquad
(-K_Z)^3=18,\qquad (-K_{Y^\circ})^3=6.
\]
The varieties $Y$ and $Z$ are $\mathbb Q$-factorial,
whereas $Z^\circ$ and $Y^\circ$ are not: their ray
polytopes are triangular prisms with quadrilateral
facets. Here $r=2$, $r^\circ=3$ and $r'=3$.
Theorem~\ref{thm:bounds}\textup{(2)} bounds the
multiplicities of $Z$ and $Y^\circ$ by $12$.

Finally, the face fan of $\operatorname{Conv}(W)$ is
smooth and has six maximal cones, so
Proposition~\ref{prop:modulus} gives
\[
|Q|=6=3+r'
=\max\bigl\{
f_0(\operatorname{Conv}(W)),
f_2(\operatorname{Conv}(W))
\bigr\}.
\]
Thus the lower volume bound of
Lemma~\ref{lem:volumeverticesfacets} is attained, while
the multiplicities remain strictly below the upper
bound in Theorem~\ref{thm:bounds}.
\end{example}

\section{Extending bounds to Mori Dream Spaces and $\Q$-Fano varieties}\label{sez:MDS}

Let $X$ be a $\Q$-Fano variety, that is a normal, projective variety whose anticanonical divisor admits an ample integer multiple $-kK_X$, for some $k\in\N$. Then Definition~\ref{def:index} of the Gorenstein index $k$ of $X$ still holds.

A celebrated result of Birkar-Cascini-Hacon-McKernan \cite[\S~1.3]{BCHMcK} guarantees that:
\begin{itemize}
  \item[($*$)] \emph{a $\Q$-factorial $\Q$-Fano variety $X$ with klt singularities is a Mori Dream Space} (MDS)
\end{itemize}
where an MDS is a $\Q$-factorial algebraic variety $X$ such that
\begin{itemize}
  \item[(i)] every invertible global function is constant i.e. $H^0(X,\cO_X^*)\cong\K^*$,
  \item[(ii)] the class group $\Cl(X)$ is a finitely generated abelian group of rank
  $$r=\rk\Cl(X)$$
  \item[(iii)] the Cox ring $\Cox(X)$ is a finitely generated $\K$-algebra
  \begin{equation*}
    \Cox(X)\cong\K[x_1,\ldots,x_m]/I
  \end{equation*}
\end{itemize}
More precisely, this is the definition of a \emph{weak} Mori
Dream Space (wMDS) \cite{R-wMDS}, whereas a Mori Dream Space in the sense
of Hu and Keel is a projective wMDS. Throughout this paper,
we use the term MDS in the broader sense above and explicitly
assume projectivity whenever it is required.

 It is a well-known fact that:
 \begin{itemize}
   \item[($**$)] \emph{an MDS $X$ admits a canonical toric embedding}.
 \end{itemize}
The reader is referred to \cite[\S~3.2.5]{ADHL} and \cite[Prop.~2.11]{Hu-Keel} for the details. In the following, we use the same notation as in \cite[\S~2.3]{R-wMDS}. Namely, one gets the following commutative diagram (see \cite[Thm.~2,\,Thm.~3]{R-wMDS}):
\begin{equation*}
  \xymatrix{X\ar@{^(->}[r]^-i&W(\Si)\ar@{^(->}[r]^-\iota&Z(\Si')\\
  \widehat{X}\ar@{->>}[u]^-{\pi_X}\ar@{^(->}[d]\ar@{^(->}[r]^-{\widehat{i}}&\widehat{W}\ar@{^(->}[d]\ar@{->>}[u]^-{\pi}
  \ar@{^(->}[r]^-{\widehat{\iota}}&\widehat{Z}\ar@{->>}[u]_-{\overline{\pi}}\ar@{^(->}[dl]\\
            \overline{X}:=\Spec(\Cox(X))\ar@{^(->}[r]^-{\bar i}&
            \overline{W}:=\Spec(\K[\X])\cong\K^m
            &}
\end{equation*}
where:
\begin{itemize}
  \item $\Cox(X)$ is the Cox ring of $X$,
  \item $W(\Si)$ is a $\Q$-factorial toric variety and $i$ a closed embedding of $X$ in $W$,
  \item $\overline{X}$ and $\overline{W}$ are the \emph{total coordinate spaces} of $X$ and $W$, respectively,
  \item $\widehat{X}, \widehat{W}, \widehat{Z}$ are the \emph{characteristic spaces} of $X,W,Z$, respectively, which are suitable Zariski open subsets inside the associated total coordinate spaces,
  \item $\X=\{x_1,\ldots,x_m\}$ is a \emph{Cox basis} of $\Cox(X)$ \cite[Def.~12]{R-wMDS}, so that $m=n+r$, where $n=\dim W$ and $r=\rk\Cl (W)$, and $\Cox(X)\cong\K[\X]/I$, with $I$ the ideal giving the embedding $\bar i$,
  \item $Z(\Si')$ is a (non-unique) \emph{sharp completion} of $W(\Si)$, that is $Z$ is $\Q$-factorial and complete and has the same Picard number as $W$ \cite[\S~2.6]{R-wMDS}: in particular, there is an inclusion of fans $\Si\subseteq\Si'$, but $\Si(1)=\Si'(1)$. A general MDS $X$, as here defined, does not admit a similar sharp completion $Z$: the latter exists if $X$ is projective by Hu and Keel \cite[Prop.~2.11]{Hu-Keel} and, more generally, if $X$ is \emph{fillable} \cite[Thm.~3]{R-wMDS}.
\end{itemize}

\begin{remark}
  Although $i:X\hookrightarrow W$ is a canonical embedding up to isomorphisms, in the sense that it depends only on the cardinality $|\X|$, which is fixed to be the minimum, the sharp completion $\iota:W\hookrightarrow Z$, if existing, is not unique. But, if $\iota':W\hookrightarrow Z'$ is another sharp completion, then there exists an isomorphism in codimension one $Z\cong_1 Z'$, which is actually a small $\Q$-factorial modification.

  In particular, if $X$ is already a complete toric variety, then both the embedding $i$ and $\iota$ are trivial, that is
  \begin{equation}\label{MDStorico}
  X=W=Z
  \end{equation}
\end{remark}

\begin{theorem}[\cite{R-CovMDS}, Thm.~4.16]\label{thm:can_1-covering} An MDS $X$ admits a canonical 1-covering $\phi:\widetilde{X}\twoheadrightarrow X$ and a canonical embedding $\widetilde{i}:\widetilde{X}\hookrightarrow\widetilde{W}$ into the universal 1-covering $\widetilde{W}$ of $W$. They fit into the following commutative diagram
\begin{equation}\label{diag-quot}
  \xymatrix{\widetilde{X}\ar@{^(->}[r]^-{\widetilde{i}}\ar@{->>}[d]^-\phi&\widetilde{W}
  \ar@{^(->}[r]^-{\widetilde{\iota}}\ar@{->>}[d]^-\vf&\widetilde{Z}\ar@{->>}[d]
  ^-\psi\\
  X\ar@{^(->}[r]^-i&W\ar@{^(->}[r]^-\iota&Z}
\end{equation}
where $Z$, if it exists, is a sharp completion of $W$ and $\psi:\widetilde{Z}\twoheadrightarrow Z$ its universal 1-covering.
\end{theorem}

If a sharp completion $\iota:W\hookrightarrow Z$ exists,
the composite $\iota\circ i$ is a locally closed embedding;
it is closed when $X$ is complete.
The induced maps on divisor class groups are described
in the following theorem.
In particular, both embeddings $i$ and $\iota\circ i$ in the statement of Theorem~\ref{thm:can_1-covering} turn out to be \emph{neat} embeddings \cite[Def.~13]{R-wMDS} meaning that the induced pull-back on divisor classes are isomorphisms.

\begin{theorem}[\cite{R-wMDS}, Thm.~3]\label{thm:pullback} Let $X$ be an MDS and consider the canonical toric embedding $i:X\hookrightarrow W$ and, if it exists, a sharp completion $\iota:W\hookrightarrow Z$. Then, there is an induced commutative diagram of group homorphisms
      \begin{equation*}
        \xymatrix{\Cl(Z)\ar[r]^-{\iota^*}_-{\cong}&\Cl(W)\ar[r]^-{i^*}_-{\cong}&\Cl(X)\\
                    \Pic(Z)\ar@{^{(}->}[u]\ar@{^{(}->}[r]^-{\iota^*}&\Pic(W)\ar@{^{(}->}[u]
                    \ar[r]^-{i^*}_-{\cong}&\Pic(X)\ar@{^{(}->}[u]}
      \end{equation*}
Moreover, homorphisms $\iota^*$ and $i^*$ extend to give $\R$-linear morphisms
      \begin{equation*}
        \xymatrix{\Cl(Z)_\R\ar[r]^-{\iota^*_\R}_-{\cong}&\Cl(W)_\R\ar[r]^-{i^*_\R}_-{\cong}&\Cl(X)_\R\\
                    \overline{\Eff}(Z)\ar@{^{(}->}[u]\ar[r]^-{\iota^*_\R}_-{\cong}&\overline{\Eff}(W)\ar@{^{(}->}[u]
                    \ar[r]^-{i^*_\R}_-{\cong}&\overline{\Eff}(X)\ar@{^{(}->}[u]\\
                    \overline{\Mov}(Z)\ar@{^{(}->}[u]\ar[r]^-{\iota^*_\R}_-{\cong}&\overline{\Mov}(W)\ar@{^{(}->}[u]
                    \ar[r]^-{i^*_\R}_-{\cong}&\overline{\Mov}(X)\ar@{^{(}->}[u]\\
                    \g=\Nef(Z)\ar@{^{(}->}[u]\ar@{^{(}->}[r]^-{\iota^*_\R}&\Nef(W)\ar@{^{(}->}[u]
                    \ar[r]^-{i^*_\R}_-{\cong}&\Nef(X)\ar@{^{(}->}[u]}
      \end{equation*}
\end{theorem}

\subsection{The multiplicity of a Mori Dream Space}\label{ssez:MDSmult}
The argument proving Theorem~\ref{thm:can_1-covering}, as given in \cite{R-CovMDS}, allows one to conclude the following

\begin{corollary}\label{cor:mm}
Projections $\phi,\vf,\psi$ in diagram (\ref{diag-quot}) are given as geometric quotients by successive restrictions of the natural action of the finite abelian group
$$\mm:=\Hom(\Tors(\Cl(W)),\K^*)\cong\Tors(\Cl(W))\cong\Tors(\Cl(Z))$$
on the total coordinate space $\overline{W}\cong\Spec(\K[\X])\cong\overline{Z}$, up to the action of the torus $H:=\Hom(\Cl(\widetilde{W}),\K^*)\cong\Hom(\Cl(\widetilde{Z}),\K^*)$.
\end{corollary}

\begin{proof}[Sketch of proof] Let us give the global picture of the action of $\mm$. For details, the interested reader is referred to the proof of Thm.~4.16 in \cite{R-CovMDS}.

\noindent Given the universal 1-covering $\vf:\widetilde{W}\twoheadrightarrow W$, we get the following short exact sequence of
  abelian groups:
  \begin{equation*}
    \xymatrix{1\rightarrow\Hom(\Cl(\widetilde{W}),\K^*)\
    \ar@{^(->}[r]^-{\vf^*}&\Hom(\Cl(W),\K^*)\ar[r]&\Hom(\Tors(\Cl(W)),\K^*)\rightarrow1}
  \end{equation*}
  Since $\Cl(\widetilde{W})$ is free, $H:=\Hom(\Cl(\widetilde{W}),\K^*)$ turns out to be a full subtorus of the
  quasi-torus $T:=\Hom(\Cl(W),\K^*)\cong\Hom(\Cl(X),\K^*)$, giving rise to the finite quotient
\begin{equation*}
  \boldsymbol\mu:=T/H\cong \Hom(\Tors(\Cl(W)),\K^*)\cong\Hom(\Tors(\Cl(X)),\K^*)
\end{equation*}
where the last isomorphism comes from the pull-back $i^*:\Cl(W)\cong\Cl(X)$.
Furthermore, we have an identification of total coordinate spaces
\begin{equation*}
 \left. \begin{array}{c}
    \overline{W}=\Spec(\Cox(W)) \\
    \overline{\widetilde{W}}\cong\Spec(\Cox(\widetilde{W})) \\
    \overline{Z}=\Spec(\Cox(Z)) \\
    \overline{\widetilde{Z}}\cong\Spec(\Cox(\widetilde{Z})) \\
  \end{array}\right\}\cong\Spec\K[\X]\cong\K^m
\end{equation*}
where $m=|\X|$. Then the quasi-torus $T$ can be thought of simultaneously acting on those total coordinate spaces, so that the open embedding $\iota:W\hookrightarrow Z$ comes from an open embedding of stable points loci $\widehat{\iota} : \widehat{W}\hookrightarrow\widehat{Z}$ determined by the inclusion of fans $\Si\subseteq\Si'$. Thinking of the partial torus action of $H$, one has a natural inclusion of universal 1-coverings
\begin{equation}\label{inclusionestabile}
  \xymatrix{\widetilde{W}:=\widehat{W}/H\ \ar@{^(->}[r]&\widehat{Z}/H=:\widetilde{Z}}
\end{equation}
on which naturally acts the finite abelian group $\mm$. Moreover, recalling that
\begin{equation*}
  \xymatrix{\overline{X}=\Spec(\Cox(X))\cong \Spec\left(\displaystyle{\K[\X]\over I}\right)\ \ar@{^(->}[r]&\K^m}
\end{equation*}
the restriction of the torus action $H$ to the subset of stable points $\widehat{X}\hookrightarrow\widehat{W}$ extends inclusion (\ref{inclusionestabile}) to give the following chain of inclusions
\begin{equation*}
  \xymatrix{\widetilde{X}:=\widehat{X}/H\ \ar@{^(->}[r]&\widetilde{W}:=\widehat{W}/H\ \ar@{^(->}[r]&\widehat{Z}/H=:\widetilde{Z}}
\end{equation*}
over which naturally acts $\mm$, giving rise to quotient maps $\phi,\vf,\psi$ and diagram (\ref{diag-quot}).
\end{proof}

\begin{definition}\label{def:multMDS}
  Let $X$ be an MDS, possibly a $\Q$-factorial, $\Q$-Fano, klt variety. Recall the finite abelian group $\mm$ defined in Corollary~\ref{cor:mm}. Define the \emph{multiplicity} of $X$, denoted by $\mult X$, as
  \begin{equation*}
    \mult X:=|\mm|
  \end{equation*}
  If $X$ admits a sharp completion $Z$ of the canonical ambient toric variety $W$, then $\mult X=\mult Z$.
\end{definition}

\begin{remark}
  Since the group $\mm$ only depends on the canonical ambient variety $W$, the given definition of $\mult X$ does not depend on the choice of the sharp completion $Z$ of $W$.
\end{remark}

\begin{proposition}\label{prop:MDSbound}
  Let $X$ be an MDS, possibly a $\Q$-factorial, $\Q$-Fano, klt variety and let
   \begin{itemize}
     \item $i:X\hookrightarrow W$ be its canonical toric embedding,
     \item $\iota:W \hookrightarrow Z$ be the choice of a sharp completion of $W$, assuming that $Z$ exists,
     \item $V$ be a fan matrix and $Q=\G(V)$ be a weight matrix of $Z$ (hence $W$) and $g_Q$ the associated weight order.
   \end{itemize}
   If $Z$ is $\Q$-Fano then
   \begin{enumerate}
     \item $\mult X=\mult Z= n!\displaystyle{\Vol(\conv(V))\over |Q|}$, with $n=\dim Z$,
     \item $\mult X\,|\,h^n g_Q=\widehat{g}^h_Q$\,, where $h$ is the factor of $Z$.
   \end{enumerate}
\end{proposition}

\begin{proof}
The first equality in item (1) comes directly from Definition~\ref{def:multMDS}. The second equality is obtained by applying Theorem~\ref{thm:degree}~(1) to the $\Q$-Fano toric variety $Z$. Item (2) follows by applying Corollary~\ref{cor:divisibility}~(1) to $Z$.
\end{proof}

 \begin{corollary}\label{cor:MDSbounds}
   Let $X$ be an MDS of rank $r$, possibly a $\Q$-factorial, $\Q$-Fano, klt variety, and let
  \begin{equation*}
    \xymatrix{X\,\ar@{^(->}[r]^-i&W\,\ar@{^(->}[r]^-\iota&Z}
  \end{equation*}
  be its canonical toric embedding followed by the choice of a sharp completion of the ambient toric variety. Assume that $Z$ is $\Q$-Fano of Gorenstein index $k$ and admits at worst canonical singularities. Then
  \begin{enumerate}
      \item if $n=\dim Z=2$ then $Z$ is Fano and
      \begin{equation*}
        \mult X\le \left[{9\over 2+r'}\right]
      \end{equation*}
      \item if $n=\dim Z=3$ then
      \begin{equation*}
        \mult X\le \left[\frac{72k^3}{3+r'}\right]
      \end{equation*}
      \item if $n=\dim Z \ge 4$ then
      \begin{equation*}
        \mult X\le \left[
\frac{2(s_n-1)^2k^n}{n+r'}
\right]
      \end{equation*}
    \end{enumerate}
    where $[\,\cdot\,]$ denotes the integer part of a rational number, $\{s_n\}$ is the Sylvester sequence, $r'=\max(r,r^\circ)$ where $r^\circ$ is the dual rank of $Z$.
 \end{corollary}

 \begin{proof}
 The statement is an immediate consequence of Definition~\ref{def:multMDS} and Theorem~\ref{thm:canonical-bound} applied to the $\Q$-Fano toric variety $Z$.
 \end{proof}

 \begin{example}\label{ex:MDS}
This example illustrates the results of
\S~\ref{ssez:MDSmult} for a non-toric $\Q$-factorial
$\Q$-Fano threefold of rank $2$ and multiplicity $4$.

Consider the grading map
\[
d:\Z^6\twoheadrightarrow K:=\Z^2\oplus\Z/4\Z
\]
whose free and torsion parts are represented, respectively, by
\[
Q=
\begin{pmatrix}
1&1&1&0&0&0\\
0&0&0&1&1&1
\end{pmatrix}
=\begin{pmatrix}\q_1&\cdots&\q_6\end{pmatrix},
\qquad
\Ga=
\begin{pmatrix}
\overline0&\overline1&\overline2&
\overline0&\overline3&\overline2
\end{pmatrix}.
\]
Set
\[
S:=\C[x_1,\ldots,x_6],\qquad
F:=x_1x_4+x_2x_5+x_3x_6,\qquad
R:=S/(F).
\]
The relation is homogeneous of degree
$\delta=(1,1,\overline0)\in K$.
Write $\overline{x}_i$ for the image of $x_i$ in $R$.

We first verify the properties needed for the Cox ring
construction. The affine quadric $\Spec R$ is normal:
it has dimension $5$ and its only singular point is the
origin. Moreover,
\[
R/(\overline{x}_1)
\cong
\C[x_2,x_3,x_4,x_5,x_6]/(x_2x_5+x_3x_6)
\]
is a domain, whereas
\[
R[\overline{x}_1^{-1}]
\cong
\C[x_1^{\pm1},x_2,x_3,x_5,x_6]
\]
is factorial. Nagata's criterion therefore shows that
$R$ is factorial, hence factorially $K$-graded.
The same argument for the other variables shows that
$\X:=\{\overline{x}_1,\ldots,\overline{x}_6\}$ consists
of pairwise nonassociated homogeneous prime elements.
These elements minimally generate $R$, and the degrees
of any five of them generate $K$.
Thus $\X$ is admissible for the bunched ring construction.
Taking the bunch consisting of the single cone
$\langle\q_1,\q_4\rangle=\R_{\geq0}^2$ gives a
projective $\Q$-factorial MDS $X$ with
\[
\Cl(X)\cong K,\qquad
\Cox(X)\cong R,\qquad
\dim X=5-2=3
\]
and Cox basis $\X$; see \cite[\S~3.2]{ADHL}.
In particular, $X$ is not toric, since its Cox ring is
not a polynomial ring.
More explicitly, let
\[
\begin{aligned}
\widetilde{\irr}
&:=(x_1x_5,x_1x_6,x_2x_4,x_2x_6,x_3x_4,x_3x_5)
\subset S,\\
\irr(X)&:=\sqrt{\widetilde{\irr}\,R},\\
\overline W&:=\Spec S,\qquad
\overline X:=\Spec R,\\
\widehat W&:=\overline W\setminus\CV(\widetilde{\irr}),
\qquad
\widehat X:=\overline X\cap\widehat W.
\end{aligned}
\]
The characteristic quasitorus is
\[
G:=\Hom(K,\C^*)\cong H\times\mm_4,
\qquad H\cong(\C^*)^2,
\]
and the corresponding geometric quotients give the
canonical toric embedding
\[
X=\widehat X/G\hookrightarrow
W:=\widehat W/G.
\]
Consider the reduced CF-matrix
\[
\widetilde V:=\G(Q)=
\begin{pmatrix}
1&0&-1&0&0&0\\
0&1&-1&0&0&0\\
0&0&0&1&0&-1\\
0&0&0&0&1&-1
\end{pmatrix}.
\]
The finite quotient prescribed by $\Ga$ is represented
by the lattice inclusion with matrix
\[
B=
\begin{pmatrix}
1&0&0&2\\
0&1&0&1\\
0&0&1&2\\
0&0&0&4
\end{pmatrix}.
\]
Consequently, a reduced fan matrix for $W$ is
\[
V=B\widetilde V=
\begin{pmatrix}
1&0&-1&0&2&-2\\
0&1&-1&0&1&-1\\
0&0&0&1&2&-3\\
0&0&0&0&4&-4
\end{pmatrix}.
\]
Indeed, $QV^T=0$, $\Ga V^T=0$ in $(\Z/4\Z)^4$,
and the columns of $V$ are primitive and generate a
sublattice of index $4$.

For $1\leq i\leq3$ and $4\leq j\leq6$, write
\[
\sigma_{ij}:=\langle V^{\{i,j\}}\rangle.
\]
The fan $\Si$ of $W$ consists of all faces of the
six maximal cones
\[
\Si(4)=
\{\sigma_{15},\sigma_{16},\sigma_{24},
  \sigma_{26},\sigma_{34},\sigma_{35}\}.
\]
Their torus-fixed points have Cox coordinates supported
on the corresponding pairs $\{i,j\}$ and belong to $X$,
since $F$ vanishes at those points.
The fan consisting of all faces of the nine cones
\[
\overline\Si(4)=
\{\sigma_{ij}\mid 1\leq i\leq3,\ 4\leq j\leq6\}
\]
is complete and simplicial. It defines a sharp
completion $Z:=Z(\overline\Si)$ of $W$, with
\[
W=Z\setminus\{p_{14},p_{25},p_{36}\},
\]
where the three omitted torus-fixed points have Cox
coordinates
\[
p_{14}=[1:0:0:1:0:0],\quad
p_{25}=[0:1:0:0:1:0],\quad
p_{36}=[0:0:1:0:0:1].
\]
None of these points belongs to $X$, since $F$ does
not vanish at them. Thus $W$ is non-complete, while
$X$ is already complete, and the completion adds
exactly these three points.
The universal $1$-covering of $Z$ is
\[
\widetilde Z\cong\P^2\times\P^2,
\]
with fan matrix $\widetilde V$. Moreover,
\[
\widetilde X=
\{x_1x_4+x_2x_5+x_3x_6=0\}
\subset\P^2\times\P^2
\]
is the smooth incidence threefold, and
$X=\widetilde X/\mm_4$.
For a primitive fourth root of unity $\zeta$, the
generator of $\mm_4$ acts by
\[
([x_1:x_2:x_3],[x_4:x_5:x_6])
\longmapsto
([x_1:\zeta x_2:\zeta^2x_3],
 [x_4:\zeta^3x_5:\zeta^2x_6]).
\]
The fixed loci of nontrivial elements have codimension
at least two in both $\widetilde Z$ and $\widetilde X$.
Hence both quotient maps are $1$-coverings, and $X$
has klt quotient singularities.

Under the identifications of the real divisor class
spaces induced by the grading, one has
\[
\Nef(X)=\Nef(W)=\Nef(Z)
=\Mov(Q)=\Eff(Q)=\langle\q_1,\q_4\rangle.
\]
The anticanonical classes, including their torsion
components, are
\[
[-K_Z]=\sum_{i=1}^6d(\e_i)=(3,3,\overline0),
\qquad
[-K_X]=\sum_{i=1}^6d(\e_i)-\delta
       =(2,2,\overline0),
\]
where the second identity follows by adjunction;
see \cite[Prop.~3.3.3.2]{ADHL}.
Their free parts lie in the interior of the respective
nef cones. Therefore both $Z$ and $X$ are $\Q$-Fano.
For completeness, the Cartier subgroup of $\Cl(X)$ is
\[
\begin{aligned}
\Pic(X)
&=\bigcap_{\substack{1\leq i\leq3,\ 4\leq j\leq6\\j\neq i+3}}
  \bigl(\Z d(\e_i)+\Z d(\e_j)\bigr)\\
&=4\Z\oplus4\Z\oplus\{\overline0\}\subset K.
\end{aligned}
\]
Thus the Gorenstein index of $X$ is $2$.
The anticanonical polytope of $Z$ is
$\D_{-K_Z}=\conv(V^\circ)$, where
\[
V^\circ=
\begin{pmatrix}
-1&-1&-1&-1&-1&-1&2&2&2\\
-1&-1&-1&2&2&2&-1&-1&-1\\
-1&-1&2&-1&-1&2&-1&-1&2\\
1&7/4&-1/2&1/4&1&-5/4&-1/2&1/4&-2
\end{pmatrix}.
\]
In particular, the Gorenstein index of $Z$ is $k=4$.
A direct lattice computation gives
\[
\Int(\conv(V))\cap\Z^4=\{\0\},
\]
so $Z$ has at worst canonical singularities.
The cover $\widetilde Z=\P^2\times\P^2$ has index
$\widehat k=1$, and the factor of $Z$ is therefore
$h=k/\widehat k=4$.
By Definition~\ref{def:multMDS},
\[
\mult X=\mult Z=|\det B|=4.
\]
The smooth face fan of $\conv(\widetilde V)$ has
nine maximal cones, each contributing one to the
normalized volume. Consequently,
\[
|Q|=4!\Vol(\conv(\widetilde V))=9,
\qquad
4!\Vol(\conv(V))=36,
\]
and Proposition~\ref{prop:MDSbound}\textup{(1)} gives
the same multiplicity $36/9=4$.

Let us now compare the upper bounds. Since
$n=\dim Z=4$, $r=2$, and $V^\circ$ has nine columns,
one has
\[
r^\circ=9-4=5,\qquad r'=5.
\]
Corollary~\ref{cor:MDSbounds}\textup{(3)}, with
$s_4=43$, yields
\[
\mult X\leq
\left[
\frac{2(43-1)^2\,4^4}{4+5}
\right]
=100\,352.
\]
A better bound follows from
Proposition~\ref{prop:MDSbound}\textup{(2)}.
In the notation of \eqref{eq:A}, the polar matrix
of the universal $1$-covering is
\[
\widetilde V^\circ=
\begin{pmatrix}
-1&-1&-1&-1&-1&-1&2&2&2\\
-1&-1&-1&2&2&2&-1&-1&-1\\
-1&-1&2&-1&-1&2&-1&-1&2\\
-1&2&-1&-1&2&-1&-1&2&-1
\end{pmatrix},
\qquad
V^\circ=B^{-T}\widetilde V^\circ.
\]
We may take
\[
\L^\circ=\G(Q^\circ)=
\begin{pmatrix}
1&0&0&0&-1&-1&1&0&0\\
0&1&0&0&1&0&-1&0&-1\\
0&0&1&0&0&1&-1&-1&0\\
0&0&0&1&1&1&-1&-1&-1
\end{pmatrix},
\]
whose columns generate $\Z^4$ and which satisfies
\[
\widehat k\,\widetilde V^\circ=A^T\L^\circ,
\qquad
A^T=
\begin{pmatrix}
-1&-1&-1&-1\\
-1&-1&-1&2\\
-1&-1&2&-1\\
-1&2&-1&-1
\end{pmatrix}.
\]
Since $\SNF(A^T)=\diag(1,3,3,3)$, it follows that
\[
G_Q=\coker(A^T)\cong(\Z/3\Z)^3,
\qquad g_Q=27.
\]
Therefore Proposition~\ref{prop:MDSbound}\textup{(2)}
gives
\[
\mult X\mid h^n g_Q=4^4\cdot27=6912
\]
\end{example}

\section{Final remarks and open problems}\label{sez:problemi}
This final section is dedicated to focusing on some by-products of the previous results leading to interesting perspectives for future research.

\subsection{Weight Matrices Analysis (WMA): an alternative computational approach to classification}\label{ssez:WMA}

It is well-known that Fano toric varieties of dimension less than or equal to 4 are systematically classified by the Kreuzer and Skarke approach via PALP (Polytope Analysis and Lattice Program) \cite{KS},\cite{PALP}.

An alternative approach can be performed as a computational by-product of the topological classification given in \S~\ref{sez:Fano} and in particular by Theorem~\ref{thm:classificazione}. The idea is following the same argument proposed by Corollary~\ref{cor:dim2}, in the easiest 2-dimensional case, and by Example~\ref{ex:blupP3}.

Throughout this procedure, weight matrices are considered up to the
natural Gale equivalence
\[
Q\sim Q'
\quad\Longleftrightarrow\quad
Q'=A Q\Pi,
\qquad
A\in\GL_r(\mathbb Z),\quad
\Pi\in\mathfrak S_m,
\]
where $A$ corresponds to a change of basis in the target lattice of
the Gale map and $\Pi$ permutes the columns, hence the rays.
The same convention is used for fan matrices, with the corresponding
change of basis in the lattice of one-parameter subgroups.
At each stage, one chooses a canonical representative of every
equivalence class.
\begin{description}
\item[Step 1] \emph{enumerate one canonical representative of each
equivalence class of admissible Fano weight matrices $Q$, modulo equivalence $\sim$,
and sort them according to the rank $r$ of the class group and the
dimension $n$;}
  \item[Step 2] \emph{by Gale duality, construct the associated fan matrices $V=\G(Q)$ to the weight matrices produced in Step 1}: the unique fan spanned by the faces of $\conv(V)$ gives a Fano toric variety simply connected in codimension one; in this way we list all the Fano toric universal 1-coverings;
  \item[Step 3] \emph{by polar duality, the vertices of the anticanonical polytope of the Fano universal 1-coverings produced in  Step 2 give a fan matrix $V^\circ$ of a maximal Fano toric quotient};
  \item[Step 4] \emph{by double Gale duality compute the fan matrix $\L^\circ=\G(\G(V^\circ))$} so getting the fan matrices of the universal 1-coverings of the polar Fano toric varieties obtained in Step 3;
  \item[Step 5] \emph{by polar duality, the anticanonical polytope of the Fano universal 1-coverings produced in Step 4 give the fan matrix $\L=(\L^\circ)^\circ$ of a maximal Fano toric quotient of Fano toric universal 1-coverings produced in Step 2};
  \item[Step 6] \emph{compute the torsion group $G$ and matrix $\Ga$ of the maximal Fano toric quotient produced in Step 5};
  \item[Step 7] \emph{for each normalized weight matrix $Q$ obtained in
Step~1, compute the quotients of the corresponding universal
$1$-covering by all subgroups of $G_Q$. After this computation,
identify subgroups lying in the same $\operatorname{Aut}(N,P)$-orbit,
where
\[
\operatorname{Aut}(N,P)
=
\{u\in\operatorname{GL}(N)\mid u(P)=P\}.
\]
The resulting list contains one representative for each toric
isomorphism class of Fano toric varieties admitting the corresponding
weight matrix.}
\end{description}

\begin{remark}
The procedure naturally produces marked quotient presentations
first. These must be reduced in two stages. First, weight matrices
are identified under the Gale equivalence
$Q\sim A Q\Pi$. Second, for a fixed normalized $Q$, quotients
corresponding to subgroups in the same
$\operatorname{Aut}(N,P)$-orbit are identified. Thus, in the
Gorenstein Fano case, the final output for a fixed $Q$ is the orbit
set
\[
\{H\leq G_Q\}/\operatorname{Aut}(N,P),
\]
or, in the notation of Corollary~\ref{s5:cor:parameter-space},
\[
\mathcal H_Q^1/\operatorname{Aut}(N,P).
\]
This quotient is necessary because distinct subgroups may define
torically isomorphic quotients after the marking is forgotten.

  Steps from 2 to 5 are in a sense standard: there are many procedures realising Gale duality and polar duality. In particular, jointly with L.~Terracini, we wrote working and efficient MAPLE procedures implementing all these steps. We also wrote MAPLE procedures for Step 6, by following items (iv) and (v) in \cite[Rem.~3.3]{RT-QUOT}. Step 7 requires some further work, due to the need of listing all subgroups of the torsion group $G$ and computing the quotient fan matrices for the action encoded in the torsion matrix $\Ga$ and restricted to those subgroups: by the way, although these procedures do not yet exist, implementing them does not seem difficult, by following the steps describing the quotients for subgroups (\ref{sottogruppi}) in Ex.~\ref{ex:blupP3}.

   Instead, Step 1 represents the real theoretical difficulty in designing an efficient procedure. The main point is that conditions from (W.a) to (W.f) in \S~\ref{sssez:matrici} defining a $W$-matrix are not sufficient to create a terminating algorithm. For the rank $r=1$, results by Kasprzyk \etal, here recalled with Theorems~\ref{thm:KX3} and \ref{thm:KXn}, may help to impose an upper bound for entries of a weight matrix (in this particular case a weight vector). This condition allows us to end up the generation of rank 1 weight matrices, that is, the classification of rank 1 Fano toric varieties, although there is some efficiency problem due to the fact that the Sylvester sequence grows rapidly so that
     \begin{equation*}
       \{2(s_n-1)/s_1\}=\{s_n-1\}=\{\prod_{i=1}^{n-1}s_i\}=\{1,\,2,\,6,\,42,\,1806,\, 3263442, \dots\}
     \end{equation*}
     For $r=1$, the varieties produced by this procedure are
Gorenstein fake weighted projective spaces, and their fan
polytopes are precisely the reflexive simplices.
The classification of these objects by means of weights and
finite lattice extensions has already been studied by Conrads
\cite{Conrads}; see also \cite[\S~9.2]{Kasprzyk-thesis}.
In dimensions at most four, they can be extracted from the Kreuzer-Skarke
classification of reflexive polytopes recalled above
\cite{KS,PALP}.
More recently, Ghirlanda has developed a classification
algorithm based on degree matrices and torsion data, obtaining
complete lists of reflexive simplices in dimensions five and six
\cite[Classification~1.1 and Cor.~1.2]{Ghirlanda}.

The comparison with classifications under the terminal
singularity assumption requires an additional restriction.
Indeed, Gorenstein Fano varieties have canonical singularities,
but need not be terminal. Among the reflexive simplices
produced by the procedure, the terminal ones are precisely those
whose only lattice points are the origin and the vertices.
For example, in dimension three, these form the Gorenstein
part of the classification of terminal fake weighted projective
spaces in \cite[Ch.~6]{Kasprzyk-thesis}.

These existing classifications provide concrete tests for an
implementation of the WMA procedure. In particular, the
computational difficulties mentioned above concern the
implementation of this procedure, since the classification
problem in rank one has already been solved in the dimensions
just recalled.
The interest of the proposed approach lies in organizing the
enumeration by the weight matrix $Q$ and its universal
$1$-covering, and in recovering the corresponding Fano quotients
from the subgroups of $G_Q$, as described in
Theorem~\ref{thm:classificazione}.
This description applies in arbitrary class group rank and
includes non-$\Q$-factorial varieties.

For $r\geq2$, the main challenge in Step~1 is to obtain an
efficient enumeration of admissible Fano weight matrices for
fixed dimension and rank.
The question of bounding the entries of a weight matrix is meaningful
only after a normal form has been fixed. Without such a normalization,
the entries of a matrix can be made arbitrarily large without
changing the underlying Gale data. For example, let
\[
Q_0=
\begin{pmatrix}
1&1&0&0\\
0&0&1&1
\end{pmatrix},
\qquad
U_t=
\begin{pmatrix}
1&t\\
0&1
\end{pmatrix}
\in\operatorname{GL}_2(\mathbb Z).
\]
Then
\[
U_tQ_0=
\begin{pmatrix}
1&1&t&t\\
0&0&1&1
\end{pmatrix},
\qquad t\geq0,
\]
represents the same Gale data as $Q_0$, although its entries are
unbounded.

Accordingly, for $r\geq2$, the relevant open problem is to obtain
effective bounds for canonical representatives of the
$\operatorname{GL}_r(\mathbb Z)\times\mathfrak S_m$-equivalence
classes of admissible Fano weight matrices, or to find an intrinsic
finite parameter set from which these classes can be enumerated.
 The finiteness of reflexive polytopes in fixed
dimension guarantees finiteness of the relevant equivalence
classes; the problem is to obtain explicit bounds and
effective enumeration procedures directly in terms of the
weight matrices.
The multiplicity bounds established in this paper may provide
useful constraints for this purpose.
\end{remark}

\subsection{Iso-weighted polar Fano toric varieties and Mirror Symmetry}\label{ssez:isoweighted}
An interesting phenomenon, essentially discovered by Conrads, is that a Fano toric variety $X$ of rank $r=1$ and its polar partner $X^\circ$ have the same weight matrix (in this case a vector) $Q=Q^\circ$ \cite[Lemma~5.3]{Conrads}. This fact does not generalize with the rank: in general, for $r\ge 2$, weight matrices $Q$ and $Q^\circ$ are different and, moreover the rank $r^\circ$ of the polar partner $X^\circ$ is not equal to the rank $r$ of $X$. But there are isolated examples of such \emph{iso-weighted} polar partners also when $r\ge 2$. For example, considering the classification of 2-dimensional Fano toric varieties given in Corollary~\ref{cor:dim2}, one obtains the following pairs of polar partners iso-weighted Fano toric surfaces of rank $r\ge2$:
 \begin{itemize}
   \item the pair $(4a,8a)$ with $r=2$ and weight matrix $Q=\left(               \begin{array}{cccc}
                   1 & 1 & 0 & 0 \\
                   0 & 0 & 1 & 1 \\
                 \end{array}
               \right)=Q^\circ$,
   \item the self dual surface $6c$ giving the pair $(6c,6c)$ with $r=2$ and weight matrix $Q=\left(
                 \begin{array}{cccc}
                   1 & 1 & 0 & 0 \\
                   0 & 1 & 1 & 2 \\
                 \end{array}
               \right)=Q^\circ$,
   \item the self dual surface $6b$ giving the pair $(6b,6b)$ with $r=3$ and weight matrix $Q=\left(
                 \begin{array}{ccccc}
                   1 & 1 & 0 & 0 & 0 \\
                   0 & 1 & 1& 1 & 0 \\
                   0 & 0 & 1 & 2 & 1
                 \end{array}
               \right)=Q^\circ$ ,
   \item the self dual surface $6a$ giving the pair $(6a,6a)$ with $r=4$ and weight matrix $Q=\left(
                           \begin{array}{cccccc}
                             1 & 1 & 0 & 0 & 0 & 0 \\
                             0 & 1 & 1 & 1 & 0 & 0 \\
                             0 & 0 & 1 & 0 & 1 &  0\\
                             0 & 0 & 0 & 1 & 0 & 1 \\
                           \end{array}
                         \right)=Q^\circ$
 \end{itemize}
Clearly, self dual Fano toric varieties are iso-weighted, but the converse is not true, that is, there are iso-weighted Fano toric varieties which are not self dual and give rise to non-tautological pairs.

This fact holds in higher dimension $n\ge 3$, too. Are there conditions characterizing iso-weighted Fano toric  varieties? It would be interesting to find effective criteria for this property.

\subsubsection{Mirror symmetry for anticanonical families of
iso-weighted polar Fano toric varieties}

We work over $\C$. Let $X$ and $X^\circ$ be iso-weighted
polar Fano toric varieties of dimension $n\geq2$, with
$Q=Q^\circ$. By Theorem~\ref{thm:classificazione}, their
universal $1$-coverings are isomorphic. We identify them
with a common Fano toric variety $Y$.
In the notation of that theorem, the equality $Q=Q^\circ$
gives
\[
Z^\circ\cong Y,
\qquad
Z\cong Y^\circ\cong Y/G_Q.
\]
Indeed, $Y$ and $Z^\circ$ have the same integral Gale dual
as a fan matrix, up to the usual equivalences, and the
Fano fan over this matrix is uniquely determined.
Thus the maximal Fano quotient of $Y$ is its polar partner.

With the compatible group identifications of
Theorem~\ref{thm:classificazione}, we obtain the universal
$1$-coverings
\[
\pi\colon Y\longrightarrow X\cong Y/G,
\qquad
\pi^\circ\colon Y\longrightarrow X^\circ\cong Y/G^\circ,
\]
where $G$ and $G^\circ$ are the respective groups of deck
transformations and
\[
G^\circ\cong G_Q/G,
\qquad
|G|\,|G^\circ|=|G_Q|.
\]
Since both quotient maps are \'etale in codimension one,
\[
\pi^*(-K_X)=-K_Y=(\pi^\circ)^*(-K_{X^\circ}).
\]
Using the linearizations induced by these pullbacks,
we therefore obtain
\[
\begin{aligned}
H^0(X,-K_X)
&\xrightarrow[\pi^*]{\sim}H^0(Y,-K_Y)^G,\\
H^0(X^\circ,-K_{X^\circ})
&\xrightarrow[(\pi^\circ)^*]{\sim}
H^0(Y,-K_Y)^{G^\circ}.
\end{aligned}
\]
Consequently, the anticanonical families of $X$ and
$X^\circ$ are identified, via pullback, with two invariant
subfamilies of the single anticanonical family $|-K_Y|$.
Explicitly, let $S\in|-K_X|$ and
$S^\circ\in|-K_{X^\circ}|$ be general members, and set
\[
\widehat S_G:=\pi^{-1}(S),
\qquad
\widehat S_{G^\circ}:=(\pi^\circ)^{-1}(S^\circ).
\]
Then $\widehat S_G$ and $\widehat S_{G^\circ}$ are members
of $|-K_Y|$, invariant under the respective group actions,
and
\[
S\cong\widehat S_G/G,
\qquad
S^\circ\cong\widehat S_{G^\circ}/G^\circ.
\]
Thus the original hypersurfaces are recovered by taking
the quotients of the corresponding invariant hypersurfaces
in $Y$.

The anticanonical families on $X$ and $X^\circ$ are paired
by Batyrev's mirror construction \cite{Batyrev94}.
Their general members admit maximal projective crepant
partial (MPCP) desingularizations, which are smooth for
$n\leq4$ \cite[Thm.~4.2.2 and Cor.~4.2.3]{Batyrev94}.
In particular, $n=4$ yields mirror families of smooth
Calabi--Yau threefolds; in higher dimensions, the MPCP
models may remain singular. The quotient descriptions
above concern the anticanonical hypersurfaces before
taking these desingularizations.

The condition $Q=Q^\circ$ is equivalent to the
isomorphism of the universal $1$-coverings of $X$ and
$X^\circ$. This is precisely the selfduality condition
of \cite[Def.~5.3.6]{Batyrev94}, which compares the
associated maximal reflexive pairs. 
Thus the description by quotients of subfamilies in a
common anticanonical family given in
\cite[Prop.~5.3.7]{Batyrev94} applies to iso-weighted
polar partners.
Our formulation expresses this description through the
common weight matrix $Q$, the weight group $G_Q$,
and the corresponding subgroups acting as deck
transformation groups.

\subsubsection{Mirror symmetry within a single anticanonical family}

We work over $\C$. Let $X$ and $X^\circ$ be iso-weighted
polar Fano toric varieties of dimension $n\geq2$, with
$Q=Q^\circ$. We show how their mirror correspondence
can be expressed as an internal duality within the
anticanonical family of their common universal
$1$-covering, through invariant subfamilies and the
corresponding quotient actions.

By Theorem~\ref{thm:classificazione}, the universal
$1$-coverings of $X$ and $X^\circ$ are isomorphic.
Identifying them with a common Fano toric variety $Y$,
we obtain
\[
Z^\circ\cong Y,
\qquad
Z\cong Y^\circ\cong Y/G_Q,
\]
where $Z$ is the maximal Fano quotient of $Y$.
Indeed, $Q=Q^\circ$ identifies the integral Gale duals
defining $Y$ and $Z^\circ$, and the Fano fan over their
ray matrix is uniquely determined.

With the compatible group identifications of
Theorem~\ref{thm:classificazione}, the universal
$1$-coverings can therefore be written as
\[
\pi\colon Y\longrightarrow X\cong Y/G,
\qquad
\pi^\circ\colon Y\longrightarrow X^\circ\cong Y/G^\circ.
\]
Their deck transformation groups correspond under polar
duality and satisfy
\[
G^\circ\cong G_Q/G,
\qquad
|G|\,|G^\circ|=|G_Q|.
\]
Since both quotient maps are \'etale in codimension one,
\[
\pi^*(-K_X)=-K_Y=(\pi^\circ)^*(-K_{X^\circ}).
\]
Using the linearizations induced by these pullbacks,
we obtain
\[
\begin{aligned}
H^0(X,-K_X)
&\xrightarrow[\pi^*]{\sim}H^0(Y,-K_Y)^G,\\
H^0(X^\circ,-K_{X^\circ})
&\xrightarrow[(\pi^\circ)^*]{\sim}
H^0(Y,-K_Y)^{G^\circ}.
\end{aligned}
\]
Thus both anticanonical families are realized, via
pullback, as invariant subfamilies of the same complete
linear system $|-K_Y|$.

Explicitly, for general members
$S\in|-K_X|$ and $S^\circ\in|-K_{X^\circ}|$, set
\[
\widehat S_G:=\pi^{-1}(S),
\qquad
\widehat S_{G^\circ}:=(\pi^\circ)^{-1}(S^\circ).
\]
These are anticanonical hypersurfaces of $Y$, invariant
under the respective group actions, and
\[
S\cong\widehat S_G/G,
\qquad
S^\circ\cong\widehat S_{G^\circ}/G^\circ.
\]
Consequently, for iso-weighted polar partners, the mirror
correspondence is realized as an
\emph{internal duality within a single anticanonical
family}. The subgroup correspondence $G\mapsto G^\circ$
pairs invariant subfamilies of $|-K_Y|$, together with
their specified quotient actions; taking the respective
quotients recovers the two mirror families.
In this sense, the ambient variety $Y$ and its
anticanonical family remain fixed, while the group data
distinguish the two sides of the mirror correspondence.

By Batyrev's construction \cite{Batyrev94}, general
anticanonical hypersurfaces on $X$ and $X^\circ$ admit
maximal projective crepant partial (MPCP)
desingularizations. These are smooth for $n\leq4$
\cite[Thm.~4.2.2 and Cor.~4.2.3]{Batyrev94}.
In particular, $n=4$ gives mirror families of smooth
Calabi--Yau threefolds; in higher dimensions, the MPCP
models may remain singular. The quotient descriptions
above concern the anticanonical hypersurfaces before
taking these desingularizations.

The realization of mirror families as quotients of
subfamilies on a common toric variety is already present
in \cite[Thm.~5.3.1 and Prop.~5.3.7]{Batyrev94}.
Here this internal description is expressed through the
common weight matrix $Q$ and the subgroup correspondence
of its weight group $G_Q$, with the subgroups acting as
deck transformation groups of the ambient universal
$1$-coverings.
It would be interesting to study the geometry and
deformations of these invariant subfamilies in relation
to the subgroup correspondence.

\bibliography{MILEA}
\bibliographystyle{acm}
\end{document}